\documentclass[12pt,reqno]{amsart}
\usepackage[letterpaper,margin=1in]{geometry}
\usepackage[T1]{fontenc}
\usepackage{lmodern,amssymb,mathtools,microtype,booktabs,capt-of,placeins,needspace,etoolbox}
\usepackage{xcolor}
\usepackage{algorithm,algpseudocode}
\usepackage{tikz,chessfss}
\usetikzlibrary{decorations.pathreplacing}
\definecolor{queensUpper}{HTML}{087E8B}
\definecolor{queensLower}{HTML}{C66A2B}
\usepackage{hyperref} 
\definecolor{bluey}{rgb}{0,0,0.6}
\hypersetup{
  pdftitle={Greedy queens and the golden ratio},
  pdfauthor={Boon Suan Ho},
  bookmarksnumbered=true,
  bookmarksopen=true,
  bookmarksdepth=2,
  pdfpagemode=UseOutlines,
  colorlinks,
  linkcolor=bluey,
  citecolor=bluey,
  urlcolor=bluey
}

\makeatletter
\AddToHook{cmd/@makecaption/before}{%
  \ifdefstring{\@captype}{table}{\addvspace{8pt}}{}%
}
\makeatother

\makeatletter
\def\thmhead@plain#1#2#3{%
  \thmname{#1}\thmnumber{\@ifnotempty{#1}{ }\@upn{#2}}%
  \thmnote{ \the\thm@notefont(#3)}}
\let\thmhead\thmhead@plain
\makeatother

\newtheorem{theorem}{Theorem}
\newtheorem{lemma}[theorem]{Lemma}
\newtheorem{proposition}[theorem]{Proposition}
\newtheorem{corollary}[theorem]{Corollary}
\theoremstyle{definition}
\newtheorem{definition}[theorem]{Definition}
\AddToHook{env/definition/begin}{%
  \postdisplaypenalty=10000
  \pushQED{\qed}}
\AddToHook{env/definition/end}{\popQED}
\newtheorem{condition}[theorem]{Condition}
\theoremstyle{remark}
\newtheorem*{remark}{Remark}
\AddToHook{env/remark/begin}{%
  \postdisplaypenalty=10000
  \pushQED{\qed}}
\AddToHook{env/remark/end}{\popQED}
\theoremstyle{plain}
\newcommand{\N}{\mathbb N}
\newcommand{\dts}{\mathinner{\ldotp\ldotp}}
\newcommand{\J}{\mathcal J}
\newcommand{\last}{\operatorname{last}_{12}}

\newcommand{\xnote}[1]{}
\makeatletter
\newcommand{\panelcaption}[3][0pt]{{\setlength{\captionindent}{0pt}%
  \def\@captionfont{\normalfont\footnotesize}\raggedright
  \captionof{figure}{#2\kern2\dimexpr#1\relax}\label{#3}}}
\makeatother
\allowdisplaybreaks[1]
\title[Greedy queens and the golden ratio]{Greedy queens and the golden ratio}
\author{Boon Suan Ho}
\address{Department of Mathematics, National University of Singapore}
\email{hbs@u.nus.edu}
\date{}
\begin{document}

\begin{abstract}
Place a queen in each successive column of an infinite $\N\times\N$ chessboard,
always choosing the lowest row such that no two queens may attack one another. 
We prove that the row~$q_n$ occupied by the queen in the $n$th column satisfies 
$q_n=n\phi+O(1)$ or $q_n=n/\phi+O(1)$, where $\phi=(1+\sqrt5)/2$ is the golden ratio.
\end{abstract}
\maketitle
\thispagestyle{empty}

%
\section{Introduction}
Let $\N=\{0,1,2,\ldots\}$. Set $q_0=0$ and, for
$n\ge1$, let $q_n$ be the least $y\in\N$ such that
\begin{equation}\label{eq:greedy}
 y\ne q_i,\qquad y-n\ne q_i-i,\qquad y+n\ne q_i+i
 \qquad(0\le i<n);
\end{equation}
that is, $(n,y)$ does not share a row, diagonal, or antidiagonal
with any previous queen.
Each step excludes only finitely many rows, so $q_n$ is defined
for every column $n$, and these queens are pairwise nonattacking. The sequence begins
\[
 (q_0,q_1,q_2,\dots)=(0,2,4,1,3,8,10,12,14,5,7,18,6,21,9,\dots);
\]
it is the \emph{greedy queens permutation of $\N$} (we prove that it is a permutation in Lemma~\ref{lem:surjective}).
Call a queen \textcolor{queensUpper}{\emph{upper}} if $q_n>n$ and
\textcolor{queensLower}{\emph{lower}} if $q_n<n$.
For a square $(x,y)$, its diagonal is indexed by $y-x$ and its antidiagonal by $y+x$.

\begin{figure}[!ht]
\centering
\definecolor{queensSquare}{HTML}{F2F3F5}
\definecolor{queensGuide}{HTML}{9CA6B8}
\definecolor{queensFocus}{HTML}{E7EBFF}
\definecolor{queensAnti}{HTML}{5865A8}
\scalebox{0.90}{%
\begin{minipage}{0.84\linewidth}
\begin{minipage}[t]{0.50\linewidth}
\vspace{0pt}
\centering
\setboardfontsize{9.5pt}
\begin{tikzpicture}[x=0.3cm,y=0.3cm,
  coordinate label/.style={font=\sffamily\scriptsize,text=black!60,inner sep=1pt},
  queen/.style={inner sep=0pt,text=black}]
  \foreach \x in {0,...,20} {
    \foreach \y in {0,...,30} {
      \pgfmathtruncatemacro{\parity}{mod(\x+\y,2)}
      \ifnum\parity=0
        \fill[queensSquare] (\x-0.5,\y-0.5) rectangle (\x+0.5,\y+0.5);
      \fi
    }
  }
  \fill[queensFocus] (19.5,-0.5) rectangle (20.5,30.5);
  \draw[queensGuide,dash pattern=on 2.2pt off 2.2pt,line width=0.45pt]
    (-0.5,-0.5) -- (20.5,20.5);
  \node[queen] at (0,0) {\BlackQueenOnWhite};
  \foreach \x/\y in {1/2,2/4,5/8,6/10,7/12,8/14,11/18,13/21,
                      15/24,16/26,17/28,18/30} {
    \node[queen,text=queensUpper] at (\x,\y) {\BlackQueenOnWhite};
  }
  \foreach \x/\y in {3/1,4/3,9/5,10/7,12/6,14/9,19/11} {
    \node[queen,text=queensLower] at (\x,\y) {\BlackQueenOnWhite};
  }
  \foreach \y in {0,...,12,14,15,...,30} {
    \draw[queensGuide,line width=0.35pt]
      (19.88,\y-0.12) -- (20.12,\y+0.12)
      (19.88,\y+0.12) -- (20.12,\y-0.12);
  }
  \node[queen,text=queensLower] at (20,13) {\BlackQueenOnWhite};
  \foreach \i in {0,2,...,20} {
    \node[coordinate label,anchor=north] at (\i,-0.65) {\i};
  }
  \foreach \i in {0,5,...,30} {
    \node[coordinate label,anchor=east] at (-0.65,\i) {\i};
  }
\end{tikzpicture}
\end{minipage}\hfill
\begin{minipage}[t]{0.40\linewidth}
\vspace{0pt}
\centering
\begin{tikzpicture}[x=0.6cm,y=0.6cm,
  index/.style={font=\small,text=black!40,inner sep=0pt},
  coordinate label/.style={font=\sffamily\scriptsize,text=black!60,inner sep=1pt}]
  \node[font=\small\bfseries] at (2,5.15) {Diagonals: $y-x$};
  \foreach \x in {0,...,4} {
    \foreach \y in {0,...,4} {
      \pgfmathtruncatemacro{\parity}{mod(\x+\y,2)}
      \ifnum\parity=0
        \fill[queensSquare] (\x-0.5,\y-0.5) rectangle (\x+0.5,\y+0.5);
      \fi
      \pgfmathtruncatemacro{\diagonalindex}{\y-\x}
      \ifnum\diagonalindex=2
        \fill[queensUpper!15] (\x-0.5,\y-0.5) rectangle (\x+0.5,\y+0.5);
        \node[index,text=queensUpper,font=\small\bfseries] at (\x,\y) {$2$};
      \else\ifnum\diagonalindex=-2
        \fill[queensLower!17] (\x-0.5,\y-0.5) rectangle (\x+0.5,\y+0.5);
        \node[index,text=queensLower,font=\small\bfseries] at (\x,\y) {$-2$};
      \else
        \node[index] at (\x,\y) {$\diagonalindex$};
      \fi\fi
    }
  }
  \foreach \i in {0,...,4} {
    \node[coordinate label,anchor=north] at (\i,-0.65) {\i};
    \node[coordinate label,anchor=east] at (-0.65,\i) {\i};
  }
  \draw[black,line width=0.5pt]
    (-0.5,4.5) -- (-0.5,-0.5) -- (4.5,-0.5);
  \node[coordinate label,text=black,anchor=west,inner sep=3pt]
    at (4.5,-0.5) {$x$};
  \node[coordinate label,text=black] at (-0.9,4.6) {$y$};
  \node[font=\footnotesize] at (2,-1.55)
    {\textcolor{queensUpper}{2nd upper diagonal}:\quad $y-x=2$};
  \node[font=\footnotesize] at (2,-2.15)
    {\textcolor{queensLower}{2nd lower diagonal}:\quad $y-x=-2$};
\end{tikzpicture}

\vspace{0.35cm}

\begin{tikzpicture}[x=0.6cm,y=0.6cm,
  index/.style={font=\small,text=black!40,inner sep=0pt},
  coordinate label/.style={font=\sffamily\scriptsize,text=black!60,inner sep=1pt}]
  \node[font=\small\bfseries] at (2,5.15) {Antidiagonals: $y+x$};
  \foreach \x in {0,...,4} {
    \foreach \y in {0,...,4} {
      \pgfmathtruncatemacro{\parity}{mod(\x+\y,2)}
      \ifnum\parity=0
        \fill[queensSquare] (\x-0.5,\y-0.5) rectangle (\x+0.5,\y+0.5);
      \fi
      \pgfmathtruncatemacro{\antidiagonalindex}{\y+\x}
      \ifnum\antidiagonalindex=3
        \fill[queensAnti!16] (\x-0.5,\y-0.5) rectangle (\x+0.5,\y+0.5);
        \node[index,text=queensAnti,font=\small\bfseries] at (\x,\y) {$3$};
      \else
        \node[index] at (\x,\y) {$\antidiagonalindex$};
      \fi
    }
  }
  \foreach \i in {0,...,4} {
    \node[coordinate label,anchor=north] at (\i,-0.65) {\i};
    \node[coordinate label,anchor=east] at (-0.65,\i) {\i};
  }
  \draw[black,line width=0.5pt]
    (-0.5,4.5) -- (-0.5,-0.5) -- (4.5,-0.5);
  \node[coordinate label,text=black,anchor=west,inner sep=3pt]
    at (4.5,-0.5) {$x$};
  \node[coordinate label,text=black] at (-0.9,4.6) {$y$};
  \node[font=\footnotesize] at (2,-1.55)
    {\textcolor{queensAnti}{3rd antidiagonal}:\quad $y+x=3$};
\end{tikzpicture}
\end{minipage}
\end{minipage}%
}
\caption{The greedy step in column $20$ (left), diagonal labels $y-x$
(upper right), and antidiagonal labels $y+x$ (lower right).}
\label{fig:infinite-board}
\end{figure}

Antti Karttunen introduced the positive-integer version of this sequence
in 2001 as OEIS \href{https://oeis.org/A065188}{A065188}; our $0$-indexing gives
\href{https://oeis.org/A275895}{A275895} \cite{OEIS}.
Dekking, Shallit, and Sloane studied a construction that scans the
board along successive antidiagonals rather than columns; they showed that it produces the same set of
queen positions as our column-by-column rule \cite[pp.~24--25]{DSS}.
They conjectured that upper and lower queens stay within bounded
vertical distances of the lines $y=x\phi$ and $y=x/\phi$ respectively
\cite[Conjecture~25]{DSS}, where
\[
 \phi=\frac{1+\sqrt5}{2}\approx1.618034,\qquad \phi^{-1}=\phi-1\approx0.618034.
\]

\begin{theorem}\label{thm:main}
For every $n\ge1$,
\[
 \left\{
 \begin{aligned}
 \left|q_n-n\phi\right|&<\frac{5}{\phi}\approx3.09017&&\text{if }q_n>n,\\[4pt]
 \Bigl|q_n-\frac{n}{\phi}\Bigr|&<4+\frac{5}{\phi}\approx7.09017&&\text{if }q_n<n.
 \end{aligned}
 \right.
\]
\end{theorem}

Knuth \cite[answer 7.2.2.1--38]{Knuth4B} observed computationally that
the queens stay within narrower ranges, and we prove these too. In his
$1$-indexed coordinates, the queen in column $c=n+1$ lies on row
$s(c)=q_n+1$.

\begin{theorem}\label{thm:knuth}
For every $n\ge1$,
\[
 \left\{
 \begin{aligned}
 \frac{19\sqrt5-49}{8}&<q_n-n\phi<\sqrt5-1&&\text{if }q_n>n,\\[4pt]
 15-8\sqrt5&<q_n-\frac{n}{\phi}<13-4\sqrt5&&\text{if }q_n<n;
 \end{aligned}
 \right.
\]
numerically, rounding outward, $-0.8144<q_n-n\phi<1.2361$ and
$-2.8886<q_n-n/\phi<4.0558$. Consequently, for every $c\ge1$,
\[
 s(c)\in
 \Bigl[\frac{c}{\phi}-3\dts\frac{c}{\phi}+5\Bigr]
 \cup[c\phi-2\dts c\phi+1],
\]
the ranges that Knuth observed for $c\le10^9$.
\end{theorem}

Theorem~\ref{thm:main} follows from a single bound by a short counting
argument (Section~\ref{sec:contraction}). Theorem~\ref{thm:knuth}
sharpens its constants via a more involved technical argument
(Section~\ref{sec:sharper}).

\begin{figure}[!ht]
\centering
\begin{tikzpicture}[x=0.043cm,y=0.043cm,
  tick label/.style={font=\small,text=black!65,inner sep=2pt},
  line label/.style={font=\small,anchor=west,inner sep=4pt}]
  \draw[black,line width=0.5pt] (0,170) -- (0,0) -- (105,0);
  \foreach \x in {0,20,40,60,80,100} {
    \draw[black,line width=0.4pt] (\x,0) -- (\x,-2);
    \node[tick label,anchor=north] at (\x,-3) {\x};
  }
  \foreach \y in {0,50,100,150} {
    \draw[black,line width=0.4pt] (0,\y) -- (-2,\y);
    \node[tick label,anchor=east] at (-3,\y) {\y};
  }
  \node[font=\small,anchor=west,inner sep=3pt] at (105,0) {$x$};
  \node[font=\small,anchor=south,inner sep=3pt] at (0,170) {$y$};
  \draw[queensUpper!45,line width=0.5pt]
    (0,0) -- (100,{100*(1+sqrt(5))/2});
  \draw[queensLower!45,line width=0.5pt]
    (0,0) -- (100,{200/(1+sqrt(5))});
  \node[line label,text=queensUpper] at (102,{100*(1+sqrt(5))/2})
    {$y=x\phi$};
  \node[line label,text=queensLower] at (102,{200/(1+sqrt(5))})
    {$y=\dfrac{x}{\phi}$};
  \foreach \n/\qn in {
    1/2,2/4,3/1,4/3,5/8,6/10,7/12,8/14,9/5,10/7,
    11/18,12/6,13/21,14/9,15/24,16/26,17/28,18/30,19/11,20/13,
    21/34,22/36,23/38,24/40,25/15,26/17,27/44,28/16,29/47,30/19,
    31/50,32/52,33/20,34/55,35/57,36/59,37/22,38/62,39/23,40/65,
    41/27,42/25,43/69,44/71,45/73,46/75,47/77,48/29,49/31,50/81,
    51/83,52/85,53/32,54/88,55/33,56/91,57/37,58/35,59/95,60/97,
    61/99,62/101,63/39,64/104,65/106,66/41,67/109,68/42,69/112,70/43,
    71/115,72/117,73/119,74/45,75/122,76/46,77/49,78/126,79/48,80/129,
    81/131,82/133,83/135,84/51,85/53,86/139,87/141,88/143,89/54,90/56,
    91/147,92/149,93/151,94/58,95/154,96/156,97/158,98/60,99/161,100/61} {
    \ifnum\qn>\n
      \filldraw[fill=queensUpper,draw=white,line width=0.3pt]
        (\n,\qn) circle[radius=1.1pt];
    \else
      \filldraw[fill=queensLower,draw=white,line width=0.3pt]
        (\n,\qn) circle[radius=1.1pt];
    \fi
  }
  \fill[black] (0,0) circle[radius=1.1pt];
\end{tikzpicture}
\caption{Queen positions $(n,q_n)$ for $0\le n\le100$ on equally scaled
axes, with the reference lines $y=x\phi$ and $y=x/\phi$.}
\label{fig:golden-lines}
\end{figure}

Larsson and W\"astlund \cite{LW} obtained analogous golden-ratio
bounds for Maharaja Nim, where the queen can also make knight moves. Dekking, Shallit, and Sloane
\cite[p.~26]{DSS} asked whether their method could be adapted to the
single-quadrant greedy-queens problem. Our proof shares the strategy of
\cite{LW}: It reduces Theorem~\ref{thm:main} to a bound on lower
diagonals, namely $|d_j-j|\le4$ for all $j\ge1$
(Lemma~\ref{lem:estimates}), where the $j$th lower queen lies on the
$d_j$th lower diagonal. We prove this bound with a finite model of
board states, and the golden-ratio bounds then follow from counting
identities.
We use different methods to establish boundedness, however, since the single-quadrant greedy-queens problem is asymmetric, unlike Maharaja Nim.

Carter \cite[Section~5]{Carter} gives a related local description of the
single-quadrant problem and proposes a two-string rewriting approach,
whose completeness is left unproved. Our construction uses the same
separation of upper and lower attacks, but encodes upper-queen information
by a \emph{history graph}. We establish the required bounds by
finite-state verification and an induction connecting the verified
transitions to the actual greedy process.

\begin{remark}[Game interpretation]
Consider a single moving queen on an otherwise empty $\N^2$ board.
Two players alternate moves of the form
\[
 (x,y)\longmapsto(x-t,y),\quad(x,y-t),\quad
 (x-t,y-t),\quad(x-t,y+t),
\]
where $t\ge1$ is an integer and the destination has nonnegative
coordinates. The player unable to move loses. Every move decreases
$2x+y$, so play terminates. The losing positions are exactly $(x,q_x)$:
No move joins two chosen squares, while every other square can reach
one. Indeed, if $y>q_x$, move down to $(x,q_x)$; if $y<q_x$, the greedy
rule supplies an attacking queen in an earlier column.
Thus the chosen squares are the $\mathcal P$-positions, or positions
of Sprague--Grundy value zero \cite[Sections~7--8]{DSS}.
Writing $G$ for the Sprague--Grundy function, $G(x,y)=0$ exactly when $y=q_x$.
Theorem~\ref{thm:main} bounds the location of this zero set, without
determining the positive values.
Deleting the antidiagonal move gives Wythoff Nim, whose losing positions
have an exact golden-ratio description \cite[Section~1]{LW}.
\end{remark}

\medskip
\noindent\textbf{Paper organization.} 
Section~\ref{sec:rows} proves some useful lemmas about greedy queen positions. 
Section~\ref{sec:contraction} uses these lemmas to turn the bound on lower 
diagonals into a recurrence for the number of upper queens up to the $n$th column. 
A contraction argument then yields the stated golden-ratio estimates of Theorem~\ref{thm:main}.
Section~\ref{sec:local} develops an encoding for local states of the board.
Section~\ref{sec:sufficiency} proves that these local states suffice for determining the position of greedy queens on the board.
Section~\ref{sec:finite} verifies closure and the diagonal discrepancy bound for
a finite collection of states by computer, proves by induction
that the actual greedy process always stays within this collection, and
proves Theorem~\ref{thm:knuth}.
Section~\ref{sec:generation} uses our results to give a sequential
algorithm for generating $q_0$, \dots, $q_n$ in $O(n)$ time using
$O(\log n)$ bits of memory.
Appendix~\ref{app:artifacts} discusses how to reproduce our results with a computer.

\medskip
\noindent\textbf{Notation.} For real $x$, $y$, write $[x\dts y]=\{k\in\mathbb Z: x\le k\le y\}$. For a proposition $P$, the \emph{Iverson bracket}
$[P]$ equals $1$ if $P$ is true and $0$ otherwise.
%
\newpage
\section{Upper queens and sorted lower rows}\label{sec:rows}
Write $(x_k^U,y_k^U)$ and $(x_k^L,y_k^L)$ for the coordinates of the $k$th
upper and lower queens respectively, whenever they exist. Each family is
ordered by increasing column, starting at $k=1$.
For $k\ge1$, we call $\{(x,y)\in\N^2:y-x=k\}$ and
$\{(x,y)\in\N^2:y-x=-k\}$ respectively the \emph{$k$th upper}
and \emph{$k$th lower diagonal}; $\{(x,y)\in\N^2:x+y=k\}$
is the \emph{$k$th antidiagonal} (Figure~\ref{fig:infinite-board}).
A row or column is \emph{upper} if it contains an upper queen, and \emph{lower} if it contains a lower queen.

The next identity is implicit in Knuth's program \texttt{infty-queens}
\cite{KnuthQueens}. 
\begin{lemma}\label{lem:upper}
The $k$th upper queen lies on the $k$th upper diagonal; that is,
$y_k^U=x_k^U+k$.
\end{lemma}
\noindent
\begin{minipage}[t]{0.38\linewidth}
\vspace{-0.5\baselineskip}
\centering
\setboardfontsize{11pt}
\begin{tikzpicture}[x=0.46cm,y=0.46cm,font=\scriptsize,line cap=round,
  earlier queen/.style={inner sep=0pt,fill=white,text=queensUpper},
  free line/.style={dash pattern=on 3pt off 2pt,line width=0.5pt}]
  \draw[black!35,line width=0.5pt] (0,0) -- (7,7);
  \draw[black!55,line width=0.5pt] (7,-0.1) -- (7,12);
  \node[anchor=north] at (7,-0.3) {$n$};
  \foreach \j in {1,...,4} {
    \draw[queensUpper!65,line width=0.6pt] (0,\j) -- (7,{7+\j});
    \node[text=black!55,inner sep=0pt] at (7,{7+\j}) {$\times$};
  }
  \draw[free line,queensUpper] (0,5) -- (9,14);
  \node[rotate=45,anchor=south,text=queensUpper,inner sep=2pt]
    at (2.8,7.8) {upper diagonal $k$};
  \draw[free line,black!40] (0,12) -- (10.8,12);
  \node[anchor=south west,inner sep=2pt,yshift=-0.8pt] at (7.5,12.05) {row $n+k$};
  \draw[free line,black!45] (4,15) -- (11.8,7.2);
  \node[anchor=west,inner sep=2pt,yshift=0.8pt] at (7.3,8) {$n+1$};
  \node[inner sep=0pt] at (8.6,9.5) {$\vdots$};
  \node[anchor=west,inner sep=2pt,yshift=0.8pt] at (7.3,11) {$n+k-1$};
  \node[align=center,anchor=north,inner sep=1pt] at (10.5,6.9)
    {antidiagonal\\$2n+k$};
  \foreach \x/\y in {1/2,2/4,5/8,6/10} {
    \node[earlier queen] at (\x,\y) {\BlackQueenOnWhite};
  }
  \foreach \x/\y in {3/1,4/3} {
    \node[earlier queen,text=queensLower] at (\x,\y) {\BlackQueenOnWhite};
  }
  \node[earlier queen,text=black] at (0,0) {\BlackQueenOnWhite};
  \filldraw[fill=white,draw=queensUpper,line width=0.8pt]
    (6.8,11.8) rectangle (7.2,12.2);
  \node[font=\footnotesize,anchor=south,yshift=-4.5pt] at (6,15.3) {$(n,n+k)$};
  \draw[->,>=stealth,black!55,line width=0.4pt]
    (6.3,15) .. controls (6.6,14.3) and (6.6,12.9) .. (6.95,12.25);
\end{tikzpicture}
\panelcaption{The induction step of Lemma~\ref{lem:upper}
(schematic).}{fig:upper-step}
\end{minipage}\hfill
\begin{minipage}[t]{0.60\linewidth}
\vspace{0pt}
\begin{proof}
Suppose inductively that the first $k-1$ upper queens have used
the upper diagonals $1$, \dots, $k-1$ respectively, and consider column $n=x_k^U$. Those upper diagonals prevent rows $n+1$, \dots, $n+k-1$ from being used,
so the lowest upper square that is possibly free is
$(n,n+k)$ (Figure~\ref{fig:upper-step}). By induction, the $k$th upper diagonal that it lies on is free,
so it remains
to show that its row and antidiagonal are free as well. 
Every earlier upper queen $(x_j^U,y_j^U)$ has $j<k$ and
$x_j^U<n$, so the inductive hypothesis gives
$y_j^U=x_j^U+j<n+k$. Every earlier lower queen, and the queen at the
origin, also lies below row $n+k$. Thus every earlier queen has both
coordinates strictly less than those of $(n,n+k)$, so none shares
its row or antidiagonal. Since the queen in column $n$ is upper and
all smaller upper rows are attacked, the greedy rule chooses row
$n+k$. This completes the induction.
\end{proof}
\end{minipage}
\par\medskip
Note that Lemma~\ref{lem:upper} implies that successive upper rows differ by at least two.

The permutation property below follows from the row-permutation
argument of Rob Pratt, Bob Selcoe, and N.~J.~A.~Sloane (July~2, 2016)
in OEIS entry~\href{https://oeis.org/A269526}{A269526} \cite[Theorem~R]{OEIS}, reproduced in
\cite[Theorem~23]{DSS}; see also Knuth
\cite[exercise~7.2.2.1--37]{Knuth4B}. 
In combinatorial game theory terms, that theorem says that every
nonnegative Sprague--Grundy value occurs exactly once in each row and
column. The lemma below is its zero-value case. 
\begin{lemma}\label{lem:surjective}
The map $n\mapsto q_n$ is a permutation of $\N$.
\end{lemma}
\begin{proof}
The rows $q_n$ are distinct by construction, so it suffices to show that
every row is occupied by some queen. Suppose otherwise, and let $m$ be the least unoccupied row. 
For $n>m$, an earlier queen $(i,q_i)$ can attack $(n,m)$ diagonally
only if
\[
 n-m=i-q_i,\qquad q_i=m+i-n<m.
\]
Since there are at most $m$ such earlier queens, it follows that $(n,m)$ has no diagonal attackers for sufficiently large $n$. 
Now after some finite time, every row below $m$ has been occupied.
Since row $m$ remains unused by hypothesis,
the greedy rule requires $n+m$ to be an occupied antidiagonal for sufficiently large $n$. Hence the
set of occupied antidiagonals contains a tail $\{N,N+1,\dots\}$ of the integers.

We now show that such a tail is impossible. Choose an integer $t>2N+1$
and consider the $t$ queens occupying the antidiagonals
$N$, $N+1$, \dots, $N+t-1$ (Figure~\ref{fig:antidiagonals}). These queens lie in distinct columns and distinct
rows. If we arrange their column coordinates in increasing order, the
first is at least $0$, the second at least $1$, and so on.

\par\medskip\noindent
\begin{minipage}[t]{0.60\linewidth}
\vspace{0pt}
Thus the sum of their column coordinates is at least
\[
 0+1+\cdots+(t-1)=\frac{t(t-1)}2.
\]
The same bound holds for their row coordinates. Thus the sum of all
their row and column coordinates is at least $t(t-1)$.

On the other hand, each queen's row coordinate plus its column coordinate
is precisely the label of its antidiagonal. Summing these labels gives
the same total as
\[
 N+(N+1)+\cdots+(N+t-1)=tN+\frac{t(t-1)}2.
\]
\end{minipage}\hfill
\begin{minipage}[t]{0.38\linewidth}
\vspace{0pt}
\centering
\setboardfontsize{12pt}
\begin{tikzpicture}[x=0.43cm,y=0.43cm,font=\scriptsize,line cap=round,
  coordinate guide/.style={black!30,dash pattern=on 2pt off 2pt,line width=0.4pt},
  count queen/.style={inner sep=0pt,text=black}]
  \fill[queensUpper!7] (0,6) -- (0,9) -- (9,0) -- (6,0) -- cycle;
  \foreach \s in {7,8} {
    \draw[queensUpper!35,line width=0.45pt] (0,\s) -- (\s,0);
  }
  \foreach \s in {6,9} {
    \draw[queensUpper!75,line width=0.65pt] (0,\s) -- (\s,0);
  }
  \foreach \x/\y in {1/5,3/6,5/2,8/0} {
    \draw[coordinate guide] (\x,0) -- (\x,\y) -- (0,\y);
  }
  \draw[black!65,line width=0.5pt] (0,9.6) -- (0,0) -- (9.6,0);
  \node[anchor=south] at (0,9.65) {$y$};
  \node[anchor=west] at (9.65,0) {$x$};
  \foreach \x/\bound in {1/{\ge0},3/{\ge1},5/{\cdots},8/{\ge t-1}} {
    \draw[black!65,line width=0.4pt] (\x,0) -- (\x,-0.15);
    \node[anchor=north,inner sep=2pt] at (\x,-0.22) {$\bound$};
  }
  \foreach \y/\bound in {0/{\ge0},2/{\ge1},5/{\vdots},6/{\ge t-1}} {
    \draw[black!65,line width=0.4pt] (0,\y) -- (-0.15,\y);
    \node[anchor=east,inner sep=2pt] at (-0.22,\y) {$\bound$};
  }
  \node[text=queensUpper,rotate=-45,anchor=north,inner sep=2pt]
    at (2.1,3.9) {$N$};
  \node[text=queensUpper,rotate=-45,anchor=south,inner sep=2pt]
    at (5.3,3.7) {$N+t-1$};
  \foreach \x/\y in {1/5,3/6,5/2,8/0} {
    \node[count queen] at (\x,\y) {\BlackQueenOnWhite};
  }
\end{tikzpicture}
\panelcaption{One queen per antidiagonal (schematic), with lower
bounds on the sorted coordinates.}{fig:antidiagonals}
\end{minipage}
\par\medskip
\noindent Consequently,
\[
 t(t-1)\le tN+\frac{t(t-1)}2,
 \qquad\text{so}\qquad t\le2N+1.
\]
This contradicts our choice of $t$. Hence every row is occupied, as required.
\end{proof}

For $n\in\N$, define
\begin{equation}\label{eq:counts}
 U(n)=\#\{1\le i\le n:q_i>i\},\qquad L(n)=n-U(n).
\end{equation}
These count upper and lower queens through column $n$. There are infinitely many lower queens: Successive upper rows differ by at least two by Lemma~\ref{lem:upper}, omitting infinitely many positive integers, and every row is occupied by Lemma~\ref{lem:surjective}.
Their columns satisfy $x_1^L<x_2^L<\cdots$, though the row values $y_j^L=q_{x_j^L}$ need not be increasing. We define the \emph{lower-diagonal magnitudes}
\begin{equation}\label{eq:rank-definition}
 d_j\coloneqq x_j^L-y_j^L>0\qquad(j\ge1),
\end{equation}
so that the $j$th lower queen lies on the $d_j$th lower diagonal.
The $d_j$ are distinct, and they are not increasing. Corollary~\ref{cor:diagonal-coverage}
will show that they form a permutation of the positive integers.

Let $v_1<v_2<\cdots$ be the lower-row values $y_j^L$, sorted into increasing order. We call $v_j$ the $j$th \emph{sorted} lower row
and $y_j^L$ the $j$th \emph{chronological} lower row (or just the $j$th lower row).

\begin{center}
\begin{minipage}{\linewidth}
\centering
\renewcommand{\arraystretch}{1.15}
\setlength{\tabcolsep}{6pt}
\begin{tabular}{@{}l*{15}{r}@{}}
\toprule
$j$     & 1 & 2 & 3 & 4 & 5 & 6 & 7 & 8 & 9 & 10 & 11 & 12 & 13 & 14 & 15 \\
\midrule
$x_j^L$ & 3 & 4 & 9 & 10 & 12 & 14 & 19 & 20 & 25 & 26 & 28 & 30 & 33 & 37 & 39 \\
$y_j^L$ & 1 & 3 & 5 & 7 & 6 & 9 & 11 & 13 & 15 & 17 & 16 & 19 & 20 & 22 & 23 \\
$d_j$   & 2 & 1 & 4 & 3 & 6 & 5 & 8 & 7 & 10 & 9 & 12 & 11 & 13 & 15 & 16 \\
$v_j$   & 1 & 3 & 5 & 6 & 7 & 9 & 11 & 13 & 15 & 16 & 17 & 19 & 20 & 22 & 23 \\
$U(j)$  & 1 & 2 & 2 & 2 & 3 & 4 & 5 & 6 & 6 & 6 & 7 & 7 & 8 & 8 & 9 \\
\bottomrule
\end{tabular}
\par
\captionof{table}{Lower-queen sequences and upper counts for $1\le j\le15$.}
\label{tab:lower-values}
\end{minipage}
\end{center}
\vspace{\baselineskip}

\begin{samepage}
The following identity is a special case of the Lambek--Moser theorem on
inverse and complementary sequences \cite{LM}, applied to $f(k)=x_k^U$. 
\begin{lemma}\label{lem:complement}
For every $j\ge1$, the $j$th sorted lower row $v_j$ satisfies
\begin{equation}\label{eq:sorted-rows}
 v_j=j+U(j-1).
\end{equation}
\end{lemma}
\noindent
\begin{minipage}[t]{0.38\linewidth}
\vspace{-0.5\baselineskip}
\centering
\setboardfontsize{11pt}
\begin{tikzpicture}[x=0.34cm,y=0.34cm,font=\scriptsize,line cap=round,
  example queen/.style={inner sep=0pt,fill=white},
  example guide/.style={dash pattern=on 3pt off 2pt,line width=0.5pt}]
  \node[font=\footnotesize,anchor=south] at (5,19.2) {$h=U(4)=2$};
  \foreach \y in {1,3,5,6} {
    \draw[queensLower!22,line width=0.35pt] (0,\y) -- (12.6,\y);
  }
  \foreach \y in {2,4} {
    \draw[queensUpper!30,line width=0.35pt] (0,\y) -- (12.6,\y);
  }
  \draw[black!55,line width=0.5pt] (0,19) -- (0,0) -- (13,0);
  \node[anchor=south] at (0,19.1) {$y$};
  \node[anchor=west] at (13,0) {$x$};
  \foreach \y in {1,3,5,6,7} {
    \node[anchor=east,text=queensLower,inner sep=2pt] at (-0.2,\y) {$\y$};
  }
  \foreach \y in {2,4,8,10,12,14,18} {
    \node[anchor=east,text=queensUpper,inner sep=2pt] at (-0.2,\y) {$\y$};
  }
  \foreach \x in {10,12} {
    \draw[black!55,line width=0.4pt] (\x,0) -- (\x,-0.2);
    \node[anchor=north,inner sep=2pt] at (\x,-0.25) {$\x$};
  }
  \draw[example guide,black!45] (5,0) -- (5,19);
  \node[anchor=north,inner sep=2pt] at (5,-0.25) {$j=5$};
  \draw[example guide,queensLower,line width=0.65pt] (0,7) -- (12.8,7);
  \node[anchor=west,text=queensLower,inner sep=2pt,yshift=0.8pt] at (12.8,7) {$r=7$};
  \foreach \x/\y in {1/2,2/4,5/8,6/10,7/12,8/14,11/18} {
    \node[example queen,text=queensUpper] at (\x,\y) {\BlackQueenOnWhite};
  }
  \foreach \x/\y in {3/1,4/3,9/5,12/6,10/7} {
    \node[example queen,text=queensLower] at (\x,\y) {\BlackQueenOnWhite};
  }
  \node[example queen,text=black] at (0,0) {\BlackQueenOnWhite};
  \draw[black!35,line width=0.5pt] (0,0) -- (12.6,12.6);
  \node[anchor=north,align=center,inner sep=2pt] at (6.5,-2.3)
    {Upper rows through $7$: \textcolor{queensUpper}{$2$, $4$}.\\
     Lower rows through $7$: \textcolor{queensLower}{$1$, $3$, $5$, $6$, $7$}.\\[2pt]
     Thus $v_5=7=5+U(4)$.};
\end{tikzpicture}
\panelcaption[11pt]{Lemma~\ref{lem:complement} for $j=5$.}{fig:sorted-rows}
\end{minipage}\hfill
\begin{minipage}[t]{0.60\linewidth}
\vspace{0pt}
\begin{proof}
Put $h=U(j-1)$ and $r=j+h$. We will show that row $r$ is a lower row
and that exactly $j$ lower rows lie in $[1\dts r]$
(Figure~\ref{fig:sorted-rows} shows the case $j=5$). This identifies $r$
as the $j$th smallest lower row, namely $v_j$.

There are exactly $h$ upper queens in columns to the left of $j$. Each has an
index $k\le h$ and a column $x_k^U\le j-1$, so Lemma~\ref{lem:upper} gives
\[
 y_k^U=x_k^U+k\le(j-1)+h=r-1.
\]
Every upper queen in column $j$ or later instead has $k\ge h+1$ and
$x_k^U\ge j$, so
\[
 y_k^U=x_k^U+k\ge j+(h+1)=r+1.
\]
Thus no upper queen occupies row $r$, and exactly $h$ of the rows in
$[1\dts r]$ are upper rows.

By Lemma~\ref{lem:surjective}, every one of these positive rows is
occupied. The remaining $r-h=j$ rows are therefore lower rows, and
row $r$ itself is one of them. Hence $r$ is the largest of these $j$
lower rows, proving $v_j=r=j+U(j-1)$.
\end{proof}
\end{minipage}
\par\medskip
\end{samepage}

%
\section{From diagonal discrepancy to the golden ratio}\label{sec:contraction}
We now combine the exact formula for the sorted lower rows (Lemma~\ref{lem:complement}) with a bound
on the lower-diagonal magnitudes (Lemma~\ref{lem:estimates}). The argument has three steps. First,
sorting converts the diagonal discrepancy bound $|d_j-j|\le4$ into an estimate for the lower
columns. That estimate then gives a recurrence for the upper count
$U(n)$. Finally, a contraction bounds the difference between $U(n)$ and
$n/\phi$, yielding the two error bounds in Theorem~\ref{thm:main}.

For complementary increasing sequences, Fraenkel and Peled
\cite[Theorem~4.3]{FP} show that $b_n-a_n=\gamma n+O(1)$, with $\gamma>0$,
implies bounded error from linear formulas for both sequences.
Larsson and W\"astlund \cite[Lemma~2.3]{LW} extend this to a nonmonotone
second sequence. Our upper columns and upper rows are not complementary
($2$ occurs in both), so we instead use the sorted-row identity to
obtain a contraction for $U$.

The following bound is the heart of our result. Sections~\ref{sec:local}--\ref{sec:finite} are devoted to proving it.
\begin{lemma}[Bounded diagonal discrepancy]\label{lem:estimates}
For every $j\ge1$, the $j$th lower-diagonal magnitude $d_j=x_j^L-y_j^L$ satisfies
\begin{equation}\label{eq:rank}
 |d_j-j|\le4.
\end{equation}
\end{lemma}
The remainder of this section shows how to deduce
Theorem~\ref{thm:main} from Lemma~\ref{lem:estimates}.

To interpret the diagonal discrepancy bound,
consider the $j$th lower queen. Among the columns $1$, \dots, $x_j^L$,
exactly $j$ contain lower queens, so the other $x_j^L-j$ contain upper
queens. Therefore $U(x_j^L)=x_j^L-j$. Since $d_j=x_j^L-y_j^L$, we obtain
\begin{equation}\label{eq:rank-row}
 y_j^L-U(x_j^L)=j-d_j;
\end{equation}
thus the diagonal bound implies $|y_j^L-U(x_j^L)|\le4$. 

We next show
that replacing the chronological lower row $y_j^L$ by the $j$th sorted
lower row $v_j$ preserves the bound, giving $|v_j-U(x_j^L)|\le4$. This will let us
substitute the exact formula $v_j=j+U(j-1)$ from
Lemma~\ref{lem:complement}.

We use a general nonnegative integer $C$ in place of $4$ to keep track of
how the diagonal discrepancy bound affects the final constants. At the end we will take
$C=4$.
\begin{lemma}\label{lem:sorting}
Let $C\in\N$. If $|d_j-j|\le C$ for every $j\ge1$, then
$|v_j-U(x_j^L)|\le C$ for every $j\ge1$. Equivalently,
\begin{equation}\label{eq:lower-column}
 |x_j^L-2j-U(j-1)|\le C\qquad(j\ge1).
\end{equation}
\end{lemma}
\begin{proof}
Set $t_j=x_j^L-j=U(x_j^L)$. The sequence $t_j$ is nondecreasing since $x_j^L$ is increasing, and the hypothesis together with \eqref{eq:rank-row} gives
\[
 |y_j^L-t_j|\le C\qquad(j\ge1).
\]
Fix $j$. We first compare the $j$th sorted lower row $v_j$ with $t_j$.

For each $i\le j$, monotonicity gives $t_i\le t_j$, and hence
\[
 y_i^L\le t_i+C\le t_j+C.
\]
There are therefore at least $j$ lower-row values at most $t_j+C$.
Their $j$th smallest value must satisfy $v_j\le t_j+C$.

For the other inequality, every $i\ge j$ satisfies
\[
 y_i^L\ge t_i-C\ge t_j-C.
\]
Thus only the first $j-1$ chronological rows can lie below $t_j-C$,
and so the $j$th sorted row
must satisfy $v_j\ge t_j-C$. Together the two inequalities give
$|v_j-t_j|\le C$.

We now use Lemma~\ref{lem:complement}: Substituting $v_j=j+U(j-1)$ and
$t_j=x_j^L-j$ gives
\[
 |x_j^L-2j-U(j-1)|=|t_j-v_j|\le C.\qedhere
\]
\end{proof}

Lemma~\ref{lem:sorting} says that the $j$th lower column is within $C$ of
$2j+U(j-1)$. To turn this into information about an arbitrary column
$n$, we take $j=L(n)$, the number of lower queens placed through column
$n$. When $j\ge1$, the column $n$ lies between two successive lower columns: $x_j^L\le n<x_{j+1}^L$.
Applying the estimate of \eqref{eq:lower-column} at both endpoints will give a recurrence
$U(n)=\frac{1}{2}(n+U(n-U(n))-\eta_n)$ with $|\eta_n|\le C+1$.
\begin{lemma}[Counting recurrence]\label{lem:counting-recursion}
Let $C\in\N$, and suppose
$|x_j^L-2j-U(j-1)|\le C$ for every $j\ge1$. Then, with
$L(n)=n-U(n)$,
\begin{equation}\label{eq:integer-recursion}
 \eta_n\coloneqq n-2U(n)+U(L(n)),\qquad |\eta_n|\le C+1\quad(n\ge0).
\end{equation}
\end{lemma}
\begin{proof}
Fix $n\ge1$, and let
$j=L(n)=n-U(n)$. If $j\ge1$, exactly $j$ lower queens have been placed
through column $n$, so
\[
 x_j^L\le n<x_{j+1}^L.
\]
The assumed estimate, applied to the $j$th and $(j+1)$st lower
columns, bounds these two endpoints and gives
\[
 2j+U(j-1)-C\le n<2j+2+U(j)+C.
\]
The count $U$ increases by at most one between consecutive columns, so
$U(j-1)\ge U(j)-1$. Also, the strict upper bound can be decreased by one
because all the quantities are integers. It follows that
\[
 2j+U(j)-(C+1)\le n\le2j+U(j)+(C+1).
\]
Thus $|2j+U(j)-n|\le C+1$. Substituting $j=n-U(n)=L(n)$,
the expression inside the absolute value becomes $\eta_n$.
This proves the bound when $j\ge1$.

If $j=0$, no lower queen has yet appeared, so $n<x_1^L$. The estimate
\eqref{eq:lower-column} with index $1$, together with $U(0)=0$, gives
$x_1^L\le C+2$. Hence $n\le C+1$. Since $L(n)=j=0$, we have $U(n)=n$, 
so $\eta_n=-n$ and the same bound holds. At $n=0$ we have $\eta_0=0$.
We have therefore established \eqref{eq:integer-recursion} for every
$n\ge0$.
\end{proof}

Lemma~\ref{lem:counting-recursion} converts the lower-column estimate
\eqref{eq:lower-column} into the relation~\eqref{eq:integer-recursion}
between $U(n)$ and $U(n-U(n))$, with error at most $C+1$.
We now use this relation to bound $|U(n)-n/\phi|$.

\begin{proposition}[Contraction and queen positions]\label{prop:stability}
Let $C\in\N$, and suppose $|d_j-j|\le C$ for every $j\ge1$. Then
\begin{equation}\label{eq:Dbound}
 |U(n)-\phi^{-1}n|<\phi^{-1}(C+1)\qquad(n\ge0).
\end{equation}
The upper and lower queen positions satisfy
\[
 \begin{aligned}
 |q_n-n\phi|&<\phi^{-1}(C+1)&&\text{if }q_n>n,\\[3pt]
 \Bigl|q_n-\frac{n}{\phi}\Bigr|&<C+\phi^{-1}(C+1)&&\text{if }q_n<n.
 \end{aligned}
\]
\end{proposition}
\begin{proof}
Lemmas~\ref{lem:sorting} and~\ref{lem:counting-recursion} give
\eqref{eq:integer-recursion}. We now use it to bound the count $U(n)$, 
and then the queen positions $q_n$.

To see why the golden ratio appears, consider a linear model
$U(n)\approx\theta n$. Then $L(n)\approx(1-\theta)n$, and the expression
defining $\eta_n$ would have leading term
\[
 \bigl(1-2\theta+\theta(1-\theta)\bigr)n.
\]
For this expression to stay bounded, its coefficient must vanish:
\[
 1-2\theta+\theta(1-\theta)=0,
 \qquad\text{or equivalently}\qquad \theta^2+\theta=1.
\]
The positive solution is $\theta=\phi^{-1}$. We now bound the error
without assuming such a model.

Put
\[
 \varepsilon(n)=U(n)-\phi^{-1}n.
\]
Since $1-\phi^{-1}=\phi^{-2}$, we also have
\[
 L(n)=n-U(n)=\phi^{-2}n-\varepsilon(n).
\]
Substitute this expression and
$U(L(n))=\phi^{-1}L(n)+\varepsilon(L(n))$ into
\eqref{eq:integer-recursion}. The terms proportional to $n$ cancel,
leaving
\[
 \eta_n=\varepsilon(L(n))-(2+\phi^{-1})\varepsilon(n).
\]
Because $2+\phi^{-1}=\phi^2$, we can solve for $\varepsilon(n)$:
\begin{equation}\label{eq:contraction}
 \varepsilon(n)=\phi^{-2}\varepsilon(L(n))-\phi^{-2}\eta_n.
\end{equation}
The factor $\phi^{-2}$ is less than one. Taking absolute values and using
$|\eta_n|\le C+1$ gives
\[
 |\varepsilon(n)|\le\phi^{-2}|\varepsilon(L(n))|+\phi^{-2}(C+1).
\]
This bounds the error at $n$ by a smaller multiple of the error at
$L(n)$, plus a fixed amount.

For $n\ge1$, the upper queen $q_1=2$ ensures $U(n)\ge1$, so
$0\le L(n)<n$. Repeatedly replacing the argument $n$ by its lower count $L(n)$
therefore reaches $0$ after finitely many steps. If this takes $t$
steps, repeated application of the last inequality, with $\varepsilon(0)=0$,
yields
\[
 |\varepsilon(n)|\le(C+1)\sum_{\ell=1}^{t}\phi^{-2\ell}
 <\frac{\phi^{-2}(C+1)}{1-\phi^{-2}}=\phi^{-1}(C+1).
\]
Since $\varepsilon(0)=0$, the bound also holds at $n=0$.
This proves \eqref{eq:Dbound}.

It remains to translate the count estimate into bounds on the queen
positions. If column~$n$ contains an upper queen, its index among the
upper queens is $U(n)$. Lemma~\ref{lem:upper} therefore gives
$q_n=n+U(n)$, and hence
\[
 q_n-n\phi=U(n)-(\phi-1)n=\varepsilon(n).
\]
The upper error is consequently less than $\phi^{-1}(C+1)$.

If column $n$ contains the $j$th lower queen, then $n=x_j^L$ and
\eqref{eq:rank-row} gives $q_n-U(n)=j-d_j$. The diagonal discrepancy bound therefore
implies $|q_n-U(n)|\le C$. Combining this with the count estimate gives
\[
 \Bigl|q_n-\frac{n}{\phi}\Bigr|
 \le |q_n-U(n)|+|\varepsilon(n)|
 <C+\phi^{-1}(C+1).\qedhere
\]
\end{proof}

\begin{proof}[Proof of Theorem~\ref{thm:main}]
Lemma~\ref{lem:estimates} supplies the hypothesis of
Proposition~\ref{prop:stability} with $C=4$. The proposition gives an
upper error strictly less than $\phi^{-1}(4+1)=5/\phi$ and a lower error
strictly less than $4+\phi^{-1}(4+1)=4+5/\phi$. These are exactly the two
bounds asserted in Theorem~\ref{thm:main}.
\end{proof}

\begin{remark}[$1$-indexed coordinates]
Under the hypotheses of Proposition~\ref{prop:stability}, changing to
$1$-indexed coordinates gives a common error bound $1+C\phi$ for $C\ge1$.
Indeed, the same queen has column $c=n+1$
and row $s(c)=q_n+1$. Thus
\[
 \begin{aligned}
 s(c)-c\phi&=(q_n-n\phi)-\phi^{-1},\\
 s(c)-\frac{c}{\phi}&=\Bigl(q_n-\frac{n}{\phi}\Bigr)+\phi^{-2}.
 \end{aligned}
\]
The $1$-indexed upper error is therefore less than $(C+2)\phi^{-1}$,
and the $1$-indexed lower error is less than
$C+\phi^{-1}(C+1)+\phi^{-2}$. For $C\ge1$, the latter is a common bound,
since
\[
 (C+2)\phi^{-1}\le
 C+\phi^{-1}(C+1)+\phi^{-2}=1+C\phi.
\]
Knuth \cite[answer 7.2.2.1--38]{Knuth4B} observed
computationally that, for $1\le c\le10^9$,
\[
 s(c)\in
 \Bigl[\frac{c}{\phi}-3\dts\frac{c}{\phi}+5\Bigr]
 \cup[c\phi-2\dts c\phi+1].
\]
For comparison, translating Theorem~\ref{thm:main} to $1$-indexed
coordinates and rounding gives, for every $c\ge1$,
\[
 s(c)\in
 \Bigl[\frac{c}{\phi}-6.709\dts\frac{c}{\phi}+7.473\Bigr]
 \cup[c\phi-3.709\dts c\phi+2.473].
\]
Theorem~\ref{thm:knuth}, proved in Section~\ref{sec:sharper}, shows
that Knuth's ranges hold for every $c\ge1$.
\end{remark}

%
\section{Computing one queen from local records}\label{sec:local}

We now begin our work on proving $|d_j-j|\le4$ 
(Lemma~\ref{lem:estimates}). 
While the greedy-queens process is deterministic, describing its state 
by the entire configuration of earlier queens gives an infinite 
collection of growing states.
This section introduces a notion of \emph{local state} (Definition~\ref{def:local-state}) 
that tracks only selected information about earlier queens. 
Doing so will allow us to model the entire infinite process using only 
finitely many states. After introducing local states, the rest of this
section details an algorithm, which we call the \emph{calculation} (Algorithm~\ref{alg:local-successors}), for how one can use the local state before
column~$n$ to place a queen on column~$n$ and update the state for the
next column. However, since the local state only has partial information
about the board, different configurations of earlier queens may yield the
same local state. Thus the calculation will sometimes return several possible
next placements and states.

Section~\ref{sec:sufficiency} proves that, if certain conditions hold, then there always 
exists a branch of the calculation that agrees with the 
greedy algorithm on where to place the next queen and how to update 
the records.

Section~\ref{sec:finite} then computationally verifies that the initial state of 
the board satisfies those conditions, and that exploring every 
possible successor yields a finite collection of states, all 
satisfying those conditions (the state graph, Definition~\ref{def:state-graph}).
Finally, induction using the results of Section~\ref{sec:sufficiency} proves that
the actual greedy process produces states that always stay
within the state graph.
Together with a direct check of the initial placements,
the conditions imply that $|d_j-j|\le4$ for all lower queen
placements, which proves Lemma~\ref{lem:estimates}.

The next two subsections start our work by introducing the eight
records that comprise a local state.

\subsection{Testing lower attacks}\label{sec:local-rank}
Imagine that we are now trying to place a queen in column $n$.
We first test if there is a free lower square for our queen to use.
Writing $m\ge1$ for the least unused row and $d\ge1$ for the least 
unused lower-diagonal magnitude, we see that a lower square can
be chosen only if its row lies in $[m\dts n-d]$.
Indeed, all rows below $m$ are used, and the square $(n, n-i)$ lies on 
the $i$th lower diagonal, which is used for $i<d$.
 
Write
\begin{equation}\label{eq:window}
 w=n-m-d,
\end{equation}
so that the lower squares to test are $(n,m+r)$ for offsets $r=0,\dots,w$.
If $w<0$, there are no lower candidates, so the queen in column $n$
is upper. If $w\ge0$, it is possible for all $w+1$ candidate squares
to be attacked, in which case the queen is upper as well.

To test candidates, the local state tracks attacks from lower and upper
queens separately. Upper attack records will be discussed in Section~\ref{sec:upper-word}.
To track lower attacks, the local state maintains three sets $R$, $D$, $A$ of
nonnegative integers. Consider all earlier lower queens; that is, lower
queens before column $n$. For each such queen $(x,y)$, its row,
lower-diagonal magnitude, and antidiagonal index are $y$, $x-y$,
and $x+y$. We record their offsets from the reference values $m$,
$d$, and $n+m$, keeping only nonnegative values (as the negative values cannot attack our candidate squares):
\begin{align}
 R&=\{y-m:(x,y)\text{ is an earlier lower queen},\ y\ge m\},\notag\\
 D&=\{x-y-d:(x,y)\text{ is an earlier lower queen},\ x-y\ge d\},
       \label{eq:sets}\\
 A&=\{x+y-(n+m):(x,y)\text{ is an earlier lower queen},\ x+y\ge n+m\}.
       \notag
\end{align}
It follows that
\begin{center}
\begin{tabular}{lll}
\toprule
an earlier lower queen attacks $(n,m+r)$&when&record used\\
\midrule
along its row&$r\in R$&row offset\\
along its diagonal&$w-r\in D$&diagonal offset\\
along its antidiagonal&$r\in A$&antidiagonal offset\\
\bottomrule
\end{tabular}
\end{center}
Since row $m$ and magnitude $d$ are unused, we have
$0\notin R$ and $0\notin D$.

The diagonal discrepancy already has a useful expression in these
records. Suppose the next queen is the $j$th lower queen and it uses
offset $r$, so that it has coordinates $(n,m+r)$. 
Since the $j-1$ previous lower queens have used the $d-1$
magnitudes less than $d$, together with the $|D|$ magnitudes greater than it, it follows that
\begin{equation}\label{eq:local-rank}
 j=d+|D|,\qquad d_j=d+w-r,\qquad
 d_j-j=w-r-|D|.
\end{equation}
In particular, $w\le4$ and $D\subseteq[1\dts4]$ would prove the
desired bound, since whenever a lower queen is chosen, 
both $w-r$ and $|D|$ would then belong to $[0\dts4]$. 
We will establish these bounds in Section~\ref{sec:finite}.

\subsection{Testing upper attacks}\label{sec:upper-word}

We now test if a candidate lower queen $(n,m+r)$ is attacked
by any upper queens. Upper queens cannot attack lower queens
along diagonals, so we only need to test if a lower candidate
is attacked by an upper queen along a row or an antidiagonal.

We first introduce an infinite word that tracks upper rows and columns.

\begin{definition}[Queen word]\label{def:queen-word}
For $i\ge1$, let $u_i=[\text{column $i$ has an upper queen}]=[q_i>i]$ and
$b_i=[\text{row $i$ has an upper queen}]$.
The \emph{queen word} is $\sigma=\sigma_1\sigma_2\dots$, where
its \emph{$i$th symbol}
\[
 \sigma_i=2u_i+b_i\in\{0,1,2,3\}
\]
records whether column $i$ and row $i$ contain upper queens.
We call $u_i$ and $b_i$ the \emph{column bit} and the
\emph{row bit} of $\sigma_i$.
\end{definition}

Before column $n$, the prefix $\sigma_1\dots\sigma_{n-1}$
is determined by the queens already placed: For $i<n$,
column $i$ has been filled, and any upper queen in row $i$
must lie in a column less than $i$.

The local state will store subwords of the queen word from
two locations: near index $m$, for testing the lower candidates;
and near index $n$, for checking the new symbol $\sigma_n$
produced at the current step.
The records maintained at these two indices are
\[
 \begin{array}{ll}
 \text{the input history}&H_{\mathrm{in}}=\sigma_{m-12}\dots\sigma_{m-1},\\
 \text{the queue}&Q=\sigma_m\dots\sigma_{m+|Q|-1},\\
 \text{the output history}&H_{\mathrm{out}}=\sigma_{n-12}\dots\sigma_{n-1}.
 \end{array}
\]
The calculation reads the queen word near index $m$ and writes a new
symbol at index $n$; hence the names input and output. The queue $Q$
holds the input symbols read so far, from index $m$ on. The candidate
rows $m+r$ start there, so the attack tests
read most of their bits from $Q$. Unlike the two histories, $Q$ has
no fixed length: It grows when the calculation needs a symbol beyond
its end, and it loses symbols from its front when $m$ advances. It
is never empty. The input history $H_{\mathrm{in}}$ holds the twelve
symbols just before $Q$. Its column bits locate the upper queens
in the columns just before $m$, which the tests still need
(Section~\ref{app:upper}), and together with $Q$ it is the context
from which further symbols are read (Section~\ref{app:graph}). The
output history $H_{\mathrm{out}}$ holds the twelve most recent
symbols, against which the new symbol $\sigma_n$ is checked
(Section~\ref{app:graph}).

The length twelve is a parameter of the method. The histories must
be long enough for the history graph of Section~\ref{app:graph} to
constrain the symbols it supplies, or the verification of
Section~\ref{sec:finite} fails; twelve turns out to be the smallest
length that works (Section~\ref{sec:finite-statement} and
Appendix~\ref{app:construction}).

We introduce one final record. The upper-queen tests of
Section~\ref{app:upper} compare positions of upper queens with
positions involving the current column. By Lemma~\ref{lem:upper},
the upper queen in column $c$ has row $c+U(c)$, which involves the
growing count $U$. The state records the single quantity
\[
 z=n-m-U(m-1),
\]
which Section~\ref{app:upper} shows is enough to make these
comparisons. Together with the records from
Section~\ref{sec:local-rank}, we get the eight fields that comprise
a local state:

\begin{definition}[Local state]\label{def:local-state}
A (local) \emph{state} is an $8$-tuple
\begin{equation}\label{eq:state}
 (w,z,R,D,A,H_{\mathrm{in}},Q,H_{\mathrm{out}}).
\end{equation}
Here $w$, $z$ are integers, $R$, $D$, $A$ are finite sets of nonnegative
integers, $H_{\mathrm{in}}$, $H_{\mathrm{out}}$ are words of length
twelve over $\{0,1,2,3\}$, and $Q$ is a nonempty word over the
same alphabet. On an actual partial board, the fields have the
meanings discussed in Section~\ref{sec:local-rank} and this subsection;
but we also allow states that do not correspond to actual board
configurations.
\end{definition}
Note that the absolute indices $n$, $m$, $d$ are not stored in the local state.
As a concrete example of a state, immediately before column $n=41$,
we have $m=25$, $d=14$, and $U(24)=16$, so
\[
 w=n-m-d=2,\qquad z=n-m-U(m-1)=0,\qquad
 R=A=\varnothing,\quad D=\{1,2\}.
\]
Here $H_{\mathrm{in}}=\sigma_{13}\dots\sigma_{24}=\texttt{212223003223}$,
$Q=\sigma_{25}\dots\sigma_{29}=\mathtt{01212}$, and
$H_{\mathrm{out}}=\sigma_{29}\dots\sigma_{40}=\texttt{212203230303}$.
Here $Q$ and $H_{\mathrm{out}}$ share the symbol $\sigma_{29}$. The
two words overlap exactly when $n-m\le|Q|+11$. On the actual board,
$Q$ has at most eleven symbols: It starts with eleven
(Section~\ref{app:seed}), and requests end by index $m+6$
(Lemma~\ref{lem:local-exactness}). The gap $n-m$, however, grows
without bound (Section~\ref{app:graph}), and the state before column
$41$ is the last in which they overlap.
The lower candidates are rows $(m,\dots,m+w)=(25,26,27)$. Their symbols
$\sigma_{25}\sigma_{26}\sigma_{27}=\texttt{012}$ have row bits
$b_{25}b_{26}b_{27}=\texttt{010}$. Thus the candidate $(41,26)$ is attacked along its row
by an upper queen. Figure~\ref{fig:state-portrait} shows all eight records.
\begin{figure}[!b]
\centering
\setlength{\abovecaptionskip}{0pt}
\begin{tikzpicture}[x=1cm,y=1cm,font=\footnotesize,>=stealth]
 \node[anchor=west,font=\small] at (.2,4.8) {The board before column $41$};
 \node[anchor=west,font=\small] at (6.5,4.8) {The lower attack records};
 \begin{scope}[shift={(.75,1.15)},x=.51cm,y=.51cm]
  \setboardfontsize{13pt}
  \foreach \xx in {0,...,5} \foreach \yy in {0,...,6} {
   \pgfmathtruncatemacro{\parity}{mod(\xx+\yy,2)}
   \ifnum\parity=0 \fill[black!4] (\xx-.5,\yy-.5) rectangle (\xx+.5,\yy+.5);\fi
  }
  \fill[queensLower!13] (3.5,2.5) rectangle (4.5,5.5);
  \draw[black!50] (4,-.5)--(4,6.5);
  \draw[queensLower!55,densely dashed] (-.5,.5)--(5.4,6.4);
  \draw[queensLower,thick] (0,0)--(4.45,4.45);
  \draw[queensLower,thick] (2,1)--(4.45,3.45);
  \draw[queensUpper,dashed] (-.5,4)--(4.5,4);
  \node[text=queensLower,inner sep=0pt] at (0,0) {\BlackQueenOnWhite};
  \node[text=queensLower,inner sep=0pt] at (2,1) {\BlackQueenOnWhite};
  \foreach \yy in {3,4,5} \node[queensLower] at (4,\yy) {$\times$};
  \node[anchor=east,text=queensLower] at (1.05,2.5) {$d=14$};
  \node[anchor=west] at (4.6,3) {$m=25$};
  \node[anchor=west] at (4.6,5) {$27$};
  \draw[<->,black!65] (7.85,3)--(7.85,5);
  \node[anchor=west] at (8.05,4) {$w=2$};
  \foreach \xx/\lab in {0/37,2/39,4/41}
   \node[anchor=north,text=black!65] at (\xx,-.7) {$\lab$};
  \node[anchor=south,queensUpper,font=\scriptsize] at (1.2,4.08) {upper row $26$};
 \end{scope}
 \draw[->,black!45] (5.15,2.3)--(6.05,2.3);
 \node[align=center,text=black!60,font=\scriptsize] at (5.6,1.77) {record\\offsets};
 \foreach \yy/\letter/\recset in {4.05/R/\varnothing,2.75/D/{\{1,2\}},1.45/A/\varnothing} {
  \node[anchor=west] at (6.5,\yy) {$\letter=\recset$};
  \draw[->,black!35] (9.25,\yy)--(14.3,\yy);
  \foreach \i in {0,...,4} {
   \draw[black!40] ({9.6+.9*\i},\yy-.08)--({9.6+.9*\i},\yy+.08);
   \node[anchor=south,font=\scriptsize,text=black!65] at ({9.6+.9*\i},\yy+.12) {$\i$};
  }
 }
 \foreach \i in {1,2}
  \filldraw[fill=queensLower!25,draw=queensLower] ({9.6+.9*\i-.16},2.59) rectangle ({9.6+.9*\i+.16},2.91);
 \node[anchor=west,text=black!60,font=\scriptsize] at (6.5,3.62) {rows measured from $m=25$};
 \node[anchor=west,text=black!60,font=\scriptsize] at (6.5,2.32) {magnitudes $15,16$, measured from $d=14$};
 \node[anchor=west,text=black!60,font=\scriptsize] at (6.5,1.02) {antidiagonals measured from $n+m=66$};
 \draw[black!15] (.2,.12)--(14.4,.12);
 \node[anchor=west,font=\small] at (.2,-.3) {The upper-word records};
 \node[anchor=west,font=\scriptsize,text=black!80] at (.2,-.72) {$z=n-m-U(m-1)=41-25-16=0$};
 \node[anchor=east,text=black!60,font=\scriptsize] at (14.4,-.3) {filled cell $=1$; empty cell $=0$};
 \begin{scope}[yshift=-.45cm]
 \foreach \s [count=\j from 0] in {2,1,2,2,2,3,0,0,3,2,2,3,0,1,2,1,2,1,2,2,0,3,2,3,0,3,0,3} {
  \pgfmathtruncatemacro{\ub}{floor(\s/2)}
  \pgfmathtruncatemacro{\bb}{mod(\s,2)}
  \ifnum\ub=1 \def\ucolor{queensUpper!55}\else\def\ucolor{white}\fi
  \ifnum\bb=1 \def\bcolor{queensUpper!55}\else\def\bcolor{white}\fi
  \filldraw[fill=\ucolor,draw=black!25,line width=.25pt] ({2.7+.39*\j},-1.4) rectangle ({3.09+.39*\j},-1.04);
  \filldraw[fill=\bcolor,draw=black!25,line width=.25pt] ({2.7+.39*\j},-1.9) rectangle ({3.09+.39*\j},-1.54);
 }
 \node[anchor=east] at (2.45,-1.22) {upper column $u_i$};
 \node[anchor=east] at (2.45,-1.72) {upper row $b_i$};
 \foreach \j/\lab in {0/13,12/25,16/29,27/40}
  \node[anchor=north,text=black!60,font=\scriptsize] at ({2.895+.39*\j},-2.02) {$\lab$};
 \draw[decorate,decoration={brace,amplitude=4pt},black!65] (8.94,-.86)--(13.62,-.86);
 \node at (11.28,-.52) {$H_{\mathrm{out}}$};
 \draw[decorate,decoration={brace,mirror,amplitude=4pt},black!65] (2.7,-2.48)--(7.38,-2.48);
 \node at (5.04,-2.88) {$H_{\mathrm{in}}$};
 \draw[decorate,decoration={brace,mirror,amplitude=4pt},black] (7.38,-2.48)--(9.33,-2.48);
 \node[text=black] at (8.355,-2.88) {$Q$};
 \draw[black!55] (7.38,-1)--(7.38,-1.94);
 \draw[->,black!55] (13.95,-2.45)--(13.95,-1.9);
 \node[anchor=north,text=black!65,font=\scriptsize] at (13.95,-2.52) {$n=41$};
 \end{scope}
\end{tikzpicture}
\par\nointerlineskip
\caption{The eight records $(w,z,R,D,A,H_{\mathrm{in}},Q,H_{\mathrm{out}})$ of the local state before column $41$.}
\label{fig:state-portrait}
\end{figure}

The separation into lower attack records and information about
upper queens is closely related to Carter's reduced description
\cite[Section~5]{Carter}.

\subsection{The history graph}\label{app:graph}
Given a state before column $n$, the calculation chooses the
queen in column~$n$, produces the new symbol $\sigma_n$, and updates
the eight records for column $n+1$. It reads only the stored records
and the history graph defined below, so it can be carried out on any
state, whether or not that state comes from the actual board.
Section~\ref{app:output} gives it in full as
Algorithm~\ref{alg:local-successors}. As this subsection explains,
it may produce several possible successor states.

The upper-attack tests read their bits from the queue. In the
column-$41$ example, the three candidate rows $25$, $26$, $27$ all lie
within $Q=\sigma_{25}\dots\sigma_{29}$, so their row bits
can be read directly. In general, however, the calculation may
need a symbol beyond the end of~$Q$. For instance, a candidate row
$m+r$ with $r\ge|Q|$ has its row bit $b_{m+r}$ outside the
queue. The antidiagonal test and the search for the next unused
row, described in Sections~\ref{app:upper} and~\ref{app:output},
can also read past the end of $Q$.

On the board itself, such a symbol is already determined. Its index
will be less than $n$ (Lemma~\ref{lem:local-exactness}), and the prefix
$\sigma_1\dots\sigma_{n-1}$ is fixed by the queens already placed.
The difficulty is that the local state does not contain those queens.
Nor can the state simply keep every symbol from index $m$ onward,
because the two ends of the word that it uses move apart: The gap
$n-m$ is $16$ before column $41$, and $381$ before column $1000$.
A record of the whole stretch from index $m$ to index $n-1$ would
grow without bound.

Instead of storing more symbols, we record which symbols may follow
which. When the calculation needs the symbol after the stored words,
it looks at their last twelve symbols and tries every symbol that
the history graph (defined below) allows after them. Write $\last(W)$ for the
last twelve symbols of a word $W$ of length at least twelve.

\begin{definition}[History graph]\label{def:history-graph}
The \emph{history graph} is a finite directed graph whose vertices
are twelve-symbol words $W$ over $\{0,1,2,3\}$. Each edge is
labeled by a symbol $s\in\{0,1,2,3\}$ and has the form
\[
 W\xrightarrow{s}\last(Ws).
\]
It is constructed in Section~\ref{sec:finite-statement} 
and stored in \texttt{history.json}.
\end{definition}

An edge $W\xrightarrow{s}W'$ says that the symbol $s$ may follow
the twelve symbols of $W$. Its destination $W'$ is the next
twelve-symbol window: Append $s$ and drop the oldest symbol.
We find this graph by running the calculation of this section
itself (Section~\ref{sec:finite-statement}). The proof does not
depend on how the graph was found: Once built, it is held fixed,
and Section~\ref{sec:closure-check} checks the properties we need
directly on it. The graph built there has $2092$ vertices and $2603$
edges, more than the queen word needs: The word passes through exactly
$1824$ of the vertices and uses exactly $2195$ of the edges
(Appendix~\ref{app:oeis-checks}). 

The calculation uses the history graph in two places: to supply symbols
of the earlier word that it needs but has not stored, and to check
the symbol that it produces for column~$n$.

\medskip\noindent
\textbf{Reading symbols from the graph.}
Suppose the calculation needs the symbol immediately after the end
of $Q$; we call this a \emph{request}. Since $Q$ continues directly
after $H_{\mathrm{in}}$, the twelve symbols before the missing one
are $\last(H_{\mathrm{in}}Q)$, ending at index $m+|Q|-1$. The
possible next symbols are the labels of the edges leaving this
vertex. If there is one such edge, its label is appended to $Q$. If
there are several, the calculation makes a separate copy of its
records for each edge, which we call a \emph{branch}, and continues
each branch independently. Every vertex of the history graph has at
least one outgoing edge (Section~\ref{sec:finite-statement}), so a
request has an answer whenever it reads from a vertex. A later request in the same
branch reads from the vertex that the last edge led to, so the
symbols appended to $Q$ in one branch are the labels along a walk in
the history graph. This walk starts at a vertex: Before column~$30$,
$\last(H_{\mathrm{in}}Q)$ is the vertex $\sigma_{18}\dots\sigma_{29}$
(Section~\ref{app:seed}), each request
moves along an edge, and the update at the end of the calculation,
which moves symbols from the front of $Q$ to the end of
$H_{\mathrm{in}}$, leaves $\last(H_{\mathrm{in}}Q)$ unchanged. So in
every state that the calculation reaches from the state before
column~$30$, any twelve consecutive symbols of $H_{\mathrm{in}}Q$ form
a vertex: Either they are a window of $\sigma_1\dots\sigma_{29}$, and
these windows are vertices (Section~\ref{app:seed}), or they are a
vertex on this walk. In particular, in these states
$H_{\mathrm{in}}$ is a vertex and every request has an answer.

For example, before column $53$ we have $m=32$ and
$Q=\sigma_{32}=\mathtt{2}$, and the calculation needs $\sigma_{33}$
(Section~\ref{sec:worked-transition}). The relevant vertex is
$\last(H_{\mathrm{in}}Q)=\sigma_{21}\dots\sigma_{32}$, which has two
outgoing edges:
\[
 \begin{tikzpicture}[baseline=(source.base),>=latex]
  \node[anchor=east,inner sep=2pt] (source) at (0,0)
    {\texttt{3\underline{22301212122}}};
  \node[anchor=west,inner sep=2pt] (zero) at (1.9,.42)
    {\texttt{\underline{22301212122}0}};
  \node[anchor=west,inner sep=2pt] (three) at (1.9,-.42)
    {\texttt{\underline{22301212122}3}};
  \draw[->,semithick] (source.east) --
    node[midway,above=4pt,inner sep=1pt,font=\small] {\texttt{0}} (zero.west);
  \draw[->,semithick] (source.east) --
    node[midway,below=4pt,inner sep=1pt,font=\small] {\texttt{3}} (three.west);
 \end{tikzpicture}
\]
The calculation therefore continues with two queues, $\mathtt{20}$
and $\mathtt{23}$. On the board $\sigma_{33}=\mathtt{0}$, so the first queue
is the actual one; the second is an extra possibility allowed by
the graph.

\medskip\noindent
\textbf{Checking the new symbol.}
The graph is used a second time, at the other end of the word.
We say that a word $\sigma_1\dots\sigma_T$, with $T\ge12$,
\emph{follows the history graph} if every window
$\sigma_{t-11}\dots\sigma_t$ with $12\le t\le T$ is a vertex, and
for $12\le t<T$ there is an edge
\[
 \sigma_{t-11}\dots\sigma_t\xrightarrow{\;\sigma_{t+1}\;}
 \sigma_{t-10}\dots\sigma_{t+1}.
\]
Once the calculation has chosen the queen in column $n$, it produces
the new symbol $\sigma_n$ (Section~\ref{app:output}). If
$\sigma_1\dots\sigma_{n-1}$ follows the graph, its last window is
$H_{\mathrm{out}}=\sigma_{n-12}\dots\sigma_{n-1}$, so
$\sigma_1\dots\sigma_n$ follows the graph exactly when
$H_{\mathrm{out}}$ has an outgoing edge labeled $\sigma_n$. The
\emph{output check} confirms that this edge exists; the calculation
then replaces $H_{\mathrm{out}}$ by its destination. Since
$H_{\mathrm{out}}$ is the vertex $\sigma_{18}\dots\sigma_{29}$ before
column $30$, it is a vertex in every state that the calculation
reaches from there. The output check is the second
stage of the step in Section~\ref{app:output}. This is why the
state stores $H_{\mathrm{out}}$: On the board, it is the vertex at
which the walk of the queen word currently ends.

\medskip\noindent
\textbf{The branch that describes the board.}
The two uses fit together. On the board, every requested symbol has
index less than $n$ (Lemma~\ref{lem:local-exactness}), so it belongs to
$\sigma_1\dots\sigma_{n-1}$. Section~\ref{sec:identification} proves
by induction that this prefix follows the graph, so each requested
symbol labels an edge leaving the vertex from which it is read. The
branch that answers each request with the actual symbol of the queen
word is therefore one of the branches of the calculation, and
Section~\ref{sec:sufficiency} shows that it
produces the actual next queen and records. Its output check extends
the walk of the queen word by one symbol, so the requests made in
the next column are covered in turn.

\medskip\noindent
\textbf{The other branches.}
The graph allows more than the queen word does: It has vertices that
never occur as subwords of~$\sigma$, and edges that $\sigma$ never
uses. The calculation follows every edge and cannot tell which
branch is the actual one. In particular, it never compares the
symbols it reads with those in $H_{\mathrm{out}}$, although on the
board the two can overlap: In the column-$41$ example, $\sigma_{29}$
is both the last symbol of $Q$ and the first symbol of
$H_{\mathrm{out}}$. The other branches are not discarded. Each one
must still pass the output check and give a
successor satisfying the bounds of Section~\ref{sec:sufficiency};
Section~\ref{sec:finite} verifies this for every branch.

\subsection{Recovering the upper attacks and counts}\label{app:upper}

The sets $R$, $D$, $A$ decide which lower candidates are attacked by lower
queens. An upper queen in row $m+r$ appears as the bit $b_{m+r}=1$ in
the queue, requested from the history graph if necessary. Two
questions about upper queens remain:
\begin{enumerate}
\item Does an upper queen attack a candidate $(n,m+r)$ along its
antidiagonal?
\item Does row $n$ contain an upper queen? This is the bit $b_n$
needed for the new symbol $\sigma_n=2u_n+b_n$.
\end{enumerate}
Propositions~\ref{prop:upper-antidiagonal-test}
and~\ref{prop:upper-row-count} answer both questions using only the
stored words and $z$. Their derivations refer to the actual board.
Both answers take into account only the upper queens in columns
whose symbols are stored in $H_{\mathrm{in}}$ and $Q$;
Section~\ref{sec:sufficiency} shows that, after the requests made in
Section~\ref{app:output}, no other upper queen can affect either
answer.

\medskip\noindent
\textbf{The columns represented by the stored symbols.}
Write a column as $m+h$, so that $h$ is its offset from $m$. The
symbols of $H_{\mathrm{in}}$ belong to columns $m-12$, \dots, $m-1$, and
those of $Q$ to columns $m$, \dots, $m+|Q|-1$. We call
\[
 h\in[-12\dts|Q|-1]
\]
the \emph{retained offsets}: For each of them the symbol of column
$m+h$ is stored, so the calculation can read
$u_{m+h}=\lfloor\sigma_{m+h}/2\rfloor$. Each request
adds the next nonnegative offset.

\medskip\noindent
\textbf{Relative upper-queen positions.}
By Lemma~\ref{lem:upper}, the upper queen in column $c$ has row
$c+U(c)$. The state stores neither the column nor the growing count
$U(c)$, so we measure both from a common reference. Write
$\kappa=U(m-1)$ for the number of upper queens in columns less than $m$.
Like $n$ and $m$, it is used only to interpret the records; the
calculation never computes it. For a retained offset $h$, the
difference $\Gamma(h)=U(m+h)-\kappa$ is a count of stored
column bits:
\begin{equation}\label{eq:C}
 \Gamma(h)=
 \begin{cases}
  \displaystyle\sum_{i=0}^{h}u_{m+i},&h\ge0,\\[4pt]
  \displaystyle-\sum_{i=h+1}^{-1}u_{m+i},&h<0.
 \end{cases}
\end{equation}
For $h\ge0$ it counts the upper columns $m$, \dots, $m+h$; for $h<0$ it
subtracts the upper columns $m+h+1$, \dots, $m-1$, so that
$\Gamma(-1)=0$. Substituting $U(m+h)=\kappa+\Gamma(h)$ into
Lemma~\ref{lem:upper}, the upper queen in column $m+h$, when
$u_{m+h}=1$, has
\begin{equation}\label{eq:source-offsets}
 \begin{array}{ll}
 \text{row}&m+\kappa+h+\Gamma(h),\\
 \text{antidiagonal index}&2m+\kappa+2h+\Gamma(h).
 \end{array}
\end{equation}
The terms $h+\Gamma(h)$ and $2h+\Gamma(h)$ can be computed from the
stored symbols; the references $m+\kappa$ and $2m+\kappa$ will cancel
from both tests. For example, before column $41$ we have $m=25$ and
$\kappa=16$. Among columns $25$, $26$, $27$, only $27$ is upper, so
$\Gamma(2)=1$, and the queen in column $27$ has row
$25+16+2+1=44$.

\medskip\noindent
\textbf{Locating the current column.}
Both tests compare an upper queen with a position involving the
current column $n$. This is why the state stores $z=n-m-\kappa$
(Definition~\ref{def:local-state}): In the same coordinates,
\begin{equation}\label{eq:relative}
 n=m+\kappa+z.
\end{equation}
In the column-$41$ example, $z=41-25-16=0$.

\Needspace{11\baselineskip}
\begin{proposition}[Upper-antidiagonal attack criterion]\label{prop:upper-antidiagonal-test}
Fix a lower candidate $(n,m+r)$, where $0\le r\le w$.
Let $m+h$ be an upper column represented in $H_{\mathrm{in}}$
or $Q$, so $h\in[-12\dts |Q|-1]$ and $u_{m+h}=1$.
The upper queen in column $m+h$ attacks the candidate along
its antidiagonal if and only if
\begin{equation}\label{eq:upper-sum}
 2h+\Gamma(h)=z+r.
\end{equation}
\end{proposition}

\begin{proof}
By \eqref{eq:source-offsets}, the upper queen's antidiagonal
index is $2m+\kappa+2h+\Gamma(h)$. The candidate's is
\[
 n+m+r=2m+\kappa+z+r.
\]
The two squares lie on the same antidiagonal precisely when
these indices agree. Canceling $2m+\kappa$ gives
\eqref{eq:upper-sum}.
\end{proof}

\medskip\noindent
\textbf{Counting upper rows to determine the new bit.}
Once the queen in column $n$ has been chosen, $u_n$ is known, and it
remains to find $b_n$. Write
\[
 B(T)=\sum_{i=1}^{T}b_i
\]
for the number of upper queens in rows at most $T$. Then
$b_n=B(n)-B(n-1)$, so it suffices to compute $B$ at the two
thresholds $n-1$ and $n$.

The count $\kappa$ includes upper queens according to their
\emph{columns}: exactly those in columns less than $m$. To count them by
their \emph{rows} instead, start from $\kappa$ and make two
adjustments:
\begin{itemize}
\item subtract the upper queens in columns less than $m$ whose rows are
greater than $T$;
\item add the upper queens in columns at least $m$ whose rows are at
most $T$.
\end{itemize}
Thus
\begin{equation}\label{eq:upper-row-adjustment}
 \begin{split}
 B(T)={}&\kappa-\#\{1\le c<m:u_c=1,\ q_c>T\}\\
        &+\#\{c\ge m:u_c=1,\ q_c\le T\}.
 \end{split}
\end{equation}

\begin{figure}[!htbp]
\centering
\begin{tikzpicture}[font=\small]
 \foreach \xx/\threshold/\caseTitle in {0/0/{Subtract from $\kappa$},7.65/2.0125/{Add to $\kappa$}} {
  \begin{scope}[xshift=\xx cm,x=.82cm,y=.84cm]
   \setboardfontsize{9pt}
   \node[font=\small\bfseries] at (-.3,4.0) {\caseTitle};
   \draw[black!65,line width=.65pt] (0,-2.27)--(0,2.65);
   \draw[black!65,line width=.65pt] (-2.95,\threshold)--(2.75,\threshold);
   \node[anchor=north] at (0,-2.43) {$m$};
   \node[anchor=west] at (2.87,\threshold) {$T$};

   \foreach \c/\r in {-2.60/-2.0125,-2.20/-1.4375,-1.00/-.2875,
                       -.60/.2875,-.20/.8625,.20/1.4375,1.40/2.5875} {
    \ifdim\r pt>\threshold pt
     \ifdim\c pt<0pt \def\qcolor{black!55}\else\def\qcolor{black!30}\fi
    \else \def\qcolor{queensUpper}\fi
    \node[text=\qcolor,inner sep=0pt] at (\c,\r) {\BlackQueenOnWhite};
   }

   \draw[decorate,decoration={brace,amplitude=4pt},black!65]
    (-2.95,2.52)--(-.12,2.52);
   \node[anchor=south] at (-1.535,2.81) {$\kappa=U(m-1)$};
   \draw[decorate,decoration={brace,amplitude=4pt},black!65]
    (-3.20,-2.16)--(-3.20,{\threshold-.08});
   \node[anchor=east] at (-3.48,{(\threshold-2.24)/2}) {$B(T)$};
  \end{scope}
 }
\end{tikzpicture}
\caption{The adjustments in \eqref{eq:upper-row-adjustment}. Left:
upper queens left of $m$ and above $T$ are subtracted from $\kappa$.
Right: those right of $m$ and at or below $T$ are added. Upper-queen
rows increase with their columns (Lemma~\ref{lem:upper}), so at most
one adjustment is nonzero. Teal queens are counted by $B(T)$.}
\label{fig:row-count-adjustment}
\end{figure}

For instance, on the column-$41$ board, counting upper rows through
$T=39$ means subtracting the queen at $(24,40)$ from the sixteen
counted by $\kappa$, with nothing to add, so $B(39)=15$.

For an upper queen in a retained column, the comparison with $T$ can
be made from the stored symbols. Write the threshold as
$T=m+\kappa+x$. By \eqref{eq:source-offsets}, the queen in column
$m+h$ has row $m+\kappa+h+\Gamma(h)$, so it lies at or below row $T$
exactly when $h+\Gamma(h)\le x$. Let $\J(x)$ be the net adjustment
computed from the retained columns, with additions from $Q$ and
subtractions from $H_{\mathrm{in}}$:
\begin{equation}\label{eq:J}
 \begin{split}
 \J(x)={}&\#\{0\le h<|Q|:u_{m+h}=1,\ h+\Gamma(h)\le x\}\\
 &-\#\{-12\le h<0:u_{m+h}=1,\ h+\Gamma(h)>x\}.
 \end{split}
\end{equation}

\begin{proposition}[Upper-row counting formula]\label{prop:upper-row-count}
Fix a nonnegative integer threshold $T$ and write
$T=m+\kappa+x$. Suppose that every upper queen in a column
less than $m-12$ lies at or below row $T$, and every upper
queen in a column at least $m+|Q|$ lies above row $T$. Then
\[
 B(T)=\kappa+\J(x).
\]
If these assumptions hold for both $T=n-1$ and $T=n$, then
\begin{equation}\label{eq:J-count}
 B(n-1)=\kappa+\J(z-1),\qquad B(n)=\kappa+\J(z),
\end{equation}
and hence
\[
 b_n=\J(z)-\J(z-1).
\]
\end{proposition}

\begin{proof}
In \eqref{eq:upper-row-adjustment}, the first assumption means
that no upper queen before column $m-12$ needs to be subtracted.
Every required subtraction is therefore represented in
$H_{\mathrm{in}}$. The second assumption means that no upper
queen in column $m+|Q|$ or later needs to be added, so every
required addition is represented in $Q$. The two adjustments
are exactly the two terms defining $\J(x)$ in \eqref{eq:J}.
This proves $B(T)=\kappa+\J(x)$.

Since $n=m+\kappa+z$, the thresholds $T=n-1,n$ correspond
to $x=z-1,z$, respectively. Substituting these values gives
\eqref{eq:J-count}. Taking the difference cancels $\kappa$
and gives the stated formula for $b_n$.
\end{proof}

The calculation therefore finds $b_n$ from the two adjustments
$\J(z-1)$ and $\J(z)$, without knowing $\kappa$. Neither $\Gamma$ nor
$\J$ is stored; both are computed from the stored words when needed.

\subsection{The calculation}\label{app:output}

We now give the whole calculation on a state before column $n$. It 
has four stages: choosing the queen, producing and checking the 
new symbol \(\sigma_n\), updating the row and diagonal references, 
and updating the input words. Whenever a stage needs a symbol 
beyond the end of $Q$,
the calculation makes a request, answered from the history graph as
in Section~\ref{app:graph}. This may split the calculation into
branches, and the remaining stages are carried out separately in
each branch. Symbols appended to $Q$ stay there for the rest of the
branch, even if the candidate that needed them is rejected. When a
request for $\sigma_k$ is answered, the branch commits to that answer
as its value of $\sigma_k$. Keeping the symbol ensures that later tests
in the branch use the same value of $\sigma_k$, and that $Q$ remains
the word $\sigma_m\dots\sigma_{m+|Q|-1}$.
Throughout, $u_{m+h}$ and $b_{m+h}$ denote the two bits of the stored
symbol at offset $h$; the calculation never uses the values of $n$,
$m$, or $d$.

Before choosing the queen, extend $Q$ by requests, if necessary, to
length at least $z$. This supplies the symbols needed later for
$\J(z-1)$ and $\J(z)$.

\medskip\noindent
\textbf{Choose the queen.}
The lower candidates are $(n,m+r)$ for $r=0,\dots,w$; there are none
when $w<0$. Test them in increasing order of $r$, so that the first
one accepted is the lowest. First test attacks from lower queens:
Reject the candidate if $r\in R$, $w-r\in D$, or $r\in A$. 
A~candidate that passes these three tests is then tested for upper
attacks along its row and antidiagonal. The row test reads
$b_{m+r}$, the row bit at offset $r$. By
Proposition~\ref{prop:upper-antidiagonal-test}, the upper queen in
column $m+h$ attacks along the antidiagonal only if
$2h+\Gamma(h)=z+r$. For $h\ge0$ we have $\Gamma(h)\ge0$, so such an
offset satisfies $h\le\lfloor(z+r)/2\rfloor$; the negative retained
offsets are already stored in $H_{\mathrm{in}}$. Both tests are
therefore covered by extending $Q$, if necessary, to length at least
\[
 1+\max\{r,\lfloor(z+r)/2\rfloor\}.
\]
Reject the candidate if $b_{m+r}=1$, or if some retained offset $h$
with $u_{m+h}=1$ satisfies \eqref{eq:upper-sum}. Choose the first
candidate that passes all five tests. If there are no candidates, or
all are rejected, choose an upper queen. Knuth's \texttt{infty-queens} program \cite{KnuthQueens} also tries
the lower candidates first and places an upper queen only when all of
them are attacked; here the tests use only the local records.

\medskip\noindent
\textbf{Produce and check the new symbol.}
The choice gives $u_n=1$ if the queen is upper and $u_n=0$ if it is
lower. Following Proposition~\ref{prop:upper-row-count}, the
calculation computes
\begin{equation}\label{eq:output}
 b_{\mathrm{out}}=\J(z)-\J(z-1),\qquad
 s=2u_n+b_{\mathrm{out}}.
\end{equation}
It then performs the output check: It confirms that $s$ labels an
edge leaving $H_{\mathrm{out}}$, and
replaces $H_{\mathrm{out}}$ by $\last(H_{\mathrm{out}}s)$. In the
exhaustive check of Section~\ref{sec:closure-check}, a missing output edge in any
branch makes the whole verification fail.

\medskip\noindent
\textbf{Update the row and diagonal references.}
Placing the queen may use up the least unused row $m$ or the least
unused magnitude $d$, so the next state measures its records from
new reference values. First record the queen if it is lower, by
inserting $r$, $w-r$, and $r$ into $R$, $D$, and $A$; an upper queen
adds nothing, since these sets record only lower queens. Write
$\widetilde R$, $\widetilde D$, $\widetilde A$ for the resulting sets.
Write $\mu$ and $\nu$ for the advances of the row and diagonal
references, so that the new least unused row is $m+\mu$ and the new
least unused magnitude is $d+\nu$. Measured from the current
references $m$ and $d$, row $m+h$ is used when a recorded lower queen
($h\in\widetilde R$) or an upper queen ($b_{m+h}=1$) occupies it, and
magnitude $d+h$ is used when $h\in\widetilde D$; upper queens lie on
no lower diagonal. Hence
\[
 \mu=\min\{h\ge0:h\notin\widetilde R,\ b_{m+h}=0\},\qquad
 \nu=\min(\N\setminus\widetilde D).
\]
The search for $\mu$ reads $b_{m+h}$ for $h=0,\dots,\mu$, making
requests when necessary; the zero bit at offset $\mu$ confirms that
this row is unused. When the new queen is upper,
Section~\ref{sec:sufficiency} proves that the search stops below its
row.

The new references are $n+1$, $m+\mu$, and $d+\nu$. Measuring each
recorded position from its new reference, and discarding offsets
that become negative, gives
\begin{equation}\label{eq:update}
 \begin{aligned}
 w'&=w+1-\mu-\nu,\\
 z'&=z+1-\mu-\textstyle\sum_{h=0}^{\mu-1}u_{m+h},\\
 R'&=\{a-\mu:a\in\widetilde R,\ a\ge\mu\},\\
 D'&=\{a-\nu:a\in\widetilde D,\ a\ge\nu\},\\
 A'&=\{a-(1+\mu):a\in\widetilde A,\ a\ge1+\mu\}.
 \end{aligned}
\end{equation}
The formula for $w'$ follows from $w=n-m-d$. For $z=n-m-U(m-1)$, the
column advances by one and the row reference by $\mu$, while
$U(m+\mu-1)$ exceeds $U(m-1)$ by the number of upper columns among
$m$, \dots, $m+\mu-1$; the row search has read all of their symbols. The
antidiagonal reference $n+m$ advances by $1+\mu$. Discarding negative
offsets loses nothing: The references $m$, $d$, and $n+m$ never
decrease, so a row, diagonal, or antidiagonal whose index has fallen
below its reference cannot contain a later candidate.

\medskip\noindent
\textbf{Update the input words.}
The input history must now end at index $m+\mu-1$, and the queue
must begin at index $m+\mu$. Remove the first $\mu$ symbols of the
extended $Q$ and append them to $H_{\mathrm{in}}$, keeping its last
twelve symbols. The symbol at offset $\mu$ stays as the first symbol
of $Q'$, so $Q'$ is nonempty. These two words, together with
$w'$, $z'$, $R'$, $D'$, $A'$ and the new output history, form the successor
state.

\medskip
Algorithm~\ref{alg:local-successors} collects the four stages. It
follows a single branch: $Q$ denotes the queue of that branch, which
grows as requests are made, and \Call{Extend}{$k$} makes requests
until $|Q|\ge k$. Each request may split the calculation, as
described in Section~\ref{app:graph}. In the column-$53$ example
of that section, \Call{Extend}{$2$} makes one request, for
$\sigma_{33}$, which splits the calculation into branches with queues
$\mathtt{20}$ and $\mathtt{23}$. The successors of a state are the
states returned by all of its branches. In the
listing, $\J$ and $\Gamma$ are computed from $H_{\mathrm{in}}$ and the
current $Q$.

\begin{algorithm}[!t]
\caption{One branch of the calculation}
\label{alg:local-successors}
\renewcommand{\algorithmicrequire}{\textbf{Input:}}
\renewcommand{\algorithmicensure}{\textbf{Output:}}
\begin{algorithmic}[1]
\Require A state $(w,z,R,D,A,H_{\mathrm{in}},Q,H_{\mathrm{out}})$.
\Ensure A successor state, or failure if the output check fails.
\State \Call{Extend}{$z$}\Comment{prepare for $\J(z-1)$, $\J(z)$}
\State $c\gets\texttt{UPPER}$; $r\gets0$\Comment{choose the queen}
\While{$c=\texttt{UPPER}$ and $r\le w$}
  \If{$r\notin R$ and $w-r\notin D$ and $r\notin A$}\Comment{no lower-queen attack}
    \State \Call{Extend}{$1+\max\{r,\lfloor(z+r)/2\rfloor\}$}\Comment{reach the symbols the tests need}
    \If{$b_{m+r}=0$ and no retained offset $h$ with $u_{m+h}=1$ has $2h+\Gamma(h)=z+r$}
      \State $c\gets r$\Comment{no upper-queen attack}
    \EndIf
  \EndIf
  \State $r\gets r+1$
\EndWhile
\State Set $u_n\gets1$ if $c=\texttt{UPPER}$ and $u_n\gets0$ otherwise.\Comment{produce the new symbol}
\State Compute $b_{\mathrm{out}}$ and $s$ by \eqref{eq:output}.
\If{$H_{\mathrm{out}}$ has no outgoing edge labeled $s$}\Comment{output check}
  \State \Return failure
\EndIf
\State $H'_{\mathrm{out}}\gets\last(H_{\mathrm{out}}s)$.
\State Form $\widetilde R$, $\widetilde D$, $\widetilde A$, inserting $c$, $w-c$, $c$ if $c\ne\texttt{UPPER}$.\Comment{update the references}
\State $\nu\gets\min(\N\setminus\widetilde D)$; $\mu\gets0$.
\While{$\mu\in\widetilde R$ or $b_{m+\mu}=1$}\Comment{row $m+\mu$ is used}
  \State $\mu\gets\mu+1$; \Call{Extend}{$\mu+1$}.
\EndWhile
\State Compute $w'$, $z'$, $R'$, $D'$, $A'$ by \eqref{eq:update}.
\State Split $Q=PQ'$ with $|P|=\mu$, and set $H'_{\mathrm{in}}\gets\last(H_{\mathrm{in}}P)$.\Comment{update the input words}
\State \Return $(w',z',R',D',A',H'_{\mathrm{in}},Q',H'_{\mathrm{out}})$
\end{algorithmic}
\end{algorithm}

Starting from the state before column $30$
(Section~\ref{app:seed}) and applying the calculation to each
successor in turn produces a directed graph of states, which
Section~\ref{sec:finite} examines exhaustively.

\begin{definition}[State graph]\label{def:state-graph}
The \emph{state graph} has as vertices the
states reached from the state before column $30$ by repeating the
calculation. It has a directed edge $S\to S'$ when some branch of
the calculation on $S$ produces $S'$.
\end{definition}

Different branches may produce the same successor; they then give a
single edge. Section~\ref{sec:finite} verifies that every output
check passes, that every successor satisfies
Condition~\ref{cond:state}, and that only finitely many states are
reached. The definition assumes none of these properties. In fact the
state graph has $7014$ states and $8327$ edges
(Section~\ref{sec:closure-check}), and the actual board passes through
exactly $4749$ of these states and uses exactly $5207$ of the edges
(Appendix~\ref{app:oeis-checks}).

\subsection{Worked examples of the state transition}\label{sec:worked-transition}

We first carry out the complete calculation for column $41$,
starting from the state of Figure~\ref{fig:state-portrait}:
\[
 \begin{gathered}
 w=2,\qquad z=0,\qquad R=A=\varnothing,\qquad D=\{1,2\},\\
 H_{\mathrm{in}}=\mathtt{212223003223},\qquad Q=\mathtt{01212},\\
 H_{\mathrm{out}}=\mathtt{212203230303}.
 \end{gathered}
\]
The three words hold indices $13$--$24$, $25$--$29$, and $29$--$40$.
The board values $n=41$, $m=25$, $d=14$, and $\kappa=16$ only
explain the numbers; they are not stored. The queue already
has every symbol this step needs, so the calculation makes no
request and does not branch. We then look briefly at three
further steps: one in which the input words move, one in which an
upper attack rules out the only lower candidate, and one in which
the calculation branches.

\medskip\noindent
\textbf{Choose the queen.}
Since $w=2$, the lower candidates have offsets $r=0,1,2$, in rows
$25$, $26$, $27$. Since $R$ and $A$ are empty, no lower queen attacks a
candidate along its row or antidiagonal. The lower-diagonal test
gives
\begin{center}
\begin{tabular}{clcl}
\toprule
$r$&Candidate&$w-r$&Lower-diagonal test\\
\midrule
$0$&$(41,25)$&$2$&Reject: $2\in D$\\
$1$&$(41,26)$&$1$&Reject: $1\in D$\\
$2$&$(41,27)$&$0$&Pass: $0\notin D$\\
\bottomrule
\end{tabular}
\end{center}
Only $(41,27)$ reaches the upper tests. Its row bit $b_{27}$ comes
from $\sigma_{27}=\mathtt{2}$ and is zero. An antidiagonal attack needs
an upper column with
\[
 2h+\Gamma(h)=z+r=2.
\]
The upper columns in $H_{\mathrm{in}}$ have $h<0$ and $\Gamma(h)\le0$,
so for them $2h+\Gamma(h)\le-2$. In $Q=\mathtt{01212}$, the upper
columns are at offsets $h=2,4$, with $\Gamma(h)=1,2$, giving
$2h+\Gamma(h)=5$ and $10$. The candidate passes all five tests, and the queen is placed
at $(41,27)$. It is the $j=d+|D|=16$th lower queen, with magnitude
$14$, so \eqref{eq:local-rank} gives the discrepancy bounded by
Lemma~\ref{lem:estimates}:
\[
 d_j-j=w-r-|D|=2-2-2=-2.
\]

\medskip\noindent
\textbf{Produce and check the new symbol.}
The queen is lower, so $u_{41}=0$. Since $z=0$, the new row bit $b_{41}$
comes from $\J(-1)$ and $\J(0)$. The upper columns in $H_{\mathrm{in}}$
have relative rows $h+\Gamma(h)\le-1$, so neither $\J(-1)$ nor
$\J(0)$ subtracts them. Those in $Q$ have relative rows $2+1=3$ and
$4+2=6$, so neither adds them either. Hence
\[
 \J(-1)=\J(0)=0,\qquad
 b_{\mathrm{out}}=0-0=0,\qquad s=2\cdot0+0=0.
\]
On the board the two thresholds are rows $40$ and $41$. The upper
rows nearest them are $40$ and $44$, so $B(40)=B(41)$, which is why
$b_{41}=0$. The output check
succeeds, since the history graph has the edge
\[
 \mathtt{212203230303}\xrightarrow{\;\texttt{0}\;}
 \mathtt{122032303030},
\]
and $H_{\mathrm{out}}$ becomes its destination.

\medskip\noindent
\textbf{Update the row and diagonal references.}
The new queen has row $27$, lower-diagonal magnitude $14$, and
antidiagonal index $68$. Their offsets from the references
$25$, $14$, $66$ are $2$, $0$, $2$, so
\[
 \widetilde R=\{2\},\qquad
 \widetilde D=\{0,1,2\},\qquad
 \widetilde A=\{2\}.
\]
Row $25$ is still unused: The candidate $(41,25)$ was rejected
because of a diagonal attack, not because its row was occupied.
Hence $\mu=0$. Magnitude $14$ fills the gap below the used
magnitudes $15$ and $16$, so the least unused magnitude becomes $17$
and $\nu=3$. By \eqref{eq:update},
\[
 R'=\{2\},\qquad D'=\varnothing,\qquad A'=\{1\},\qquad
 w'=2+1-0-3=0,\qquad z'=0+1-0-0=1.
\]
All three offsets in $\widetilde D$ are discarded, because the
magnitudes $14$, $15$, $16$ now lie below the reference $17$. On the board,
$w'=42-25-17=0$ and $z'=42-25-16=1$, so column $42$ has the single
lower candidate $(42,25)$.

\medskip\noindent
\textbf{Update the input words.}
Since $\mu=0$, no symbols move, and $H_{\mathrm{in}}$ and $Q$ are
unchanged. Table~\ref{tab:local-transition} summarizes the step,
which is one edge of the state graph.
\begin{table}[!htbp]
\centering
\begin{tabular}{lll}
\toprule
Field&Before column $41$&Before column $42$\\
\midrule
$w$&$2$&$0$\\
$z$&$0$&$1$\\
$R$&$\varnothing$&$\{2\}$\\
$D$&$\{1,2\}$&$\varnothing$\\
$A$&$\varnothing$&$\{1\}$\\
$H_{\mathrm{in}}$&\texttt{212223003223}&\texttt{212223003223}\\
$Q$&\texttt{01212}&\texttt{01212}\\
$H_{\mathrm{out}}$&\texttt{212203230303}&\texttt{122032303030}\\
\bottomrule
\end{tabular}
\caption{One edge of the state graph. The input words do not move
because the least unused row is still $25$; the output history
shifts and appends $\mathtt{0}$.}\label{tab:local-transition}
\end{table}

\medskip\noindent
\textbf{\boldmath When the input words move: column $42$.}
The sole lower candidate $(42,25)$ passes all five tests. After it is
placed, the search for the least unused row passes rows $25$
through $28$, occupied respectively by the new lower queen, an upper
queen, the queen placed in column $41$, and an upper queen. Row $29$
is unused, so $\mu=4$. The first four symbols $\mathtt{0121}$ of
$Q=\mathtt{01212}$, for indices $25$ through $28$, move into
$H_{\mathrm{in}}$:
\[
 \begin{aligned}
 H'_{\mathrm{in}}&=\last(H_{\mathrm{in}}\,\mathtt{0121})
                  =\mathtt{230032230121},\\
 Q'&=\mathtt{2}.
 \end{aligned}
\]
The new input history covers indices $17$ through $28$, and
$Q'=\sigma_{29}$, whose zero row bit confirms that row $29$ is
unused. Among the moved columns only $27$ is upper, so
$z'=1+1-4-1=-3$, in agreement with $43-29-U(28)=43-29-17$. No
request was needed: The step moved symbols already in $Q$.

\medskip\noindent
\textbf{\boldmath When a lower candidate is attacked: column $47$.}
Before column $47$, we have $m=29$, $w=0$, $z=1$, and
$R=D=A=\varnothing$. The sole lower candidate $(47,29)$ passes the
three lower-queen tests, and $Q=\mathtt{2}$ gives $b_{29}=0$, so no
upper queen occupies its row. However, the upper queen $(29,47)$
shares its antidiagonal $76$. This queen has offset $h=0$ and
$\Gamma(0)=1$, so \eqref{eq:upper-sum} holds: $2h+\Gamma(h)=1=z+r$
with $r=0$. The candidate is rejected, and the $30$th upper queen is
placed at $(47,77)$. Thus a nonempty candidate interval can still
end with an upper queen.

\medskip\noindent
\textbf{\boldmath When the calculation branches: column $53$.}
As in Section~\ref{app:graph}, before column $53$ we have $m=32$,
$Q=\mathtt{2}$, and $z=2$. The calculation first extends $Q$ to
length two, and the vertex $\mathtt{322301212122}$ has two outgoing
edges, giving two branches with queues $\mathtt{20}$ and
$\mathtt{23}$. The appended symbols are the two values allowed for
the earlier symbol $\sigma_{33}$, not choices of a queen in
column $53$.

Both branches choose the lower square $(53,32)$ and produce
$\sigma_{53}=\mathtt{0}$, but their record updates differ. With
$Q=\mathtt{20}$, the row bit at index $33$ is zero, so the next
unused row is $33$ and $\mu=1$. With $Q=\mathtt{23}$, that bit is
one, and the row search makes another request. The vertex reached
has a single outgoing edge, labeled $\mathtt{0}$, so $Q=\mathtt{230}$ and the
search stops at row $34$, giving $\mu=2$. The two branches therefore
produce different successors, although they chose the same queen and
produced the same symbol. On the board $\sigma_{33}=\mathtt{0}$, so the first
branch gives the actual next state; the verification checks the
second as well.
%
\section{Why the local records suffice}\label{sec:sufficiency}

Section~\ref{sec:local} described a calculation on states. We now
show that, when the state comes from the board, one of its branches
carries out the actual greedy step. Two things could go wrong: An
upper queen whose symbol is not stored could affect an attack test or
the row count, or the calculation could request a symbol that the
board has not yet determined. The bounds in this section rule out
both.

The argument concerns a single column $n$: It assumes that the records
before column $n$ satisfy the bounds below. Section~\ref{sec:finite}
checks that the state before column $30$ satisfies them and that every
successor does too, and Section~\ref{sec:identification} combines these
checks with the result of this section by induction on $n$.

\begin{condition}[Bounds on the stored records]\label{cond:state}
The records satisfy
\begin{equation}\label{eq:state-bounds}
 w\le4,\qquad -4\le z\le5,\qquad R,D\subseteq[1\dts4].
\end{equation}
\end{condition}

Each bound has a specific use. The bounds on $w$ and $D$ give the
diagonal discrepancy estimate, as explained after
\eqref{eq:local-rank}. The upper bound on $z$, together with the
bounds on $w$ and~$R$, keeps every request within six places of
index $m$. The lower bound on $z$, together with $U(m-1)\ge12$, then
guarantees that every requested symbol is already determined, and it
keeps upper queens before the input history from affecting the
tests.

These bounds alone do not confine the records to finitely many
states. That only finitely many states are reached is a result of
the verification in Section~\ref{sec:finite}.

\Needspace{19\baselineskip}
\begin{lemma}[The records determine the actual step]\label{lem:local-exactness}
Consider the actual board before column~$n$ and its records
(Definition~\ref{def:local-state}). Suppose that
\begin{enumerate}
\item the records satisfy Condition~\ref{cond:state};
\item $U(m-1)\ge12$ and $m+|Q|\le n$; and
\item $\sigma_1\dots\sigma_{n-1}$ follows the history graph.
\end{enumerate}
Carry out the calculation, answering each request with the symbol of
the queen word at the requested index. We call this computation the
\emph{actual branch}. Then:
\begin{enumerate}
\renewcommand{\labelenumi}{(\roman{enumi})}
\item every requested index is at most $m+6<n$, and each answer
labels an edge leaving the vertex from which it is read, so the
actual branch is one of the branches of the calculation;
\item the branch places the actual queen in column $n$, produces
$\sigma_n$, and gives the actual records before column $n+1$;
\item the row reference advances by at most six places.
\end{enumerate}
\end{lemma}

\begin{proof}
As before, write $\kappa=U(m-1)$. Since $z=n-m-\kappa$ on the board,
the hypotheses give
\begin{equation}\label{eq:causality}
 n-m=\kappa+z\ge12-4=8.
\end{equation}

We first show that every requested index is less than $n$. The preliminary
request, to length $z\le5$, reaches at most offset
$4$. A candidate request reaches offset
$\max\{r,\lfloor(z+r)/2\rfloor\}\le4$, since $r\le w\le4$ and
$z+r\le9$. Only the row search goes further. After a lower choice is
inserted, $\widetilde R\subseteq[0\dts4]$, since $R\subseteq[1\dts4]$
and $r\le w\le4$. Of the rows at offsets $5$ and $6$, at least one
has row bit zero, because successive upper rows differ by at least
two (Lemma~\ref{lem:upper}). Hence an unused row occurs by
offset~$6$, and $\mu\le6$. This proves (iii).

Every request is therefore for an index at most $m+6<n$, and the
symbols already in $Q$ have indices less than $n$ by hypothesis~(2). So every bit that the actual branch reads is
already determined by the queens placed before column $n$. In
particular, if the queen placed in column $n$ is upper, its row is
greater than $n$, so it cannot affect a row search that ends by
$m+6$.

Next we show that the actual branch is one of the branches of the
calculation. Every requested symbol lies in $\sigma_1\dots\sigma_{n-1}$, which
follows the history graph by hypothesis~(3), and the vertex from which
it is read consists of the twelve actual symbols before it. So the
symbol labels an edge leaving that vertex, and the calculation follows
that edge.
With the bound on the requests, this proves (i).

It remains to prove (ii). We begin by showing that upper queens outside the
retained columns do not affect the tests. On the board, extend $\Gamma(h)=U(m+h)-\kappa$ to all offsets with
$m+h\ge1$; it agrees with \eqref{eq:C} on the retained offsets. By
\eqref{eq:source-offsets}, an upper queen in column $m+h$ has
relative row $h+\Gamma(h)$ and relative antidiagonal index
$2h+\Gamma(h)$. The tests compare relative rows with the thresholds
$z-1$ and~$z$, both at least $-5$, and relative antidiagonal indices
with $z+r\ge-4$.

An upper queen before the input history has $h\le-13$ and
$\Gamma(h)\le0$, so its relative row is at most $-13$ and its
relative antidiagonal index is at most $-26$. It lies below both
thresholds, so $B(T)$ should count it; $\kappa$ already does, and no
subtraction is needed. Its antidiagonal index is too small to equal
$z+r\ge-4$, so it attacks no candidate.

An upper queen after the stored queue has $h\ge|Q|$ and
$\Gamma(h)\ge1$, since $\Gamma(h)$ counts column $m+h$ itself. When
$\J(z-1)$ and $\J(z)$ are computed, $|Q|\ge z$, so such a queen has
relative row at least $|Q|+1>z$ and is added to neither count. When
candidate $r$ is tested,
$|Q|>\lfloor(z+r)/2\rfloor$, so such a queen has relative
antidiagonal index at least $2|Q|+1>z+r$ and cannot attack the
candidate.

The sets $R$, $D$, $A$ detect every attack by a lower queen; the offsets
they discard are irrelevant, as shown in Section~\ref{app:output}. The
row bit $b_{m+r}$, read from $Q$, detects an upper queen in the
candidate's row, and by the previous paragraphs
Proposition~\ref{prop:upper-antidiagonal-test} detects every upper
queen on its antidiagonal. Conversely, all these attacks come from
queens already placed: The retained columns are less than $n$, and an
upper queen in row $m+r<n$ lies in a column less than $m+r$. The
branch therefore chooses the first unattacked lower candidate, as the
greedy rule does. If there is none, the actual queen is upper, and so
is the branch's choice. The two previous paragraphs verify both
hypotheses of Proposition~\ref{prop:upper-row-count} at the thresholds
$n-1$ and $n$, so it gives the actual row bit $b_n$, and hence the
actual symbol $\sigma_n$ in \eqref{eq:output}.

With the actual symbols, the searches find the actual advances $\mu$
and~$\nu$, \eqref{eq:update} gives the actual new sets and the actual
new values of $w$ and $z$, and the word updates keep exactly the
segments of $\sigma$ at their new positions, all with indices less
than $n+1$. With the previous paragraph, this proves (ii).
\end{proof}

Lemma~\ref{lem:local-exactness} does not assert that the output check
passes, or that the successor satisfies Condition~\ref{cond:state}. The exhaustive check
of Section~\ref{sec:closure-check} establishes both.
%
\section{The finite verification and its consequences}\label{sec:finite}

We now join the two parts of the proof. By
Lemma~\ref{lem:local-exactness}, the actual branch carries out the
actual next step. The computer checks that
every branch passes the output check and gives a successor satisfying
Condition~\ref{cond:state}. Induction then repeats these two facts for
every column.

\subsection{The starting board}\label{app:seed}

Generate columns $0$ through $29$ directly from the greedy rule.
Their rows $q_0$, \dots, $q_{29}$ are
\[
\begin{array}{rrrrrrrrrr}
0&2&4&1&3&8&10&12&14&5\\
7&18&6&21&9&24&26&28&30&11\\
13&34&36&38&40&15&17&44&16&47\mathrlap{.}
\end{array}
\]
They give the following state before column $30$:
\begin{center}
\begin{tabular}{ll}
\toprule
Stored field&Value\\
\midrule
$w$&$0$\\
$z$&$-1$\\
$R$&$\varnothing$\\
$D$&$\{1\}$\\
$A$&$\varnothing$\\
$H_{\mathrm{in}}$&\texttt{230121212223}\\
$Q$&\texttt{00322301212}\\
$H_{\mathrm{out}}$&\texttt{300322301212}\\
\bottomrule
\end{tabular}
\end{center}
Here $m=19$, $d=11$, and $\kappa=U(18)=12$, so
$w=30-19-11=0$ and $z=30-19-12=-1$. The input history covers indices $7$--$18$, the
queue $19$--$29$, and the output history $18$--$29$. The verifiers
compute these data from the board rather than accepting the table.
They also check that every lower queen in this prefix satisfies
$|d_j-j|\le4$, and that $\sigma_1\dots\sigma_{29}$ follows the
history graph.

Column $30$ is chosen so that the input history has positive
indices and $\kappa\ge12$, as Lemma~\ref{lem:local-exactness}
requires. On the board $m$ never decreases, so $U(m-1)\ge12$
remains true in every later column. No periodicity of the board is
assumed: Section~\ref{sec:identification} follows the actual queens.

\Needspace{26\baselineskip}
\subsection{Constructing the history graph}\label{sec:finite-statement}

Algorithm~\ref{alg:construction} builds the history graph by running
the calculation of Section~\ref{app:output}, with two differences: A computed symbol that fails the
output check adds its edge and destination vertex to the graph
instead, and a request at a vertex without outgoing edges ends its
branch. The algorithm starts from the eighteen windows of
$\sigma_1\dots\sigma_{29}$, those ending at indices $12$, \dots, $29$, and
their seventeen edges. A successor violating
Condition~\ref{cond:state} ends the construction with failure.

\begin{algorithm}[H]
\caption{Constructing the history graph}
\label{alg:construction}
\renewcommand{\algorithmicensure}{\textbf{Output:}}
\begin{algorithmic}[1]
\Ensure The history graph, or failure.
\State $G\gets$ the windows of $\sigma_1\dots\sigma_{29}$ and the edges between consecutive windows.
\State $\mathcal S\gets\{$the state before column $30\}$.
\Repeat\Comment{one exploration}
  \State Mark every state in $\mathcal S$ as unexamined.
  \While{some state $S\in\mathcal S$ is unexamined}
    \State Mark $S$ as examined.
    \State Run Algorithm~\ref{alg:local-successors} on every branch from $S$, adding each missing output edge to $G$.
    \For{each successor $S'$ found}
      \If{$S'$ violates Condition~\ref{cond:state}}
        \State \Return failure
      \EndIf
      \State Add $S'$ to $\mathcal S$ as unexamined, unless it is already in $\mathcal S$.
    \EndFor
  \EndWhile
\Until{this exploration added no edge to $G$}
\State \Return $G$
\end{algorithmic}
\end{algorithm}

Before column $30$, the queue already holds every symbol the
calculation needs, so the first step makes no request. The queen goes
to $(30,19)$. It is lower, so $u_{30}=0$, and row $30$ contains the
upper queen $(18,30)$, so $b_{30}=1$ and $\sigma_{30}=\mathtt{1}$. The
starting graph has no edge leaving
$H_{\mathrm{out}}=\texttt{300322301212}$, so the output check fails,
and the construction adds the edge
\[
 \texttt{300322301212}\xrightarrow{\;\texttt{1}\;}
 \texttt{003223012121}.
\]

An added edge is a new answer to later requests, so it can create
new branches at states already examined. The construction therefore
repeats the exploration over every state found so far until one adds
no edge. A request with no outgoing edge ends its branch for that
exploration; a later one continues it if an edge has since been
added. In our run, thirteen explorations
add edges and the fourteenth adds none, giving $2092$ vertices and
$2603$ edges. Every vertex of the result has at least one outgoing
edge, so with the finished graph no request goes unanswered.

The result does not depend on the order of exploration. Enlarging
the graph never removes a branch, since a missing edge only ends
one. Hence every edge that the construction adds belongs to every
graph that contains the starting windows and has an edge for each
output that it permits. The final exploration shows that the
constructed graph has this property, so it is the smallest such
graph.

The graph contains many words that never occur in $\sigma$. Because
the state omits most of the board, its records and the symbols it
reads can combine in ways that the board never produces, and the
outputs of all these combinations must be included as well. The
twelve-symbol windows serve two purposes: They retain the
column bits that the attack tests need, and they restrict which
symbols can be read together. With windows of length $4$ through
$11$, the construction reaches a successor violating
Condition~\ref{cond:state} (Appendix~\ref{app:construction}).

\subsection{Checking every branch}\label{sec:closure-check}

This subsection supplies what Lemma~\ref{lem:local-exactness} leaves open:
that the output check passes and that the successor satisfies
Condition~\ref{cond:state}. The actual branch cannot be identified in
advance, so the check covers every branch of every state reached.
We now hold the completed history graph fixed and check it with
Algorithm~\ref{alg:check}, adding no edges. The states reached and
their successors form the state graph (Definition~\ref{def:state-graph}).

Before exploring, the verifiers check the starting board as in
Section~\ref{app:seed}: $\sigma_1\dots\sigma_{29}$ follows the graph,
every lower queen before column $30$ satisfies $|d_j-j|\le4$, and the
state before column $30$ satisfies Condition~\ref{cond:state}.

\begin{algorithm}[H]
\caption{Constructing the state graph}
\label{alg:check}
\renewcommand{\algorithmicensure}{\textbf{Output:}}
\begin{algorithmic}[1]
\Ensure The state graph, or failure.
\State $\mathcal S\gets\{$the state before column $30\}$, unexamined; $\mathcal E\gets\varnothing$.
\While{some state $S\in\mathcal S$ is unexamined}
  \State Mark $S$ as examined.
  \State Run Algorithm~\ref{alg:local-successors} on every branch from $S$.\Comment{fails if an output check fails}
  \For{each successor $S'$ found}
    \If{$S'$ violates Condition~\ref{cond:state}}
      \State \Return failure
    \EndIf
    \State Add $S'$ to $\mathcal S$ as unexamined, unless it is already in $\mathcal S$.
    \State Add the edge $S\to S'$ to $\mathcal E$.
  \EndFor
\EndWhile
\State \Return $(\mathcal S,\mathcal E)$
\end{algorithmic}
\end{algorithm}

A state already examined need not be examined again, since its
branches depend only on its eight fields and the history graph. The
check does not limit the number of states; it ends only when no 
unexamined state remains. If an output check fails or a successor 
violates Condition~\ref{cond:state}, the whole check fails with an 
error; no branch is ever dropped to make it pass.

\begin{proposition}[Finite verification]\label{lem:computation}
The history graph passes the checks above. From the initial state,
the calculation reaches $7014$ states and $8327$ directed 
state-graph edges, and no check fails. In particular, the initial 
state satisfies Condition~\ref{cond:state}, and for every reached 
state, every branch of the calculation passes the output check and 
gives a successor satisfying Condition~\ref{cond:state}.
\end{proposition}
\begin{proof}
Algorithm~\ref{alg:check} terminates with no unexamined state and no
failed check. Appendix~\ref{app:artifacts} gives the code, with an
independent second implementation that reaches the same states.
\end{proof}

The number $8327$ counts ordered pairs of states, not branches:
Several branches can give the same successor, and the row search can
branch again after the queen has been chosen.

Proposition~\ref{lem:computation} concerns the calculation on
states. It remains to show that the actual board always
stays among these states, to which we now turn.

\subsection{Following the actual queens forever}\label{sec:identification}

We can now prove Lemma~\ref{lem:estimates}, and with it
Theorem~\ref{thm:main}. When the $j$th
lower queen is placed in column $n$, equation~\eqref{eq:local-rank}
gives $d_j-j=w-r-|D|$, with $w$, $D$ and the offset $r$ taken from the
records before column $n$. So it suffices to show that the actual
records satisfy Condition~\ref{cond:state} before every column
$n\ge30$; the lower queens in earlier columns were checked directly
(Section~\ref{app:seed}). Lemma~\ref{lem:local-exactness} shows that,
as long as they do, the calculation carries out the actual step, but
not that the next records satisfy the condition again.
Proposition~\ref{lem:computation} supplies this for every state
reachable from the start. The proof is therefore an induction along
the actual board, which keeps the actual state among these reached
states.

The induction uses two facts. If the actual word $\sigma_1\dots\sigma_{n-1}$ follows the
history graph, then every request can be answered by the actual
symbol, since that symbol labels an edge leaving the current window;
so the actual branch is one of the branches the calculation explores.
If moreover the state is one of the states reached in
Proposition~\ref{lem:computation}, the actual branch passes the output
check and leads to another reached state, so both facts hold again
before the next column.

\begin{proof}[Proof of Lemma~\ref{lem:estimates}]
We argue by induction on the column $n\ge30$. Let $S_{30}$ be the
state before column $30$ from Section~\ref{app:seed}, and for
$n\ge30$ let $S_{n+1}$ be the successor of $S_n$ given by the actual
branch. The induction hypothesis for $n$ consists of three assertions:
\begin{enumerate}
\item $S_n$ is a reached state of the checked state graph; in
particular, it satisfies Condition~\ref{cond:state}.
\item $S_n$ describes the actual board before column $n$: Its fields
$w$, $z$, $R$, $D$, $A$ are the actual records, and its words are the actual
segments of the queen word $\sigma$ at their stated positions, all
with indices less than $n$. (The length of $Q$ depends on the requests
made so far.)
\item The actual word $\sigma_1\dots\sigma_{n-1}$ follows the history
graph.
\end{enumerate}
Assertions~(1) and~(2) together say that the actual records satisfy
Condition~\ref{cond:state}, which is what we need; assertions~(2)
and~(3) are the hypotheses that let Lemma~\ref{lem:local-exactness}
act on the actual board.

\emph{Base case.} For $n=30$, assertion~(1) holds because $S_{30}$ is
the initial state of the check, and assertions~(2) and~(3) are the
starting checks of Section~\ref{app:seed}.

\emph{Inductive step.} Assume the three assertions for some $n\ge30$.
We first check the hypotheses of Lemma~\ref{lem:local-exactness}.
Condition~\ref{cond:state} holds by~(1). The least unused row $m$
never decreases, and $U(18)=12$ before column $30$, so
$U(m-1)\ge12$; and $m+|Q|\le n$, because the queue occupies the indices
$m$, \dots, $m+|Q|-1$, which are less than $n$ by~(2). The word hypothesis
is~(3). By the lemma, the actual branch is a branch of the
calculation; it places the
actual queen in column $n$, produces~$\sigma_n$, and gives as $S_{n+1}$
the actual records before column $n+1$, with every stored symbol at
an index less than $n+1$; this is~(2) for $n+1$. Since $S_n$ is a reached
state, Proposition~\ref{lem:computation} applies to this branch. Its
symbol $\sigma_n$ passes the output check, that is, it is the label 
of an edge
leaving $H_{\mathrm{out}}$, so the actual word follows the graph
through $\sigma_n$; this is~(3) for $n+1$. Its successor $S_{n+1}$ is
again a reached state; this is~(1) for $n+1$. This completes the
induction.

\emph{The discrepancy.} Let the $j$th lower queen lie in column
$n\ge30$, at offset $r$. By~(2), the records of $S_n$ are the actual
records before column $n$, so \eqref{eq:local-rank} gives
\[
 d_j-j=w-r-|D|.
\]
By~(1), Condition~\ref{cond:state} gives $0\le w-r\le4$ and
$0\le|D|\le4$, so $|d_j-j|\le4$. The lower queens in columns less than
$30$ satisfy the same bound by the starting checks of
Section~\ref{app:seed}. Hence the bound holds for every lower queen.
\end{proof}

\begin{corollary}[Diagonal coverage]\label{cor:diagonal-coverage}
Every diagonal $y-x=h$, $h\in\mathbb Z$, contains exactly one
queen. Equivalently, $n\mapsto q_n-n$ is a bijection from $\N$
to $\mathbb Z$. In particular, the lower-diagonal magnitudes
$(d_j)_{j\ge1}$ form a permutation of the positive integers.
\end{corollary}
\begin{proof}
Immediately before the $j$th lower queen is placed in a column
$n\ge30$, let $d$ be the least unused lower-diagonal magnitude; the
record $D$ holds the offsets from $d$ of the used magnitudes above it.
The induction gives $D\subseteq[1\dts4]$, and
\eqref{eq:local-rank} gives
\[
 d=j-|D|\ge j-4.
\]
There are infinitely many lower queens, as shown after
\eqref{eq:counts}. The least unused lower-diagonal magnitude
therefore tends to infinity, so every positive magnitude is
eventually used.

Proposition~\ref{prop:stability}, with $C=4$, gives
$U(n)\to\infty$. Lemma~\ref{lem:upper} then puts a queen on
every positive upper diagonal. The origin occupies the main
diagonal, and the nonattacking condition gives uniqueness on
each diagonal.
\end{proof}

\begin{remark}[Zero values on diagonals]
In game terms, every diagonal $y-x=h$ contains exactly one
zero of the Sprague--Grundy function $G$ of the game remark in the
introduction. This proves the zero-value case of
\cite[Conjecture~24]{DSS}, which asks whether every such diagonal
contains every nonnegative Sprague--Grundy value exactly once.
\end{remark}

\subsection{Consequences for related OEIS sequences}\label{sec:oeis-consequences}

Several OEIS entries~\cite{OEIS} record statistics of the same board
and ask questions about them. Some follow directly from the results
above. Proposition~\ref{prop:stability} proves the upper-queen slope
limit stated in \href{https://oeis.org/A275884}{A275884}. In terms of
our counts, \href{https://oeis.org/A275890}{A275890},
\href{https://oeis.org/A275891}{A275891}, and
\href{https://oeis.org/A275892}{A275892} are $1+U(N-1)$, $L(N-1)$, and
$1+U(N-1)-L(N-1)$, so the same proposition gives $N/\phi+O(1)$,
$N/\phi^2+O(1)$, and $N/\phi^3+O(1)$. By
Corollary~\ref{cor:diagonal-coverage}, the signed diagonal sequences
\href{https://oeis.org/A065185}{A065185} and
\href{https://oeis.org/A276325}{A276325} each enumerate $\mathbb Z$
exactly once.

Other entries record the pattern of upper and lower columns: the
lengths of runs of upper or lower columns, and the gaps
between them. Since the queen word follows the history graph
(Section~\ref{sec:identification}), a pattern of column bits that no
path in the graph contains never occurs on the board. This bounds each
statistic in every column after the start, which is checked directly;
for one of them we use a refined graph with longer histories.
Actually computing the greedy sequence then finds occurrences of 
every value that the bounds allow.

\Needspace{20\baselineskip}
\begin{corollary}[Column runs and gaps]\label{cor:oeis-runs}
The following are the exact sets of values of the indicated sequences.
Runs are maximal; the upper-column sequences include the origin, as in
the OEIS definitions.
\begin{center}
\begin{tabular}{lll}
\toprule
Sequence&Statistic&Set of values\\
\midrule
\href{https://oeis.org/A275885}{A275885}&Lower-column run lengths&$\{1,2,3\}$\\
\href{https://oeis.org/A275886}{A275886}&Upper-column run lengths&$\{1,\ldots,5\}$\\
\href{https://oeis.org/A275887}{A275887}&Run lengths of equal terms in \href{https://oeis.org/A275885}{A275885}&$\{1,\ldots,9,11\}$\\
\href{https://oeis.org/A275888}{A275888}&Gaps between upper columns&$\{1,2,3,4\}$\\
\href{https://oeis.org/A275889}{A275889}&Gaps between lower columns&$\{1,\ldots,6\}$\\
\bottomrule
\end{tabular}
\end{center}
In particular, a $4$ never occurs in \href{https://oeis.org/A275885}{A275885}, and a $10$ never occurs
in \href{https://oeis.org/A275887}{A275887}.
\end{corollary}
\begin{proof}
We first bound the runs and gaps. By the induction of
Section~\ref{sec:identification}, every twelve consecutive symbols of
$\sigma$ form a vertex of the history graph, and inspection of the
graph shows that no vertex contains four consecutive symbols with
column bit $0$ or six with column bit~$1$. Since column bit $0$ marks
a lower column, there are never four consecutive lower columns or six
consecutive upper columns among columns $1$, $2$, \dots; a direct check
covers the origin, which the upper-column sequences count as upper.
So lower runs have length at most three and upper runs at most five.
A gap between consecutive upper columns is one more than the lower run
between them, and similarly for lower columns, so the gaps between
upper columns are at most four and those between lower columns at
most six.

The runs of equal terms in \href{https://oeis.org/A275885}{A275885}
need more. Write $r$ for \href{https://oeis.org/A275885}{A275885} and
$g$ for \href{https://oeis.org/A275888}{A275888}; the terms of $r$ are
the terms $k>1$ of $g$, each reduced by one. The twelve-symbol graph
is too coarse to exclude ten consecutive equal terms in $r$, so
Appendix~\ref{app:oeis-checks} repeats the verification with
forty-symbol histories. From the refined state graph it derives a
finite graph with edges labeled $1$, $2$, $3$ such that, after column $80$,
the terms of $r$ are the labels along a walk. For each~$c$, the edges
labeled $c$ form an acyclic subgraph, so a maximal run of $c$'s is one
of finitely many paths in it. Let $\mathcal L_c$ be the set of lengths
of those paths that can be preceded and followed by edges with other
labels; enumerating them gives
\[
 \mathcal L_1=\{1,\ldots,9,11\},\qquad
 \mathcal L_2=\{1,\ldots,6\},\qquad
 \mathcal L_3=\{1\}.
\]
Every maximal run of equal terms in $r$ after column $80$ therefore has
length in $\{1,\ldots,9,11\}$, and the earlier runs are checked
directly.

Finally, computing the greedy sequence exhibits every value in the
table (Appendix~\ref{app:oeis-checks}), so each set of values is exact.
\end{proof}

These conclusions settle the explicit bound questions in
\href{https://oeis.org/A275885}{A275885} and
\href{https://oeis.org/A275888}{A275888}, and establish the empirical
range recorded in \href{https://oeis.org/A275887}{A275887}.

The same refined graph settles a catalogue of observations about
$g$. Cut $g$ immediately after each $3$; the pieces are the
\emph{return words} to $3$. They are exactly the $156$ words listed in
the note linked from \href{https://oeis.org/A275888}{A275888}
\cite{OEISWords}: The graph admits no others, and each occurs. Their
lengths range from $2$ to $26$, and exactly five contain a $4$.
Exactly $63$ of them are \emph{faithful}, meaning that every
occurrence is followed by the same return word: For these $63$ the
graph allows only one successor, and for each of the others two
different successors occur. The factors $11111$, $2222$, $33$, $44$,
and $3213$ never occur in $g$. Consecutive $4$'s are at least $71$
terms apart, and every pair at that distance has the same fill,
listed in the note. Appendix~\ref{app:oeis-checks} describes the
checks and occurrence witnesses; the complete catalogues are
in the companion repository.

\subsection{Sharper constants and Knuth's ranges}\label{sec:sharper}
Theorem~\ref{thm:main} used the bound $|d_j-j|\le4$ of 
Lemma~\ref{lem:estimates}, fed into the counting argument 
of Section~\ref{sec:contraction}. In fact the verification
proves the sharper, asymmetric bound
\[
 -4\le d_j-j\le2\qquad(j\ge1).
\]
The lower end follows from $|d_j-j|\le4$. For the upper end, every
lower choice $(n,m+r)$ of every state of the state graph has $w-r-|D|\le2$
(Appendix~\ref{app:verification}), so \eqref{eq:local-rank} and the
induction of Section~\ref{sec:identification} give $d_j-j\le2$ from
column $30$ on, and Table~\ref{tab:lower-values} shows
$d_j-j\in\{-1,1\}$ before that. Both ends occur: The $50$th lower queen
has $d_j-j=2$, and the $170$th has $d_j-j=-4$.

This subsection proves Theorem~\ref{thm:knuth} by a second
contraction, different from the one in
Section~\ref{sec:contraction}. It compares the error at column $n$ with
the error at the least unused row $m$, rather than at $L(n)$. As in the
proof of Proposition~\ref{prop:stability}, write
$\varepsilon(x)=U(x)-x/\phi$.

\begin{lemma}[One-step error identity]\label{lem:error-identity}
Consider the actual board before column $n\ge1$, with least unused row
$m$ and records $w$, $z$, and $D$. Then
\begin{equation}\label{eq:error-identity}
 \varepsilon(n-1)=Y-\frac{\varepsilon(m-1)}{\phi},\qquad
 Y\coloneqq w-|D|+1-\frac{z}{\phi}.
\end{equation}
\end{lemma}

\begin{proof}
The columns $1$, \dots, $n-1$ contain $L(n-1)$ lower queens. They use
every lower-diagonal magnitude less than $d$ (the least unused one)
and the $|D|$ magnitudes recorded in $D$, so $L(n-1)=d-1+|D|$. Since $w=n-m-d$, it follows that
$U(n-1)=n-1-L(n-1)=m+w-|D|$. By the definition of $z$, also
$n=m+U(m-1)+z$. Substituting this and
$U(m-1)=(m-1)/\phi+\varepsilon(m-1)$,
\[
 \begin{aligned}
 \varepsilon(n-1)&=m+w-|D|-\frac{n-1}{\phi}\\
 &=w-|D|-\frac{z}{\phi}-\frac{\varepsilon(m-1)}{\phi}
 +m\Bigl(1-\frac1\phi-\frac1{\phi^2}\Bigr)+\frac1\phi+\frac1{\phi^2}.
 \end{aligned}
\]
Since $1/\phi+1/\phi^2=1$, this is \eqref{eq:error-identity}.
\end{proof}

The identity relates column $n$ to column $m\approx n/\phi$,
and the factor $-1/\phi$ makes it a contraction whose sign alternates:
A~lower bound on $\varepsilon(n-1)$
needs an upper bound on $\varepsilon(m-1)$. The term
$Y$ is a function of the state; write $Y(S)$ for its value at a state
$S$. Bounding $Y(S)$ and $\varepsilon(m-1)$ separately, as
Section~\ref{sec:contraction} bounds its error terms, would lose what
makes the identity useful. The two are correlated through their
histories, as follows.

\medskip\noindent
\textbf{The link between the two columns.}
For $n\ge30$, let $S_n$ be the state before column $n$ given by the
induction of Section~\ref{sec:identification}, which describes the
actual board. By that induction, $S_n$ is a state of the state graph, and its input history is
$H_{\mathrm{in}}(S_n)=\sigma_{m-12}\dots\sigma_{m-1}$. If also
$m\ge30$, the state $S_m$ before column $m$ has output history
$H_{\mathrm{out}}(S_m)=\sigma_{m-12}\dots\sigma_{m-1}$. Hence
\begin{equation}\label{eq:window-link}
 H_{\mathrm{in}}(S_n)=H_{\mathrm{out}}(S_m).
\end{equation}

\medskip\noindent
\textbf{Unrolling the identity.}
The identity \eqref{eq:error-identity} can be applied again at column
$m$, then at the least unused row before column $m$, and so on. Write
$n_0=n$, and let $n_{k+1}$ be the least unused row before column $n_k$,
so that $n_1=m$ and, by \eqref{eq:causality}, $n_0>n_1>n_2>\dotsb$. As
long as these columns are at least $30$, applying the identity $t$ times
gives
\begin{equation}\label{eq:unrolled}
 \varepsilon(n-1)=\sum_{k=0}^{t-1}\Bigl(-\frac1\phi\Bigr)^kY(S_{n_k})
 +\Bigl(-\frac1\phi\Bigr)^t\varepsilon(n_t-1),
\end{equation}
that is, $Y(S_{n_0})-Y(S_{n_1})/\phi+Y(S_{n_2})/\phi^2-\dotsb$ plus a
remainder, and by \eqref{eq:window-link} each state $S_{n_k}$ has as
input history the output history of the next state $S_{n_{k+1}}$.
Since $\varepsilon$ is bounded (Proposition~\ref{prop:stability}), the
remainder shrinks geometrically: The error is essentially an
alternating sum of the values of $Y$ along a chain of linked states,
dominated by its first few terms. The twelve symbols
$W=H_{\mathrm{out}}(S_n)$ before column $n$ restrict the first state of
the chain to the states with output history $W$; the input history of
that state restricts the second state, and so on. 

We therefore attach bounds on the error to the vertices of the history
graph, the twelve-symbol windows: We look for numbers
$\mathrm{lo}(W)\le\mathrm{hi}(W)$, one pair for each vertex $W$, such
that $\mathrm{lo}(W)\le\varepsilon(n-1)\le\mathrm{hi}(W)$ whenever the
twelve symbols before column $n$ are $W$. (These symbols form a vertex,
since the queen word follows the graph; see
Section~\ref{sec:identification}.) To prove such bounds, we need not
follow these chains to the end: It is enough to pass them from the
window before $m$ to the window before $n$, by induction on $n$. If such bounds hold before column $m$, then
$\mathrm{lo}(V)\le\varepsilon(m-1)\le\mathrm{hi}(V)$ with
$V=H_{\mathrm{in}}(S_n)$, and \eqref{eq:error-identity} gives
$Y(S_n)-\mathrm{hi}(V)/\phi\le\varepsilon(n-1)\le
Y(S_n)-\mathrm{lo}(V)/\phi$. For the bounds to hold before column $n$
as well, it suffices that this lower bound is at least
$\mathrm{lo}(W)$ and this upper bound at most $\mathrm{hi}(W)$, where
$W=H_{\mathrm{out}}(S_n)$. The following definition asks for this at
every state of the state graph; each requirement involves a single
state, and both histories of every state are vertices
(Section~\ref{app:graph}), so it makes sense.

\Needspace{9\baselineskip}
\begin{definition}[Consistent window bounds]\label{def:window-bounds}
Real-valued functions $\mathrm{lo}$ and $\mathrm{hi}$ on the vertices of
the history graph are \emph{consistent} if every state $S$ of the state graph
satisfies, with $V=H_{\mathrm{in}}(S)$ and $W=H_{\mathrm{out}}(S)$,
\begin{equation}\label{eq:window-consistency}
 \mathrm{lo}(W)\le Y(S)-\frac{\mathrm{hi}(V)}{\phi},\qquad
 Y(S)-\frac{\mathrm{lo}(V)}{\phi}\le\mathrm{hi}(W).
\end{equation}
\end{definition}

\Needspace{7\baselineskip}
\begin{lemma}[Window bounds hold on the board]\label{lem:window-bounds}
Let $\mathrm{lo}$ and $\mathrm{hi}$ be consistent, and suppose that
$\mathrm{lo}(W)\le\varepsilon(n-1)\le\mathrm{hi}(W)$ with
$W=\sigma_{n-12}\dots\sigma_{n-1}$ for $30\le n\le48$. Then the same
holds for every $n\ge30$.
\end{lemma}

\begin{proof}
We use strong induction on $n$. Before column $49$ the least unused row
is at least $30$, as a direct computation shows, and it never
decreases; so for $n\ge49$ we have $m\ge30$, and $m<n$ by
\eqref{eq:causality}. The induction
hypothesis at $m$ gives
$\mathrm{lo}(V)\le\varepsilon(m-1)\le\mathrm{hi}(V)$ with
$V=H_{\mathrm{out}}(S_m)=H_{\mathrm{in}}(S_n)$, by
\eqref{eq:window-link}. With $W=H_{\mathrm{out}}(S_n)$, the identity
\eqref{eq:error-identity} and the consistency of $\mathrm{lo}$ and $\mathrm{hi}$
at the state $S_n$ give
\[
 \mathrm{lo}(W)\le Y(S_n)-\frac{\mathrm{hi}(V)}{\phi}
 \le\varepsilon(n-1)\le Y(S_n)-\frac{\mathrm{lo}(V)}{\phi}\le\mathrm{hi}(W).
 \qedhere
\]
\end{proof}

Consistency thus acts as an inductive invariant: It is a finite list of
conditions, one for each state of the state graph, and together with
the start it implies the bounds before every column.

\medskip\noindent
\textbf{Computing the bounds.}
The unrolled identity \eqref{eq:unrolled} suggests which bounds to
take. A~\emph{chain} is a sequence $S_0$, $S_1$, \dots{} of states of the
state graph with $H_{\mathrm{in}}(S_k)=H_{\mathrm{out}}(S_{k+1})$ for
each $k$; it \emph{starts at} $H_{\mathrm{out}}(S_0)$. Every vertex of
the history graph is the output history of some state, as the
verification checks. Since every input history is a vertex, chains
start at every vertex, and every chain can be continued, one state at a
time. For a vertex $W$ and $t\ge0$, let $\mathrm{lo}_t(W)$ and
$\mathrm{hi}_t(W)$ be the least and greatest values of
\[
 \sum_{k=0}^{t-1}\Bigl(-\frac1\phi\Bigr)^kY(S_k)
\]
over the chains $S_0$, \dots, $S_{t-1}$ of $t$ states that start at $W$.
These are the sums of \eqref{eq:unrolled} without the remainder, over
every chain that the state graph allows; in particular
$\mathrm{lo}_0=\mathrm{hi}_0=0$. Splitting off the first state
$S=S_0$ writes the sum for a chain of $t+1$ states as
\[
 \sum_{k=0}^{t}\Bigl(-\frac1\phi\Bigr)^kY(S_k)
 =Y(S)-\frac1\phi\sum_{k=0}^{t-1}\Bigl(-\frac1\phi\Bigr)^kY(S_{k+1}),
\]
where the last sum is the sum for the rest of the chain, which starts
at $H_{\mathrm{in}}(S)$. Since the rest enters with a minus sign, the
least sum takes the greatest rest, and the greatest sum the least rest.
This gives the recursion
\begin{equation}\label{eq:rounds}
 \begin{aligned}
 \mathrm{lo}_{t+1}(W)&=\min_S\Bigl(Y(S)-\frac{\mathrm{hi}_t(H_{\mathrm{in}}(S))}{\phi}\Bigr),\\
 \mathrm{hi}_{t+1}(W)&=\max_S\Bigl(Y(S)-\frac{\mathrm{lo}_t(H_{\mathrm{in}}(S))}{\phi}\Bigr),
 \end{aligned}
\end{equation}
where $S$ ranges over the states of the state graph with
$H_{\mathrm{out}}(S)=W$, a nonempty set.

Every chain of $t$ states extends to one of $t+1$ states, and every
chain of $t+1$ states shortens to one of $t$ states; the two sums differ
by the extra term $(-1/\phi)^tY(S_t)$, which has absolute value at most $M/\phi^t$, where
$M$ is the largest value of $|Y(S)|$ over the states $S$ of the state
graph. Hence
$|\mathrm{lo}_{t+1}(W)-\mathrm{lo}_t(W)|\le M/\phi^t$, and likewise for
$\mathrm{hi}$, so $\mathrm{lo}_t$ and $\mathrm{hi}_t$ converge to
limits $\mathrm{lo}$ and $\mathrm{hi}$. Letting $t\to\infty$ in
\eqref{eq:rounds} gives
\[
 \mathrm{lo}(W)=\min_S\Bigl(Y(S)-\frac{\mathrm{hi}(H_{\mathrm{in}}(S))}{\phi}\Bigr),\qquad
 \mathrm{hi}(W)=\max_S\Bigl(Y(S)-\frac{\mathrm{lo}(H_{\mathrm{in}}(S))}{\phi}\Bigr).
\]
The limits $\mathrm{lo}$ and $\mathrm{hi}$ are consistent in the sense
of Definition~\ref{def:window-bounds}. Indeed, for a state $S$ with
$V=H_{\mathrm{in}}(S)$ and $W=H_{\mathrm{out}}(S)$, the number
$Y(S)-\mathrm{hi}(V)/\phi$ is one of the terms whose minimum is
$\mathrm{lo}(W)$, so $\mathrm{lo}(W)\le Y(S)-\mathrm{hi}(V)/\phi$;
likewise $Y(S)-\mathrm{lo}(V)/\phi\le\mathrm{hi}(W)$. 

\medskip\noindent
\textbf{Exact values.}
The proof needs $\mathrm{lo}$ and $\mathrm{hi}$ exactly, not only to
floating-point precision: Some of the bounds in
Proposition~\ref{prop:window-check} below hold with equality, and the
proof of Theorem~\ref{thm:knuth} uses the exact constants to show that
its inequalities are strict. The rounds \eqref{eq:rounds}, computed in
floating point, serve only to find which state attains each minimum and
maximum in the limit. Once these states are fixed, each value is given
by a single state $S$ with $H_{\mathrm{out}}(S)=W$:
\[
 \mathrm{lo}(W)=Y(S)-\frac{\mathrm{hi}(V)}{\phi}
 \quad\text{or}\quad
 \mathrm{hi}(W)=Y(S)-\frac{\mathrm{lo}(V)}{\phi},
 \qquad V=H_{\mathrm{in}}(S).
\]
So each value is determined by one other value, at the vertex $V$.
There are finitely many values, so if we pass repeatedly from a value
to the one that determines it, we eventually return to a value already
met and run around a cycle.
Going once around a cycle of $k$ values, through states
$S_0$, \dots, $S_{k-1}$, expresses its first value $x$ in terms of itself:
\[
 x=a+\Bigl(-\frac1\phi\Bigr)^kx,\qquad
 a=\sum_{i=0}^{k-1}\Bigl(-\frac1\phi\Bigr)^iY(S_i),
 \qquad\text{so}\qquad
 x=\frac{a}{1-(-1/\phi)^k}.
\]
Each $Y(S)$ lies in $\mathbb Q(\sqrt5)$, and so does $x$. Every other
value is $Y(S)$ minus $1/\phi$ times a value found before it, so all
the values lie in $\mathbb Q(\sqrt5)$.

The next proposition records what the computation finds. Property (i)
is what the proof of Theorem~\ref{thm:knuth} needs for upper queens,
where $q_n-n\phi=\varepsilon(n)$, and property (ii) is what it needs
for lower queens at offset $r$.

\begin{proposition}[Finite verification of the window bounds]\label{prop:window-check}
There are consistent functions $\mathrm{lo}$ and $\mathrm{hi}$ with values in
$\mathbb Q(\sqrt5)$, satisfying the hypothesis of
Lemma~\ref{lem:window-bounds}, with the following properties.
\begin{enumerate}
\renewcommand{\labelenumi}{(\roman{enumi})}
\item Every vertex $W$ whose last symbol has column bit $1$ satisfies
\[
 \frac{19\sqrt5-49}{8}\le\mathrm{lo}(W)
 \quad\text{and}\quad
 \mathrm{hi}(W)\le\sqrt5-1.
\]
\item Let $S$ be a state of the state graph, with record $z$, and let
$V=H_{\mathrm{in}}(S)$. If a branch of the calculation on $S$ chooses
a lower queen $(n,m+r)$ at offset $r$, then
\[
 15-8\sqrt5\le r+\frac1{\phi^2}-\frac{z+\mathrm{hi}(V)}{\phi}
 \quad\text{and}\quad
 r+\frac1{\phi^2}-\frac{z+\mathrm{lo}(V)}{\phi}\le13-4\sqrt5.
\]
\end{enumerate}
\end{proposition}
\begin{proof}
The exploration of Section~\ref{sec:closure-check} gives the $7014$
states of the state graph, with $Y(S)$, $H_{\mathrm{in}}(S)$,
$H_{\mathrm{out}}(S)$, and the offsets chosen by their branches. From
these, the computation just described produces $\mathrm{lo}$ and
$\mathrm{hi}$ with values in $\mathbb Q(\sqrt5)$. A computer then
checks, in exact arithmetic:
\begin{enumerate}
\renewcommand{\labelenumi}{(\alph{enumi})}
\item every vertex is the output history of some state, so the
computation above is well defined;
\item \eqref{eq:window-consistency} holds at every state, so
$\mathrm{lo}$ and $\mathrm{hi}$ are consistent;
\item $\mathrm{lo}(W)\le\varepsilon(n-1)\le\mathrm{hi}(W)$ with
$W=\sigma_{n-12}\dots\sigma_{n-1}$ for $30\le n\le48$, which is the
hypothesis of Lemma~\ref{lem:window-bounds};
\item property (i) at every vertex whose last symbol has column bit $1$,
and property (ii) at every lower choice of every state.
\end{enumerate}
Appendix~\ref{app:verification} gives the code, with an independent
second implementation that finds the same constants.
\end{proof}

\begin{figure}[!ht]
\centering
\begin{tikzpicture}[font=\footnotesize,
  tick/.style={font=\scriptsize,text=black!65,inner sep=2pt}]
  \fill[black!9] (0,-2.416) rectangle (12.2,-0.184);
  \node[inner sep=0pt,anchor=south west] at (0,-2.600) {\includegraphics[width=12.2cm,height=2.6cm]{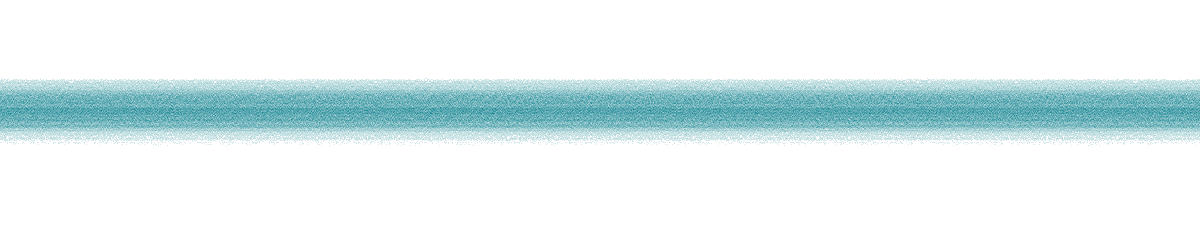}};
  \draw[black,line width=0.45pt] (0,-1.594) -- (13.85,-1.594);
  \draw[black,line width=0.45pt] (0,-0.854) -- (13.85,-0.854);
  \draw[black!55,line width=0.45pt,dash pattern=on 2.5pt off 1.8pt] (0,-1.799) -- (13.85,-1.799);
  \draw[black!55,line width=0.45pt,dash pattern=on 2.5pt off 1.8pt] (0,-0.716) -- (13.85,-0.716);
  \draw[black!40,line width=0.3pt,dash pattern=on 0.6pt off 1.4pt] (0,-1.300) -- (12.2,-1.300);
  \fill[queensUpper!75] (12.350,-1.594) rectangle (12.354,-1.559) (12.350,-1.559) rectangle (12.383,-1.524) (12.350,-1.524) rectangle (12.521,-1.488) (12.350,-1.488) rectangle (12.561,-1.453) (12.350,-1.453) rectangle (12.709,-1.418) (12.350,-1.418) rectangle (13.049,-1.383) (12.350,-1.383) rectangle (13.428,-1.347) (12.350,-1.347) rectangle (13.433,-1.312) (12.350,-1.312) rectangle (13.676,-1.277) (12.350,-1.277) rectangle (13.679,-1.241) (12.350,-1.241) rectangle (13.839,-1.206) (12.350,-1.206) rectangle (13.850,-1.171) (12.350,-1.171) rectangle (13.594,-1.136) (12.350,-1.136) rectangle (13.632,-1.100) (12.350,-1.100) rectangle (13.565,-1.065) (12.350,-1.065) rectangle (13.522,-1.030) (12.350,-1.030) rectangle (13.309,-0.995) (12.350,-0.995) rectangle (13.017,-0.959) (12.350,-0.959) rectangle (12.804,-0.924) (12.350,-0.924) rectangle (12.691,-0.889) (12.350,-0.889) rectangle (12.507,-0.854);
  \draw[black,line width=0.4pt] (0,0.000) -- (0,-2.600) -- (12.2,-2.600);
  \draw[black,line width=0.4pt] (0,-2.383) -- (-0.08,-2.383) node[tick,anchor=east] {$-3$};
  \draw[black,line width=0.4pt] (0,-2.022) -- (-0.08,-2.022) node[tick,anchor=east] {$-2$};
  \draw[black,line width=0.4pt] (0,-1.661) -- (-0.08,-1.661) node[tick,anchor=east] {$-1$};
  \draw[black,line width=0.4pt] (0,-1.300) -- (-0.08,-1.300) node[tick,anchor=east] {$0$};
  \draw[black,line width=0.4pt] (0,-0.939) -- (-0.08,-0.939) node[tick,anchor=east] {$1$};
  \draw[black,line width=0.4pt] (0,-0.578) -- (-0.08,-0.578) node[tick,anchor=east] {$2$};
  \draw[black,line width=0.4pt] (0,-0.217) -- (-0.08,-0.217) node[tick,anchor=east] {$3$};
  \node[anchor=south west,inner sep=1pt] at (0,0.050) {\textcolor{queensUpper}{Upper queens:} $q_n-n\phi$};
  \fill[black!9] (0,-6.013) rectangle (12.2,-3.587);
  \node[inner sep=0pt,anchor=south west] at (0,-6.100) {\includegraphics[width=12.2cm,height=2.6cm]{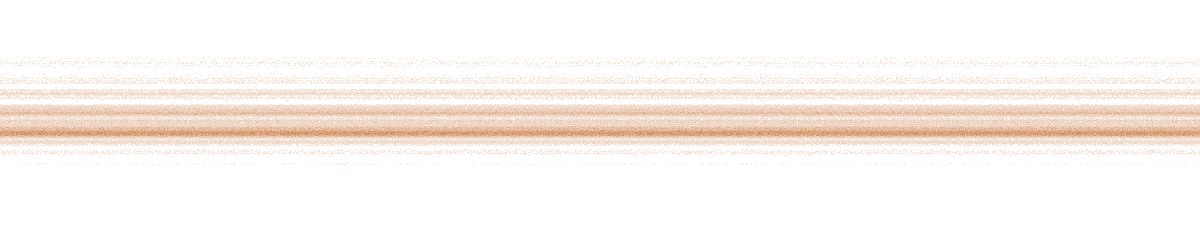}};
  \draw[black,line width=0.45pt] (0,-5.294) -- (13.85,-5.294);
  \draw[black,line width=0.45pt] (0,-4.106) -- (13.85,-4.106);
  \draw[black!55,line width=0.45pt,dash pattern=on 2.5pt off 1.8pt] (0,-5.378) -- (13.85,-5.378);
  \draw[black!55,line width=0.45pt,dash pattern=on 2.5pt off 1.8pt] (0,-4.010) -- (13.85,-4.010);
  \draw[black!40,line width=0.3pt,dash pattern=on 0.6pt off 1.4pt] (0,-4.800) -- (12.2,-4.800);
  \fill[queensLower!75] (12.350,-5.294) rectangle (12.358,-5.258) (12.350,-5.222) rectangle (12.350,-5.186) (12.350,-5.186) rectangle (12.396,-5.150) (12.350,-5.150) rectangle (12.386,-5.114) (12.350,-5.114) rectangle (12.354,-5.078) (12.350,-5.078) rectangle (12.553,-5.042) (12.350,-5.042) rectangle (12.599,-5.006) (12.350,-5.006) rectangle (12.931,-4.970) (12.350,-4.970) rectangle (13.850,-4.934) (12.350,-4.934) rectangle (13.621,-4.898) (12.350,-4.898) rectangle (13.148,-4.862) (12.350,-4.862) rectangle (12.771,-4.826) (12.350,-4.826) rectangle (12.786,-4.790) (12.350,-4.790) rectangle (12.584,-4.754) (12.350,-4.754) rectangle (13.167,-4.718) (12.350,-4.718) rectangle (13.011,-4.682) (12.350,-4.682) rectangle (12.764,-4.646) (12.350,-4.646) rectangle (12.428,-4.610) (12.350,-4.610) rectangle (12.362,-4.574) (12.350,-4.574) rectangle (12.433,-4.538) (12.350,-4.538) rectangle (12.619,-4.502) (12.350,-4.502) rectangle (12.496,-4.466) (12.350,-4.466) rectangle (12.352,-4.430) (12.350,-4.430) rectangle (12.379,-4.394) (12.350,-4.394) rectangle (12.413,-4.358) (12.350,-4.358) rectangle (12.386,-4.322) (12.350,-4.322) rectangle (12.357,-4.286) (12.350,-4.250) rectangle (12.356,-4.214) (12.350,-4.214) rectangle (12.381,-4.178) (12.350,-4.178) rectangle (12.370,-4.142) (12.350,-4.142) rectangle (12.359,-4.106);
  \draw[black,line width=0.4pt] (0,-3.500) -- (0,-6.100) -- (12.2,-6.100);
  \draw[black,line width=0.4pt] (0,-5.826) -- (-0.08,-5.826) node[tick,anchor=east] {$-6$};
  \draw[black,line width=0.4pt] (0,-5.484) -- (-0.08,-5.484) node[tick,anchor=east] {$-4$};
  \draw[black,line width=0.4pt] (0,-5.142) -- (-0.08,-5.142) node[tick,anchor=east] {$-2$};
  \draw[black,line width=0.4pt] (0,-4.800) -- (-0.08,-4.800) node[tick,anchor=east] {$0$};
  \draw[black,line width=0.4pt] (0,-4.458) -- (-0.08,-4.458) node[tick,anchor=east] {$2$};
  \draw[black,line width=0.4pt] (0,-4.116) -- (-0.08,-4.116) node[tick,anchor=east] {$4$};
  \draw[black,line width=0.4pt] (0,-3.774) -- (-0.08,-3.774) node[tick,anchor=east] {$6$};
  \draw[black,line width=0.4pt] (0.000,-6.100) -- (0.000,-6.180) node[tick,anchor=north] {$0$};
  \draw[black,line width=0.4pt] (2.440,-6.100) -- (2.440,-6.180) node[tick,anchor=north] {$200\,000$};
  \draw[black,line width=0.4pt] (4.880,-6.100) -- (4.880,-6.180) node[tick,anchor=north] {$400\,000$};
  \draw[black,line width=0.4pt] (7.320,-6.100) -- (7.320,-6.180) node[tick,anchor=north] {$600\,000$};
  \draw[black,line width=0.4pt] (9.760,-6.100) -- (9.760,-6.180) node[tick,anchor=north] {$800\,000$};
  \draw[black,line width=0.4pt] (12.200,-6.100) -- (12.200,-6.180) node[tick,anchor=north] {$1\,000\,000$};
  \node[anchor=north] at (6.100,-6.550) {column $n$};
  \node[anchor=south west,inner sep=1pt] at (0,-3.450) {\textcolor{queensLower}{Lower queens:} $q_n-n/\phi$};
\end{tikzpicture}
\caption{The distance from each queen to its line, for $1\le n\le10^6$: the density of
the points $(n,q_n-n\phi)$ for upper queens (top) and $(n,q_n-n/\phi)$ for lower queens
(bottom), with a histogram of the distances on the right. Shaded: the bounds of
Theorem~\ref{thm:main}. Solid: the bounds of Theorem~\ref{thm:knuth}. Dashed:
Knuth's observed ranges, in $0$-indexed coordinates.}
\label{fig:residuals}
\end{figure}

\begin{proof}[Proof of Theorem~\ref{thm:knuth}]
For $1\le n\le48$ the bounds are checked directly. Let $n\ge49$.

If the queen in column $n$ is upper, then $q_n-n\phi=\varepsilon(n)$
by Lemma~\ref{lem:upper}, as in the proof of
Proposition~\ref{prop:stability}. Lemma~\ref{lem:window-bounds}, applied
at column $n+1$, gives $\mathrm{lo}(W)\le\varepsilon(n)\le\mathrm{hi}(W)$
with $W=\sigma_{n-11}\dots\sigma_n$, a vertex whose last symbol has
column bit $u_n=1$; Proposition~\ref{prop:window-check}(i) applies.

If the queen is lower, let $m$ be the least unused row and $z$ the
record of $S_n$. The queen is $(n,m+r)$, where $r$ is the offset
chosen by the actual branch of the calculation on $S_n$, which is one
of its branches by Lemma~\ref{lem:local-exactness}, applied as in the
induction of Section~\ref{sec:identification}. By the definition of $z$,
$n=m+U(m-1)+z$; substituting
$U(m-1)=(m-1)/\phi+\varepsilon(m-1)$ and using $1/\phi+1/\phi^2=1$,
\[
 q_n-\frac{n}{\phi}=m+r-\frac{m+U(m-1)+z}{\phi}
 =r+\frac1{\phi^2}-\frac{z+\varepsilon(m-1)}{\phi}.
\]
As in the proof of Lemma~\ref{lem:window-bounds}, $30\le m<n$, so that
lemma and \eqref{eq:window-link} give
$\mathrm{lo}(V)\le\varepsilon(m-1)\le\mathrm{hi}(V)$ with
$V=H_{\mathrm{in}}(S_n)$; Proposition~\ref{prop:window-check}(ii)
applies.

In both cases Proposition~\ref{prop:window-check} gives the bounds
of Theorem~\ref{thm:knuth}, with $\le$ in place of $<$. The
inequalities are strict. For upper queens both constants have a
positive coefficient of $\sqrt5$, while
$q_n-n\phi=(q_n-\tfrac n2)-\tfrac n2\sqrt5$ has a negative one, so
neither is attained. For lower queens,
$q_n-n/\phi=(q_n+\tfrac n2)-\tfrac n2\sqrt5$ could equal $15-8\sqrt5$ or
$13-4\sqrt5$ only at $n=16$ or $n=8$, neither of which is at least $49$.

In the $1$-indexed coordinates $c=n+1$ and $s(c)=q_n+1$ of the remark
at the end of Section~\ref{sec:contraction},
$s(c)-c\phi=(q_n-n\phi)-1/\phi$ and
$s(c)-c/\phi=(q_n-n/\phi)+1/\phi^2$. Rounding outward,
\[
 s(c)\in
 \Bigl[\frac{c}{\phi}-2.507\dts\frac{c}{\phi}+4.438\Bigr]
 \cup[c\phi-1.433\dts c\phi+0.619]
\]
for every $c\ge2$, which lies inside Knuth's ranges. For $c=1$, the
origin, $s(1)=1$ lies in the first range.
\end{proof}

The constants of Theorem~\ref{thm:knuth} are close to the truth. For
$n<10^{11}$, the least and greatest values of $q_n-n\phi$ over upper
queens are $-0.762211$ and $1.232556$, against the bounds $-0.814339$
and $1.236068$; those of $q_n-n/\phi$ over lower queens are
$-2.887917$ and $3.977939$, against $-2.888544$ and $4.055728$ (all to
six decimals). The upper end for upper queens and the lower end for
lower queens are thus nearly attained, within $0.004$ and $0.001$.
Appendix~\ref{app:stream-generation} describes the scan, which records
these values in Knuth's $1$-indexed coordinates.
Figure~\ref{fig:residuals} compares the bounds of
Theorems~\ref{thm:main} and~\ref{thm:knuth} with the first million
queens.
%
\section{Fast generation with little memory}\label{sec:generation}

The local description also gives a practical way to compute the
queens. Knuth's program \texttt{infty-queens} \cite{KnuthQueens} runs
in $O(N)$ time, testing at most five rows per column ($w\le4$), but
it keeps occupancy arrays that
grow linearly with the number of queens, so it needs $\Omega(N)$ memory:
When bit-packed, they occupy about $6$~GiB for ten billion queens
(Section~\ref{sec:generation-benchmarks}). The program described here
generates $q_0$, \dots, $q_N$ in $O(N)$ time using only $O(\log N)$
memory: It stores neither the
board nor the growing queen word, and for ten billion queens its peak
memory is under $2$~MiB. It makes use of two ideas. First, the
calculation for column $n$ reads the queen word only far behind
column~$n$, so the symbols it needs can be regenerated on demand by a
second copy of the same calculation, which lags behind the first and
computes only when asked (Sections~\ref{sec:generator-idea} and~\ref{sec:generator-records}).
Second, along the actual process the calculation passes through only
finitely many records, so it can be compiled into a fixed table that
places several queens per lookup, with no attack tests at run time
(Section~\ref{sec:block-transducer}).
Section~\ref{sec:generator-cost} adds buffering between the copies and
proves the bounds, and Section~\ref{sec:generation-benchmarks}
compares the C implementation with bit-packed Knuth. The first idea
alone gives these bounds: Algorithm~\ref{alg:next-symbol} of
Section~\ref{sec:generator-records} already runs in $O(N)$ time with
$O(\log N)$ memory, by the argument that proves
Proposition~\ref{prop:stream-generation}, and the table and the buffers
improve only the constant factors.

\subsection{Regenerating the earlier word}\label{sec:generator-idea}

By Lemma~\ref{lem:local-exactness}, the actual branch of the
calculation for column $n$ needs the queen word only at indices up
to $m+6$, where $m$ is the least unused row. On the board,
$n=m+\kappa+z$ with $\kappa=U(m-1)$ and $z$ bounded, and
Proposition~\ref{prop:stability}, with $C=4$ from
Lemma~\ref{lem:estimates}, gives $U(m-1)=m/\phi+O(1)$. Since
$1+1/\phi=\phi$, this means $n=\phi m+O(1)$: The calculation writes
$\sigma_n$ but reads the word only up to about index $n/\phi\approx
0.618\,n$.

The earlier symbols can therefore be regenerated instead of stored.
Start a copy of the calculation in the state before column $30$. 
Its records suffice to produce $\sigma_{30}$, \dots, $\sigma_{47}$. To go
further, it needs $\sigma_{30}$, $\sigma_{31}$, \dots{} in order: symbols it
has already produced but did not store, since storing the word is
exactly what we want to avoid. A second copy, started in the same
state, produces the same word from $\sigma_{30}$ onward, so it can
supply them.
When the second copy needs input in turn, a third copy supplies it,
and so on, forming a chain of copies. For example, the first copy
asks for $\sigma_{30}$ before column $48$; the second copy, having
handed on its first eighteen symbols, needs $\sigma_{30}$ in turn when
the first copy stands before column $76$.
The copies take turns: A copy
computes only when the copy it supplies asks it for a symbol, so the
whole chain runs as a single sequential program. Because its inputs are the
actual symbols, each copy follows the actual branch of
Lemma~\ref{lem:local-exactness}: There is nothing to choose, and the
history graph is not consulted.

When the outermost copy reaches column $N$, the second copy works
near column $N/\phi$, the third near $N/\phi^2$, and so on, down to
a copy that has not yet needed any input. For $N=10^6$, with
symbols passed one at a time, the second, third, and fourth copies
stand before columns $618035$, $381967$, and $236072$, and the chain
has $22$ copies, the last before column $48$
(Figure~\ref{fig:recursive-generation}). The number of copies thus grows logarithmically
with $N$, while their combined work is about
\[
 N+\frac{N}{\phi}+\frac{N}{\phi^2}+\cdots=\phi^2N\approx2.618\,N
\]
queen placements; for $N=10^6$ the copies carried out $2617400$
placements in total. Section~\ref{sec:generator-cost} makes this
estimate precise.

\begin{figure}[!htbp]
\centering
\begin{tikzpicture}[x=1.35cm,y=1cm,>=stealth]
 \draw[->,black!70] (0,0) -- (10.45,0) node[right,font=\footnotesize] {column};
 \foreach \c in {10,6.18035,3.81967,2.36072,1.45904,0.90176,0.55734,
                 0.34448,0.21290,0.13159,0.08137,0.05030,0.03109,0.01923,
                 0.01190,0.00739,0.00457,0.00284,0.00177,0.00111,0.00070,0.00048}
  \draw[black!80,very thin] (\c,-.09) -- (\c,.09);
 \foreach \a/\b in {10/6.18035,6.18035/3.81967,3.81967/2.36072,
                    2.36072/1.45904,1.45904/0.90176,0.90176/0.55734,
                    0.55734/0.34448,0.34448/0.21290}
  \draw[->,queensUpper,semithick] (\a,.12) to[bend right=55] (\b,.12);
 \node[font=\footnotesize,text=queensUpper,above] at (8.09,1.42) {reads};
 \foreach \c/\k/\v in {10/1/{1\,000\,000},6.18035/2/{618\,035},
                       3.81967/3/{381\,967},2.36072/4/{236\,072}}
  \node[below=3pt,align=center,font=\footnotesize] at (\c,0)
   {copy~\k\\$\v$};
 \node[below=3pt,font=\footnotesize] at (0,0) {$0$};
\end{tikzpicture}
\caption{The copies when the first reaches column $N=10^6$, with
symbols passed one at a time. Each tick is a copy; each arc runs from
a copy to where it reads, which is where the next copy writes. The
$22$ ticks crowd to the left, the last before column $48$.}
\label{fig:recursive-generation}
\end{figure}

\subsection{The generator, one symbol at a time}\label{sec:generator-records}

Because a copy receives the actual symbols, it never has to allow for
other possibilities. The history graph and $H_{\mathrm{out}}$ serve
only to constrain the possible inputs and to check the outputs during
verification; with the actual inputs there is nothing to constrain and
nothing to check, so a copy keeps neither. Of the state
$(w,z,R,D,A,H_{\mathrm{in}},Q,H_{\mathrm{out}})$ of
Section~\ref{sec:local}, a copy thus drops the output history, and, as
we show next, it keeps only four column bits of the input history.
This subsection derives these records, \eqref{eq:generator-records},
and then gives the algorithm that runs the chain.

The input history enters the calculation in two places: the counts
$\J(z-1)$ and $\J(z)$ that give the row bit $b_n$ in
\eqref{eq:output}, and the upper-antidiagonal test
\eqref{eq:upper-sum}. For the first, a copy can compute $b_n$
directly. By \eqref{eq:source-offsets} and $n=m+\kappa+z$, an upper
queen in column $m+h$ occupies row $n$ exactly when
\begin{equation}\label{eq:generator-row-bit}
 u_{m+h}=1\quad\hbox{and}\quad h+\Gamma(h)=z.
\end{equation}
Thus $b_n$ is $1$ if this equality holds for some retained upper
column, and $0$ otherwise.

Only four column bits before $m$ are needed. Indeed, for
$h\le-5$, we have $\Gamma(h)\le0$ and hence
\[
 h+\Gamma(h)\le-5<z,\qquad
 2h+\Gamma(h)\le-10<z+r
\]
for every lower candidate $r\ge0$, using $z\ge-4$. Such a column
can satisfy neither \eqref{eq:generator-row-bit} nor the
upper-antidiagonal test \eqref{eq:upper-sum}. The row tests on the
lower candidates read their bits directly from $Q$, so the row bits
of the input history are not needed either.

The records of a copy can therefore be just
\begin{equation}\label{eq:generator-records}
 \bigl(w,z,R,D,A,(u_{m-4},u_{m-3},u_{m-2},u_{m-1}),Q\bigr).
\end{equation}
The field $z$ is updated by \eqref{eq:update}, whose sum uses the
symbols removed from $Q$; append their column bits to the four
stored bits, retaining only the last four. All other updates are
unchanged. In particular, $Q$ still stores both bits of each symbol.
Its length stays at most eleven: It starts at eleven, is extended
only to a required length of at most seven, and then loses its first
$\mu$ symbols. The starting values, before column $30$, are
\[
 \begin{gathered}
 w=0,\quad z=-1,\quad R=A=\varnothing,\quad D=\{1\},\\
 (u_{15},u_{16},u_{17},u_{18})=(1,1,1,1),\qquad
 Q=\mathtt{00322301212}.
 \end{gathered}
\]

Without branching, the order in which a step reads its symbols no
longer matters, so a copy can fetch the symbols its tests need before
testing, rather than requesting them one at a time in the middle of
the tests. To do so, first extend $Q$ to length
at least $\max(1,z,w+1)$, keeping it if it is already longer. This
covers every candidate row and every upper column needed for either
test: For $r\le w$, the request length
$\max\{r+1,\lfloor(z+r)/2\rfloor+1\}$ in Section~\ref{app:output} is at
most $\max(1,z,w+1)$. Since $w\le4$ and $z\le5$, this length is at
most five, so these preliminary requests end by $m+4$. The subsequent
row search reads further symbols as before, at most through $m+6$.
Reading a few symbols earlier does not change the queen choice or the
row and diagonal records. A copy therefore carries out a step only
when $Q$ already holds every symbol that the step reads, for the tests
and for the row search; otherwise it first requests the next symbol.

A copy consists of these records and a list of the symbols it has
computed but not yet handed on. Algorithm~\ref{alg:next-symbol}
hands on its next symbol, asking its own input copy for symbols as
needed. Once it has a symbol to hand on, a copy keeps computing until
it \emph{pauses}, that is, until its next step needs a symbol beyond
the end of $Q$. A paused copy thus stands before a column whose step
it has not begun. Between calls, then, every copy rests at a paused
record; these are the records that Section~\ref{sec:block-transducer}
tabulates, and the positions of the copies in
Section~\ref{sec:generator-idea} follow this rule.

\begin{algorithm}[!ht]
\caption{Handing on the next symbol of the queen word}
\label{alg:next-symbol}
\renewcommand{\algorithmicrequire}{\textbf{Input:}}
\renewcommand{\algorithmicensure}{\textbf{Output:}}
\begin{algorithmic}[1]
\Require A copy $C$: records \eqref{eq:generator-records} and a list of computed symbols.
\Ensure The next symbol that $C$ has not yet handed on, starting with $\sigma_{30}$.
\Function{NextSymbol}{$C$}
  \While{the list of $C$ is empty}
    \If{the next step of $C$ needs a symbol beyond the end of its $Q$}\Comment{a request}
      \If{$C$ has no input copy}
        \State Give $C$ an input copy, in the state before column $30$.
      \EndIf
      \State Append \Call{NextSymbol}{input copy of $C$} to $Q$.
    \Else
      \State Carry out the step, and append its symbol $\sigma_n$ to the list.
    \EndIf
  \EndWhile
  \While{$C$ is not paused}\Comment{its next step needs no request}
    \State Carry out the step, and append its symbol $\sigma_n$ to the list.
  \EndWhile
  \State Remove the first symbol from the list and return it.
\EndFunction
\end{algorithmic}
\end{algorithm}

For example, before column $48$ the first copy has $Q=\sigma_{29}$ and
needs $\sigma_{30}$, so it creates a second copy and calls it. In the
first loop, the second copy carries out column $30$; in the second, it
continues through column $47$ and pauses, since column $48$ needs
$\sigma_{30}$. It returns $\sigma_{30}$ and keeps
$\sigma_{31}$, \dots, $\sigma_{47}$ in its list, which answer the first
copy's next requests without further computation. When the first
copy, before column $76$, asks for $\sigma_{48}$, the second copy must
carry out its own column $48$: It needs $\sigma_{30}$, creates a third
copy, and the pattern repeats one level down.

The outermost copy runs the same loop but recovers the rows
instead of handing on symbols. It keeps the column $n$, the row
reference $m$, and the upper count $U(n-1)$ as ordinary integers: A
lower choice with offset $r$ gives $q_n=m+r$, and an upper choice
gives $q_n=n+U(n-1)+1$ by Lemma~\ref{lem:upper}. The first thirty
rows are stored as a fixed seed.

Every symbol handed on is correct. Consider all steps of all copies
in the order in which they are carried out. The inputs of a step were
handed on earlier, so by induction they are actual symbols, and a
copy that receives actual symbols carries out the actual step, by the
induction of Section~\ref{sec:identification}. The simplifications
above change neither the queen chosen nor the records kept.

Every call also returns. A call to a copy with a symbol in its list
returns at once. Otherwise a copy $C$ before column $n$ calls its
input copy $C'$, whose list is empty, for $\sigma_k$ with
$k=m+|Q|\le m+6<n$. So far $C'$ has handed on exactly $\sigma_{30}$,
\dots, $\sigma_{k-1}$, and with an empty list these are all the
symbols it has computed, so it stands before column $k<n$. A copy that
must compute to answer a call thus stands before an earlier column
than its caller, and a copy before a column less than $48$ needs no
input at all, so calls cannot descend forever.
Section~\ref{sec:generator-cost} proves this for the final algorithm,
with buffers.

\subsection{Compiling the calculation into a table}\label{sec:block-transducer}

In Algorithm~\ref{alg:next-symbol}, nearly all the work of a copy is
the calculation itself: the attack tests and record updates for every
queen. Along the actual process, however, the records take only
finitely many values. This subsection computes all of them in advance and
replaces the calculation by table lookups that place several queens
at once.

\medskip\noindent
\textbf{Paused records.}
Recall that a copy is paused when its next step needs one more
symbol in $Q$. Appending one input symbol lets the calculation run
until it pauses again, possibly placing several queens along the
way. Record all these outputs: for each queen, its symbol, its row
advance $\mu$, and its lower offset $r$ when it is lower; the last two
are needed only by the outermost copy, to recover the rows. These steps
form a directed graph whose vertices are the paused records, with an
edge for each step, labeled by the symbol it reads and the outputs
it produces. A~graph of this kind, which turns a stream of input
symbols into a stream of outputs, is called a \emph{transducer}.

For example, a copy first pauses before column $48$, after producing
$\sigma_{30}$, \dots, $\sigma_{47}$, with $m=29$ and
\[
 w=1,\quad z=2,\quad R=D=A=\varnothing,\quad
 (u_{25},u_{26},u_{27},u_{28})=(0,0,1,0),\quad Q=\sigma_{29}=\mathtt{2}.
\]
The input $\sigma_{30}=\mathtt{1}$ is not yet enough, and the copy
pauses again with $Q=\mathtt{21}$ without placing a queen. The next
input, $\sigma_{31}=\mathtt{2}$, lets it place the lower queen
$(48,29)$. The input after that, $\sigma_{32}=\mathtt{2}$, places four
queens: the lower queen $(49,31)$ and the upper queens $(50,81)$,
$(51,83)$, and $(52,85)$.

\medskip\noindent
\textbf{Which records occur.}
The table must cover every input that the actual process can supply,
but not every symbol: Arbitrary inputs would lead to records that the
actual process never reaches. The history graph says which symbols
can come next, so we explore as in Section~\ref{sec:closure-check},
with paused records in place of states. Pair each paused record with the twelve symbols
ending at the right end of $Q$, which form a history-graph vertex.
Follow every outgoing edge of this vertex: Append its symbol, compute
until the next pause, and shift the twelve symbols by one. Explore
all pairs reachable from the first pause in this way. Since the queen
word follows the history graph (Section~\ref{sec:identification}),
every actual input sequence is among those explored.

The history-graph vertex only decides which inputs to explore; the
calculation itself reads only the record and the supplied symbol. The
exploration reaches $2489$ pairs but only $300$ distinct paused
records, so the transducer is finite. We therefore discard the histories, leaving a graph of $300$ paused
records whose $476$ edges keep their input symbols and complete
output lists.

\medskip\noindent
\textbf{Combining records and input steps.}
Two compressions make the table small and fast. First, records can
share a table position when their transitions agree on every input
that both accept; an input that one of them does not accept imposes
no condition. We group the $300$ records into $82$ such classes, checking
every defined transition: Whenever two records in a class accept the
same input, they must give identical output lists and successors in
the same class. A class may thus accept an input that the record a
copy actually holds does not. This is harmless: The input always
comes from the input copy and is an actual symbol, which the actual
record accepts.

Second, four symbols fit in one byte, so we combine four consecutive
input transitions into one table entry: One lookup then replaces four
steps, and the copies pass whole bytes. These paths are formed on the
$300$-record graph before replacing their endpoints by classes. This
order matters: A class can be entered through one record and left
through another, and paths introduced only in this way have not been
checked against the local calculation. Whenever paths begin in the same class and read the same
four symbols, we check that their complete output lists agree and
that they end in the same class. Hence the class and the four input
symbols, packed into one byte, determine a single entry. Each entry
emits between three and twelve symbols. Algorithm~\ref{alg:table}
summarizes the construction.

\begin{algorithm}[!ht]
\caption{Building the table}
\label{alg:table}
\renewcommand{\algorithmicensure}{\textbf{Output:}}
\begin{algorithmic}[1]
\Ensure The table $T$, or failure.
\State $S_0\gets$ the records before column $30$, run until they pause; $W_0\gets\sigma_{18}\dots\sigma_{29}$.
\State $\mathcal P\gets\{(S_0,W_0)\}$, unexamined.
\While{some pair $(S,W)\in\mathcal P$ is unexamined}
  \State Mark $(S,W)$ as examined.
  \For{each edge $W\xrightarrow{s}W'$ of the history graph}
    \State Append $s$ to $Q$ in $S$ and run until the next pause at $S'$, with outputs $L$.
    \State Record the transition $S\to S'$ reading $s$ and emitting $L$.
    \State Add $(S',W')$ to $\mathcal P$ as unexamined, unless it is already in $\mathcal P$.
  \EndFor
\EndWhile
\State Group the paused records into compatible classes.
\For{each path of four recorded transitions from $S$ to $S'$, reading $a_1a_2a_3a_4$}
  \State $e\gets(\text{joined outputs},\ \text{class of }S')$
  \If{$T[\text{class of }S,\ a_1a_2a_3a_4]$ already holds an entry other than $e$}
    \State \Return failure
  \EndIf
  \State $T[\text{class of }S,\ a_1a_2a_3a_4]\gets e$
\EndFor
\State \Return the table $T$
\end{algorithmic}
\end{algorithm}

\medskip\noindent
\textbf{Using the table.}
At run time a copy keeps only its table position, the class of its
current paused record. Given four input symbols packed into a byte,
it reads the entry at its position and that byte, emits the entry's
output symbols, and moves to the entry's successor position.

The entries are stored in a single array of $64$-bit words, in which
the rows of the different classes overlap to save space. Each class
has a base offset into this array: Adding the input byte gives the
location to read. The entry contains the successor's base offset, the
packed output symbols, their number, and a reference to their
coordinate data. The program that builds the table checks every
defined entry and ensures that no two defined entries share a
location. The
README of the generator gives the full construction counts and
storage layout.

Only the outermost copy uses the coordinate data; the other copies
need only the symbols. The data are relative to $n$, $m$, and $U(n-1)$ at the
beginning of the \emph{block}, the placements made by one lookup.
Suppose that the block starts at column $n_0$, with row
reference $m_0$ and upper count $U_0=U(n_0-1)$. Number its placements
starting at $0$, let $\mu_i$ be the row advance after placement $i$,
and let $r_j$ be the lower offset at placement $j$ when it is lower.
Summing the updates of $m$ and $U$ gives
\begin{equation}\label{eq:block-coordinates}
 q_{n_0+j}=
 \begin{cases}
  m_0+\displaystyle\sum_{i<j}\mu_i+r_j,
    &\text{if the queen is lower},\\[4pt]
  n_0+U_0+j+\displaystyle\sum_{i=0}^{j}u_{n_0+i},
    &\text{if the queen is upper}.
 \end{cases}
\end{equation}
For each queen, the table therefore stores its row as a small offset
from one of the two bases $m_0$ and $n_0+U_0$. It also stores the total advances
of $m$ and $U$ across the block. This gives exact integer coordinates
without repeating the attack tests. The outermost copy can stop partway
through a block, for instance exactly at column $N$, and resume without repeating either a lookup or a
counter update. The table, coordinate data, and thirty seed rows
occupy $34925$ bytes in total.

\subsection{Buffers, correctness, and cost}\label{sec:generator-cost}

With the table, the copies pass symbols four at a time, packed into
one byte, the input of a lookup. This subsection adds buffering
between the copies, and then proves that the outputs are correct,
that every request returns, and the bounds on time and memory.

\medskip\noindent
\textbf{Buffers.}
The remaining overhead lies mostly in the calls between copies, so
each copy also computes ahead: It keeps a buffer of bytes ready for
the copy it supplies and refills it in batches. The input copy of the
outermost copy has a buffer budget of $1024$ bytes; each successive
budget is halved, down to a minimum of one byte. Deeper copies are called
less often, so small buffers cost them little time, and the buffers
of all copies together take about twice the first budget, plus a few
bytes per copy. Each copy starts with the eighteen symbols
$\sigma_{30}$, \dots, $\sigma_{47}$: four complete bytes and two
\emph{pending} symbols, which do not yet form a byte. A buffer holds its budget plus three bytes of
padding, enough for these four initial bytes and for the output of a
lookup, at most twelve symbols, that goes past the budget.

Algorithm~\ref{alg:recursive-generation} gives the refill rule. It is
Algorithm~\ref{alg:next-symbol} with bytes in place of symbols, a
table lookup in place of four input steps, and a buffer refilled in
batches. Each copy
saves its table position, its place in the buffer, and up to three
pending symbols. The outermost copy uses the
same table entries but decodes \eqref{eq:block-coordinates} instead
of buffering the output word.

\begin{algorithm}[!ht]
\caption{Supplying four earlier symbols to a table lookup}
\label{alg:recursive-generation}
\renewcommand{\algorithmicrequire}{\textbf{Input:}}
\renewcommand{\algorithmicensure}{\textbf{Output:}}
\begin{algorithmic}[1]
\Require A copy $C$: its table position, budget, buffer of ready bytes, and pending symbols.
\Ensure The next four symbols of its output, packed into one byte.
\Function{NextByte}{$C$}
  \If{the buffer of $C$ is empty}
    \If{$C$ has no input copy}
      \State Give $C$ an input copy, starting with $\sigma_{30}$, \dots, $\sigma_{47}$, with half its budget.
    \EndIf
    \Repeat
      \State $a\gets\Call{NextByte}{\text{input copy of }C}$
      \State Look up the entry for the table position of $C$ and the byte $a$.
      \State Move $C$ to the successor position.
      \State Append the output symbols to the pending symbols of $C$.
      \State Move complete groups of four pending symbols to the buffer, as bytes.
    \Until{the buffer holds at least $b$ bytes, where $b$ is the budget of $C$}
  \EndIf
  \State \Return the next ready byte, advancing the buffer position.
\EndFunction
\end{algorithmic}
\end{algorithm}

\medskip\noindent
\textbf{Correctness of the outputs.}
The checks of Section~\ref{sec:block-transducer} establish that correct input symbols give the
same outputs and updates as the calculation. They cover every
reachable pair of a history and a paused record, every transition
after merging, and every four-input path. The argument of
Section~\ref{sec:generator-records} then applies unchanged: In the
order in which outputs are completed across all copies, each output
uses input already produced by the next copy, so by induction its
table entry gives the correct new symbols and updates, and buffering
preserves their order. Buffering lets a copy compute ahead of what
has been requested, but the induction runs over the order in which
steps are carried out, so this changes nothing.

\medskip\noindent
\textbf{Termination of requests.}
In Algorithm~\ref{alg:next-symbol}, a copy that must compute stands
before a smaller column than the copy that called it. With buffering,
a copy may have computed ahead of what it has handed on, so we count
complete bytes instead of columns, including those still waiting in
a buffer. The key fact is that each copy has computed at least three
more complete bytes than it has consumed from its input, so a request
passed down the chain always reaches a copy that has computed strictly
fewer complete bytes, and requests cannot descend forever.

At a block boundary, let $t$ be the number of input bytes consumed,
let $H$ be the number of output symbols computed starting at
$\sigma_{30}$, and put $P=\lfloor H/4\rfloor$. Thus $H$ includes the
initial eighteen symbols, and $P$ counts complete output bytes,
whether already supplied or still buffered. The input endpoint and the
current output column satisfy
\[
 m+|Q|=30+4t,\qquad n=30+H=m+\kappa+z.
\]
At the initial pause, before column $48$, we have $m=29$ and
$\kappa=U(28)=17$, and $\kappa$ never decreases along the actual
process; the table construction checks $z-|Q|\ge-4$ for all $300$
paused records. Subtracting the two identities therefore gives
\begin{equation}\label{eq:generator-lead}
 H-4t=\kappa+z-|Q|\ge13,
 \qquad P\ge t+3.
\end{equation}
These counters describe the local
calculation represented by the table, even though the inner copies do
not store them explicitly.

First consider the part of the chain where every buffer budget is one
byte. We show by strong induction on $P$ that a refill of such a copy,
which has computed $P$ complete bytes, returns. A refill requests
more input only while it has produced no complete new byte; as soon
as it has one, it returns. A fresh copy already has four ready bytes.
Otherwise each request to the input copy either returns a ready byte
at once or finds its buffer exhausted. In the second case the input
copy has computed exactly the $t$ complete bytes consumed by the
requesting copy, and by \eqref{eq:generator-lead} this count is at
most $P-3$, so the request returns by the induction hypothesis. Each
table lookup adds at least three symbols, so at most two lookups
complete a new byte and finish the refill. Only finitely many copies precede this part of the
chain with budgets greater than one, and each of their refills needs
finitely many requests to its input copy. Hence every request
returns.

\medskip\noindent
\textbf{Time and space bounds.}
To obtain the bounds on work and storage, we make precise the estimate of
Section~\ref{sec:generator-idea}: Successive copies compute prefixes
whose lengths decrease geometrically, apart from a bounded buffering
allowance.

\begin{proposition}[Sequential generation]\label{prop:stream-generation}
With a fixed initial buffer budget, the generator of
Algorithm~\ref{alg:recursive-generation} generates
$q_0$, \dots, $q_N$ exactly in $O(N)$ operations on $O(\log N)$-bit
machine words. Its working memory is $O(\log N)$ bits, including
the recursive stack, in addition to a fixed table and excluding
retained output.
\end{proposition}
\begin{proof}
Correctness and termination were established above. For the cost,
Proposition~\ref{prop:stability} gives $\kappa=m/\phi+O(1)$.
Substituting this into the identities for the input endpoint and
output column, with bounded $z$ and $|Q|$, gives $t=P/\phi+O(1)$.
Thus a copy computing $P$ output bytes consumes only $P/\phi+O(1)$
input bytes.

Let $P_j$ count complete bytes computed by the $j$th copy, let $t_j$
count the input bytes it has consumed, and let $b_{j+1}$ be the next
copy's buffer budget. Between refills, the next copy has computed at
most these $t_j$ bytes plus those in its buffer, at most $b_{j+1}+3$.
Since $t_j=P_j/\phi+O(1)$,
\begin{equation}\label{eq:generator-contraction}
 P_{j+1}\le\frac{P_j}{\phi}+O(b_{j+1}+1).
\end{equation}
The budgets are at most the fixed initial budget, so this gives
$P_{j+1}\le P_j/\phi+c$ for a constant $c$, and hence
$P_j\le P_1\phi^{1-j}+c\phi^2$. Since $P_1=O(N)$, every copy beyond the
first $\log_\phi N+O(1)$ has computed at most a constant number $c'$ of
bytes. A call makes the next copy compute only when that copy's
buffer is exhausted, and the next copy has then computed exactly the
bytes consumed by the caller, at least three fewer than the caller
has computed, by \eqref{eq:generator-lead}. Along any chain of pending
calls, the count of computed bytes therefore drops by at least three
per copy, so such a chain reaches at most $c'/3+1$ copies
beyond the first $\log_\phi N+O(1)$. A new copy is created only at the
end of a chain of pending calls, so at most $\log_\phi N+O(1)$ copies
are ever created.

Summing \eqref{eq:generator-contraction} along the chain gives $O(N)$
total byte production and hence $O(N)$ table lookups. Each lookup,
byte transfer, and coordinate decoding has bounded cost. For $L$
inner copies with initial budget $B$, the halving budgets sum to at
most $2B+L$, and the padding uses $3L$ further bytes. Each copy also
retains a record and a stack frame of bounded size: A~frame needs only
a return point in the fixed program and bounded local values, since
the copy that a call serves is the next one in the chain. The copies
are kept in chain order, in an array indexed by depth, so the link
from a copy to its input copy is implicit as well. Only the
outermost copy keeps unbounded integers, namely $n$, $m$, and
$U(n-1)$, each of $O(\log N)$ bits. Since $B$ is fixed and
$L=O(\log N)$, the working memory is $O(\log N)$ bits.
\end{proof}

\subsection{C implementation, measurements, and reference values}\label{sec:generation-benchmarks}

On the Amazon EC2 instance described below, our C generator produced and
hashed ten billion queen rows in a median $25.413$~s, compared with
$49.040$~s for bit-packed Knuth: 
a speedup of $1.93$. Its median peak
memory use was $1.76$~MiB, compared with $6243.51$~MiB.

\begin{table}[!ht]
\centering
\begin{tabular}{@{}r r r r r@{}}
\toprule
& \multicolumn{2}{c}{Elapsed time (s)}
& \multicolumn{2}{c}{Peak memory (MiB)}\\
\cmidrule(lr){2-3}\cmidrule(l){4-5}
$M$ & Our C & Packed Knuth & Our C & Packed Knuth\\
\midrule
$10^{6}$ & 0.003 & 0.005 & 1.76 & 2.20\\
$10^{7}$ & 0.026 & 0.049 & 1.67 & 7.82\\
$10^{8}$ & 0.253 & 0.486 & 1.76 & 64.07\\
$10^{9}$ & 2.524 & 4.862 & 1.67 & 625.88\\
$10^{10}$ & 25.413 & 49.040 & 1.76 & 6243.51\\
\bottomrule
\end{tabular}
\caption{Measured generation and hashing on the EC2
instance. Both programs are written in C, with the same checksum, timed in the
same way.}
\label{tab:c-generation}
\end{table}

The implementation in \path{fast_generator} of the
companion repository uses portable
C11 with fixed-width unsigned integers and loads no files at run
time.

We compared it with a bit-packed adaptation of Knuth's
\texttt{infty-queens} \cite{KnuthQueens}. The adaptation packs the
occupancy flags into $64$-bit words and widens indices to $64$ bits,
retaining Knuth's placement order and control flow; its statistics
and per-queen printing are removed. Both programs decode every row
and fold it into the same $64$-bit checksum, without retaining the
sequence. Before timing, the table was regenerated, its $3483$
four-input paths ($24811$ local steps) were replayed against a
separate implementation of the calculation, and every coordinate through the first million queens was
compared with an independent occupancy-based calculation. These finite
tests support the implementation; its correctness is
Proposition~\ref{prop:stream-generation}.

Table~\ref{tab:c-generation} reports medians of three runs after one
excluded warmup, on a fresh Amazon EC2 \texttt{r7i.2xlarge} instance
with $64$~GiB of memory, an Intel Xeon Platinum 8488C processor,
Ubuntu~24.04, and GCC~13.3.
Both programs used the same compiler options, ran single-threaded,
and were pinned to one CPU, alternating in order. Elapsed time runs
from process creation to exit; peak memory is the peak resident set
size reported by the kernel. All runs agreed on the final coordinate and checksum. The
count $M$ includes the origin, so the last column is $M-1$.

From $10^9$ to $10^{10}$ queens, both elapsed times grew by a factor
of about ten, consistent with linear running time. At $10^{10}$
queens, the generator's $42$ copies request $3560$ bytes of heap
beyond its fixed data. This storage grows
logarithmically, though at these sizes it fits in memory pages
already allocated. Knuth's packed arrays instead
need about $(2\phi+2)M/8\approx0.655M$ bytes: about $61$~GiB at
$10^{11}$ queens, close to the capacity of this machine. These results
concern bulk generation with a checksum; they do not
establish the same ratio for row-at-a-time calls, printed output, or
other hardware. The report \path{fast_generator/REPORT.md} gives the full measurement
procedure and raw observations, and Appendix~\ref{app:stream-generation}
the commands.

\Needspace{22\baselineskip}
\medskip\noindent
\textbf{\boldmath Values near $10^{11}$.}
So that readers can check their own computations,
Table~\ref{tab:values-1e11} lists $q_n$ for the twelve columns
$10^{11}-11\le n\le10^{11}$. The generator of this section produced
them by generating every row from column $0$, and a second program
in the companion repository gives the same values. In the $1$-indexed
coordinates of \cite{Knuth4B}, $s(c)=q_{c-1}+1$; for example,
$s(10^{11})=61\,803\,398\,876$.

\begin{table}[!ht]
\centering
\begin{tabular}{@{}r r r r@{}}
\toprule
$n$ & $q_n$ & $q_n-n\phi$ & $q_n-n/\phi$\\
\midrule
$99\,999\,999\,989$ & $161\,803\,398\,857$ & $-0.191$ & \\
$99\,999\,999\,990$ & $161\,803\,398\,859$ & $0.191$ & \\
$99\,999\,999\,991$ & $161\,803\,398\,861$ & $0.573$ & \\
$99\,999\,999\,992$ & $161\,803\,398\,863$ & $0.955$ & \\
$99\,999\,999\,993$ & $61\,803\,398\,870$ & & $-0.663$\\
$99\,999\,999\,994$ & $61\,803\,398\,872$ & & $0.719$\\
$99\,999\,999\,995$ & $161\,803\,398\,867$ & $0.101$ & \\
$99\,999\,999\,996$ & $161\,803\,398\,869$ & $0.483$ & \\
$99\,999\,999\,997$ & $161\,803\,398\,871$ & $0.865$ & \\
$99\,999\,999\,998$ & $61\,803\,398\,873$ & & $-0.753$\\
$99\,999\,999\,999$ & $61\,803\,398\,875$ & & $0.629$\\
$100\,000\,000\,000$ & $161\,803\,398\,875$ & $0.011$ & \\
\bottomrule
\end{tabular}
\caption{The queens in the last twelve columns through $10^{11}$, with
the distance of each from its line: $q_n-n\phi$ for upper queens and
$q_n-n/\phi$ for lower ones.}
\label{tab:values-1e11}
\end{table}

\medskip\noindent
\textbf{Local states.}
Table~\ref{tab:states} lists the actual states before the first columns
of the calculation and before columns $10^2$, \dots, $10^6$, in the
notation of Definition~\ref{def:local-state}, so that the local
calculation can also be checked step by step.
Section~\ref{sec:worked-transition} works through the steps before
columns $41$, $42$, $47$, and $53$ by hand.

\begin{table}[!ht]
\centering
\footnotesize
\begin{tabular}{@{}r r r c c c l l l@{}}
\toprule
$n$ & $w$ & $z$ & $R$ & $D$ & $A$ & $H_{\mathrm{in}}$ & $Q$ & $H_{\mathrm{out}}$\\
\midrule
$30$ & $0$ & $-1$ & $\varnothing$ & $\{1\}$ & $\varnothing$ & \texttt{230121212223} & \texttt{00322301212} & \texttt{300322301212}\\
$31$ & $-2$ & $-1$ & $\varnothing$ & $\varnothing$ & $\varnothing$ & \texttt{301212122230} & \texttt{0322301212} & \texttt{003223012121}\\
$32$ & $-1$ & $0$ & $\varnothing$ & $\varnothing$ & $\varnothing$ & \texttt{301212122230} & \texttt{0322301212} & \texttt{032230121212}\\
$33$ & $0$ & $1$ & $\varnothing$ & $\varnothing$ & $\varnothing$ & \texttt{301212122230} & \texttt{0322301212} & \texttt{322301212122}\\
$34$ & $-2$ & $-1$ & $\varnothing$ & $\varnothing$ & $\varnothing$ & \texttt{121212223003} & \texttt{22301212} & \texttt{223012121220}\\
$35$ & $-1$ & $0$ & $\varnothing$ & $\varnothing$ & $\varnothing$ & \texttt{121212223003} & \texttt{22301212} & \texttt{230121212203}\\
$36$ & $0$ & $1$ & $\varnothing$ & $\varnothing$ & $\varnothing$ & \texttt{121212223003} & \texttt{22301212} & \texttt{301212122032}\\
$37$ & $1$ & $2$ & $\varnothing$ & $\varnothing$ & $\varnothing$ & \texttt{121212223003} & \texttt{22301212} & \texttt{012121220323}\\
$38$ & $1$ & $1$ & $\varnothing$ & $\{1\}$ & $\varnothing$ & \texttt{212122230032} & \texttt{2301212} & \texttt{121212203230}\\
$39$ & $2$ & $2$ & $\varnothing$ & $\{1\}$ & $\varnothing$ & \texttt{212122230032} & \texttt{2301212} & \texttt{212122032303}\\
\addlinespace
$100$ & $2$ & $2$ & $\varnothing$ & $\{1\}$ & $\varnothing$ & \texttt{032302121032} & \texttt{23} & \texttt{300322032303}\\
$1\,000$ & $-2$ & $-1$ & $\varnothing$ & $\varnothing$ & $\varnothing$ & \texttt{300322301203} & \texttt{0} & \texttt{230303002320}\\
$10\,000$ & $0$ & $0$ & $\varnothing$ & $\{1\}$ & $\varnothing$ & \texttt{323003210321} & \texttt{2} & \texttt{203230023230}\\
$100\,000$ & $0$ & $1$ & $\varnothing$ & $\varnothing$ & $\varnothing$ & \texttt{003230023230} & \texttt{0} & \texttt{220303230122}\\
$1\,000\,000$ & $2$ & $2$ & $\varnothing$ & $\{1\}$ & $\varnothing$ & \texttt{300321032122} & \texttt{23} & \texttt{023230032303}\\
\bottomrule
\end{tabular}
\caption{The actual local states before columns $30$ to $39$ and
before columns $10^2$, \dots, $10^6$.}
\label{tab:states}
\end{table}
\clearpage
\appendix
\renewcommand{\thesubsection}{\thesection\arabic{subsection}}
\section{Reproducing the computation}\label{app:artifacts}
The programs described here are in the companion repository
\url{https://github.com/boonsuan/queens}. Its folder
\texttt{verification} contains the finite verification of
Sections~\ref{sec:local}--\ref{sec:finite}, the folder \texttt{oeis}
the checks for Section~\ref{sec:oeis-consequences}, and the folder
\path{fast_generator} the generator of Section~\ref{sec:generation}.
Each folder has a README with further details. Only the checks of
Appendix~\ref{app:verification} are needed for
Theorems~\ref{thm:main} and~\ref{thm:knuth}; the other programs
reconstruct the history graph they check, test the implementations, or
carry out further computations. The Python programs
need Python~3.10 or later and no other packages, and all
mathematical checks use integer arithmetic.

\subsection{The finite verification}\label{app:verification}
From the folder \texttt{verification}, run
\begin{verbatim}
python verify_tuples.py
python verify_bitmasks.py
python compare_verifiers.py
python check_correspondence.py
python test_rejection.py
\end{verbatim}
Each takes a few seconds.

The first command carries out the check of
Section~\ref{sec:closure-check}, including the starting checks of
Section~\ref{app:seed}. It uses the module \texttt{calculation.py},
which stores the eight fields of \eqref{eq:state} as a tuple and
implements Algorithm~\ref{alg:local-successors} with generators that
yield one result per branch. The second command
performs the same check independently: It shares no code with the
first, packs the records into integers, splits each symbol into its
two bits, and handles branching with explicit work lists. The third
converts both to a common encoding and checks that they reach the same
states with the same complete sets of successors. This agreement
checks the implementations against each other; the correspondence
with the actual infinite sequence is the induction in
Section~\ref{sec:identification}.

The fourth command computes $3000$ queens directly from
\eqref{eq:greedy}. For each column $30\le n<3000$, it runs the actual
branch of Lemma~\ref{lem:local-exactness}, and checks
that every requested symbol labels a history-graph edge and has index
at most $m+6<n$, and that the branch yields the actual next queen and
the actual records before column $n+1$. The last command removes the edge carrying $\sigma_{30}$
from a copy of the history graph and confirms that both verifiers
reject it, so the check is not vacuous.

The command \texttt{python trace.py 41 44} prints the actual records
before each column in a range, the requests made, the queen chosen,
and the new symbol, checking each step against the board. It
reproduces the examples of Section~\ref{sec:worked-transition}. The
command \texttt{python states.py 30-39 100 1000000} prints the states
of Table~\ref{tab:states} in a few seconds, and it accepts any other
columns $n\ge30$. It runs the calculation only over a short stretch of
columns before $n$, since the calculation reads only near the least
unused row, and checks each state against the board and against a run
that starts twice as far back.

The two verifiers report
\begin{center}
\begin{tabular}{lr}
\toprule
Quantity&Result\\
\midrule
History-graph vertices and edges&$2092$, $2603$\\
State-graph vertices and edges&$7014$, $8327$\\
Completed local choices&$7612$\\
Completed lower choices&$3153$\\
Largest offset of a requested symbol&$5$\\
Largest row advance $\mu$&$5$\\
Range of $w-r-|D|$ at lower choices&$[-4\dts2]$\\
Requests with no outgoing edge&$0$\\
\bottomrule
\end{tabular}
\end{center}
The state-graph counts refer to distinct states and ordered pairs of
states. The row search can branch again after a choice has been made,
so completed choices and state edges have different counts. The
observed maxima for requested symbols and row advances are smaller
than the bounds established in Lemma~\ref{lem:local-exactness}. Only the
diagonal discrepancy bound $4$ is used in the proof of
Theorem~\ref{thm:main}.

The command \texttt{python knuth\_ranges.py} repeats the exploration of
the first verifier and carries out the check of
Section~\ref{sec:sharper} in a few seconds. It finds the extremal
states by iterating, in floating point, the rounds
\eqref{eq:rounds}, solves the
resulting linear equations exactly in $\mathbb Q(\sqrt5)$, and then
checks that every vertex of the history graph is an output history and,
in exact arithmetic, that the bounds are consistent, that they
hold directly for $30\le n\le48$, and that the least unused row is at
least $30$ before column $49$. It then evaluates the bounds of
Proposition~\ref{prop:window-check} on every window whose last symbol
has column bit $1$ and at every lower choice of every state, combines
them with the columns before $49$, and checks the constants of the
theorem and Knuth's ranges. It also checks that every lower
choice has $-4\le w-r-|D|\le2$, the range stated in
Section~\ref{sec:sharper}. The command
\texttt{python knuth\_ranges\_bitmasks.py} repeats the check
independently. It shares no code with the first, takes the state graph
from the second verifier, computes in $\mathbb Q(\phi)$ rather than
$\mathbb Q(\sqrt5)$, and finds the bounds by exact policy iteration
instead of a floating-point search. It finds the same constants.

The history graph is stored in \texttt{history.json} as
\texttt{\{"memory": 12, "vertices": [[h, mask], \ldots]\}}. Each
entry is one vertex: \texttt{h} is its twelve-symbol word in base
four, with the newest symbol in the least significant digit, and bit
$s$ of \texttt{mask} is set when an edge labeled $s$ leaves it.
Since the file stores only the labels, the verifiers first check
that the destination $\last(Ws)$ of every such edge is again a
listed vertex, so that the file describes a graph as in
Definition~\ref{def:history-graph}. The
completed graph is a \emph{certificate}: finite proof data whose
required properties the verifiers check.

\subsection{Constructing the history graph}\label{app:construction}
From the folder \texttt{verification}, run
\begin{verbatim}
python construct_history_graph.py
python history_length_experiment.py
\end{verbatim}
The first command carries out the construction of
Section~\ref{sec:finite-statement}. It computes the first thirty
queens directly from \eqref{eq:greedy} and runs the calculation, adding
each missing output edge instead of failing, until a complete
exploration adds no edge. It builds the graph without using
\texttt{history.json}, reports each exploration, and finally confirms
that the result is identical to \texttt{history.json}.

To see that the length twelve is not arbitrary, the second command
repeats the construction with histories of each length $L$ from $4$
through $14$, keeping the same start, calculation,
and Condition~\ref{cond:state}. Shorter lengths are not tested:
Because $z\ge-4$, the upper-queen tests can involve the upper-column
bits of columns $m-4$, \dots, $m-1$ (compare Section~\ref{sec:generator-records}),
so the input history must contain them. Lengths $4$ through $11$ fail,
because the construction reaches a successor violating
Condition~\ref{cond:state}; for lengths $5$ through $11$, the violation
found has $w=5$. Lengths $12$, $13$, and $14$ succeed, and each
resulting graph passes the check of Section~\ref{sec:closure-check}.
Thus twelve is the smallest successful length in this setting.

\subsection{Additional checks for the OEIS consequences}\label{app:oeis-checks}

The computations for Section~\ref{sec:oeis-consequences} reuse the
programs of the folder \texttt{verification} with longer histories. From the root of
the repository, run
\begin{verbatim}
python oeis/run_all.py
\end{verbatim}
It takes about a minute and a half, prints one line for each statement checked,
and writes the graphs, languages, and witnesses to
\path{oeis/results}, where the recorded outputs are also stored.

\medskip\noindent
\textbf{The refined certificate.}
Replace both twelve-symbol histories by histories of length $40$,
keeping the records, Condition~\ref{cond:state}, and the calculation
of Section~\ref{app:output}. Start immediately before column $80$,
where direct computation from \eqref{eq:greedy} gives $m=51$, $d=32$, $U(m-1)=31$, $w=-3$,
$z=-2$, and $R=D=A=\varnothing$; the input history, queue, and output
history cover indices $11$--$50$, $51$--$79$, and $40$--$79$.
Lemma~\ref{lem:local-exactness} remains valid: Omitted upper columns
lie further in the past, requests still end by $m+6$, and
$n-m=U(m-1)+z\ge12-4=8$, so its hypothesis remains $U(m-1)\ge12$,
not $40$. The induction of Section~\ref{sec:identification}
therefore applies once the refined graph passes the check of
Section~\ref{sec:closure-check}. The construction of
Appendix~\ref{app:construction}, with length $40$ and start
column $80$, gives a graph with $16876$ vertices and $17499$ edges.
With this graph fixed, the calculation reaches $29267$ states and
$30000$ directed state edges, and every check passes.

\medskip\noindent
\textbf{How much of the twelve-symbol graphs is actual.}
The refined certificate also shows exactly which vertices and edges of
the history graph, and which states and edges of the state graph, the
actual queens use. From column $80$ on, keeping the last twelve symbols
of each history maps the actual forty-symbol windows, edges, and states
onto the actual twelve-symbol ones. For states there is one more step:
The forty-symbol state before column $80$ holds the whole queue from
$m$ to $\sigma_{79}$, while the twelve-symbol state, followed from
column $30$, may hold fewer symbols. The two actual branches agree after
this map from column $128$ on, so every later actual state is the image
of a refined state reachable from the one there. Together with the
earlier ones, this allows at most $1824$ windows, $2195$ history edges,
$4749$ states, and $5207$ state edges, and all of them occur before
column $1.6\times10^6$. The program \path{oeis/actual_counts.py}
computes both bounds.

\medskip\noindent
\textbf{Graphs for the gap and run sequences.}
The output symbol of a state-graph edge is the last symbol of the
successor's $H_{\mathrm{out}}$. Starting from a state just after an
upper output, follow lower outputs to the next upper output, and
replace this path by a single edge labeled by its length. The result
is a finite graph whose walks include the gap sequence $g$ after
column $80$. Following edges labeled $1$ and then one edge labeled
$k>1$, and labeling the result $k-1$, gives the graph for $r$ from
which the proof of Corollary~\ref{cor:oeis-runs} obtains the sets
$\mathcal L_c$.
Similarly, the labels along paths from one $3$ to the next give the
return words. The part of the gap graph avoiding $3$ is acyclic, so
this language is finite and is computed exactly. For each return word,
the union of the languages at its possible endpoints contains every
word that can follow it; it has a single element for exactly $63$
words. The shortest path between two edges labeled $4$ has length
$71$, and the paths of that length all have the same labels.

\medskip\noindent
\textbf{Occurrence witnesses.}
The graphs give only upper bounds: Attainment and nonfaithfulness
come from the actual sequence. The first $20000$ queens, computed
by a bitboard implementation of the defining rule, contain a witness
for every value in Corollary~\ref{cor:oeis-runs}. They also contain
the eleven runs of equal terms of $r$ that begin before the run graph
applies, whose lengths lie in the sets $\mathcal L_c$; deciding that
these runs are maximal needs the columns through $95$. A
million-queen prefix, generated by a separate program that agrees
with the bitboard implementation on the first $200000$ queens, contains every return word by gap index $108015$, and two
different successors of each word that is not faithful by gap index
$573517$. The file \path{oeis/AUDIT.md} separates these proved
consequences from observations that the checked graphs do not settle.

\subsection{Running the fast generator}\label{app:stream-generation}

The folder \path{fast_generator} contains the C11 generator of
Section~\ref{sec:generation}, the program that builds its table, the
checks, the benchmark tools, and the recorded measurements. From that
folder, with a C11 compiler, \texttt{make}, and Python~3.10 or later,
run
\begin{verbatim}
make
./build/queens_fast --count 1000000
make verify
make test
\end{verbatim}
The second command generates the first million rows and prints a
summary with the last row and a checksum of all rows; add
\texttt{-{}-emit} to print every column and row, and see
\path{src/queens_fast.h} for the interface. The command
\texttt{make verify} rebuilds every generated file, the table and the
test cases, from the history graph and checks that each is identical
to the committed copy, then replays every local step and
every four-input path through a separate implementation of the local
calculation. The command \texttt{make test} compares every coordinate
through the first million queens with a calculation using full
occupancy arrays, whose first $3000$ rows are checked against the
greedy rule, and exercises the interface.

The runs behind Table~\ref{tab:c-generation} are recorded in
\path{results/ec2-native.json}, with the machine and build details in
\path{results/ec2-environment.json}. The command
\begin{verbatim}
python3 tests/check_ec2_results.py \
  --environment results/ec2-environment.json
\end{verbatim}
rechecks the recorded runs and recomputes the table, and
\path{bench/run_benchmark.py} runs a new campaign on a Linux machine.
The comparator \path{knuth/knuth_packed.c} is derived from Knuth's
\texttt{infty-queens} by \path{knuth/derive_knuth_packed.py}.

The extreme values quoted at the end of Section~\ref{sec:sharper}
come from a scan of the first $10^{11}$ columns by
\path{src/scan_bounds.c}, which also checks Knuth's ranges. With GCC or Clang, run
\begin{verbatim}
make build/scan_bounds build/knuth_packed
python3 tests/check_bounds.py
python3 bench/run_bounds_scan.py --count 100000000000
\end{verbatim}
The scanner checks both intervals and their union, and records the
extreme deviations and the queens attaining them. Comparisons too close
to decide by a rigorously bounded floating-point approximation are
resolved with exact integer arithmetic. The recorded result is
\path{results/knuth-bounds-1e11.json}, and \path{REPORT.md} describes
both measurements in full.
%
\section*{Declaration of AI usage}
The proof was found with GPT-6 Pro, which also produced an initial
draft of this paper. The paper was later revised with Claude Opus~5.5
under the author's direction.

\section*{Acknowledgements}
I thank Don Knuth for many helpful comments on the first
version of this paper, as well as for writing 
\textsl{The Art of Computer Programming}, where I first
learned of this problem.

\section*{Code availability}
The code for the finite verification, the OEIS checks, and the
generator of Section~\ref{sec:generation} is available at
\url{https://github.com/boonsuan/queens}, together with a Lean
formalization.
An interactive webpage accompanying this paper is available from
\url{https://greedyqueens.com}.
%

%
\end{document}